\documentclass[letterpaper,12pt]{amsart}
\usepackage[margin=2cm]{geometry}

\usepackage{amsfonts}
\usepackage{amscd, amssymb}
\usepackage{amsmath,amscd}
\usepackage[driverfallback=hypertex]{hyperref}
\usepackage{nameref,zref-xr}                    %to include reference to other papers
\zxrsetup{toltxlabel}
\zexternaldocument*[calmodulin-]{real_Moduli_20_11}
\usepackage{comment}
\usepackage{graphicx}
\usepackage[usenames,dvipsnames]{color}
\usepackage[all]{xy}   %For commutative diagrams
\input xy
\xyoption{all}
\usepackage{epstopdf}
\usepackage[T1]{fontenc}
\usepackage[latin9]{inputenc}
\usepackage{MnSymbol}
\usepackage{enumerate}
\usepackage{tikz}

\newtheorem{thm}{Theorem}[section]
\newtheorem*{thm*}{Theorem}%[section]

\newtheorem{prop}[thm]{Proposition}

\newtheorem{lem}[thm]{Lemma}

\newtheorem{lemma}[thm]{Lemma}
\newtheorem{cor}[thm]{Corollary}

\newtheorem{conj}{Conjecture}

\theoremstyle{definition}
\newtheorem{definition}[thm]{Definition}
\newtheorem{nn}[thm]{Notation}

\theoremstyle{remark}
\newtheorem{rmk}[thm]{Remark}
\newtheorem{rem}[thm]{Remark}
\newtheorem{ex}[thm]{Example}
\newtheorem{obs}[thm]{Observation}
\newtheorem{defn}[thm]{Definition}

\newcommand{\R}{\mathbb R}
\newcommand{\C}{\mathbb C}
\newcommand{\RP}{{\mathbb {RP}}}
\newcommand{\CP}{{\mathbb {CP}}}

\newcommand{\EEE}{R}
\newcommand{\TTT}{\mathcal{T}}
\newcommand{\PPP}{\mathcal{P}}
\newcommand{\AAA}{\mathcal{A}}
\newcommand{\AAH}{\hat{\mathcal{A}}}
\newcommand{\N}{\mathbb N}

\newcommand{\Z}{\mathbb Z}

\newcommand{\CM}{{\mathcal{M}}}

\newcommand{\oCM}{{\overline{\mathcal{M}}}}

\newcommand{\CL}{{\mathbb{L}}}

\newcommand{\Univ}{\Upsilon}
\newcommand{\DIAMOND}{\Diamond}

\DeclareMathOperator{\id}{Id}

\DeclareMathOperator{\deg1}{deg}

\newcommand{\val}{\operatorname{val}}

\newcommand{\ev}{\mathrm{ev}}

\newcommand{\<}{\left<}
\renewcommand{\>}{\right>}
\newcommand{\pt}{\mathrm{pt}}
\newcommand{\Aut}{\mathrm{Aut}}

\newcommand{\OGW}{\mathrm{OGW}}
\newcommand{\Orb}{\mathrm{Orb}}
\renewcommand{\for}{\mathrm{for}}
\newcommand{\Vol}{\mathrm{Vol}}
\newcommand{\Contract}{\mathrm{Contract}}
\newcommand{\Graphs}{\mathrm{Graphs}}
\newcommand{\Enum}{\mathrm{Enum}}
\newcommand{\Maps}{\mathrm{Maps}}

\usepackage{tensor}
\newcommand{\fibp}[2]{\tensor[_{#1\thinspace}]{\times}{_{\thinspace #2}}}
\newcommand{\aso}{\mathcal{A}_{S^1}}
\newcommand{\Sym}{\operatorname{Sym}}
\newcommand{\For}{\operatorname{For}}
\newcommand{\Or}{\operatorname{Or}}
\newcommand{\bb}{\mathcal{B}}
\newcommand{\ee}{\mathcal{E}}
\newcommand{\sstar}{\smallstar}
\newcommand{\lf}{\mathsf{l}}
\newcommand{\vd}{\vec{d}}
\newcommand{\ts}{\mathcal{T}}
\newcommand{\hts}{\hat{\ts}}
\newcommand{\rr}{\mathbb{R}}
\newcommand{\FFF}{{F}}

\newcommand{\pp}{\mathbb{P}}

\newcommand{\ff}{\mathcal{F}}

\newcommand{\cc}{\mathbb{C}}

\newcommand{\zz}{\mathbb{Z}}
\newcommand{\qq}{\mathbb{Q}}
\newcommand{\ii}{\mathcal{I}}

\newcommand{\xx}{\mathcal{X}}
\newcommand{\yy}{\mathcal{Y}}
\newcommand{\mm}{\mathcal{M}}

\newcommand{\im}{\mbox{Im }}

\newcommand{\Hom}{\mathrm{Hom}}

\newcommand{\cl}{\mathcal{C}}
\newcommand{\aru}[2]{\ar[#1]^-{#2}}

\newcommand{\kf}{\mathsf{k}}

\newcommand{\Prr}{{\text{pr}}}
\newcommand{\inv}{\operatorname{inv}}

\begin{document}
\title{Open $\CP^1$ descendent theory I: The stationary sector}
\author{Alexandr Buryak\,$^{1,2,3}$ Amitai Netser Zernik\,$^{4}$ Rahul Pandharipande\,$^5$ and Ran J. Tessler\,$^6$}

\begin{abstract}
We define stationary descendent integrals on the moduli space
  of stable maps from disks to $(\CP^1,\rr\pp^1)$.
  We prove a localization formula for
  the stationary theory involving contributions from the fixed points
  and from all the corner-strata. We use the localization formula
  to prove a recursion relation and a closed formula for all genus $0$ disk cover
  invariants in the stationary case.
  For all higher genus invariants, we propose a conjectural formula.\end{abstract}

\maketitle

\pagestyle{plain}
\centerline{\tiny{$^1$ Faculty of Mathematics, National Research University Higher School of Economics, 6 Usacheva str., Moscow, 119048, Russian Federation}}
\centerline{\tiny{$^2$ Center for Advanced Studies, Skolkovo Institute of Science and Technology, 1 Nobel str., Moscow, 143026, Russian Federation}}
\centerline{\tiny{$^3$ Novosibirsk State University, 1 Pirogova str., Novosibirsk, 630090, Russian Federation}}
\centerline{\tiny{$^4$ Via Transportation, Inc.}}
\centerline{\tiny{$^5$ Department of Mathematics, ETH Zurich, Raemistrasse 101, Zurich, Switzerland}}
\centerline{\tiny{$^6$Department of Mathematics, Weizmann Institute of Science, POB 26, Rehovot 7610001, Israel}}

\tableofcontents

\section{Introduction}\label{sec:ext_intro}
\subsection{Background}\label{sec:back}
Closed Gromov-Witten theory concerns integrals over the moduli spaces
of stable maps to nonsingular projective target varieties $X$.
The Gromov-Witten theory of a point, the simplest target variety,
consists of all descendent integrals
over the moduli spaces of curves,
$$\int_{\overline{\mm}_{g,n}} \psi_1^{a_1} \cdots \psi_n^{a_n}\, .$$
The associated
generating series satisfies the KdV equation as conjectured by Witten~\cite{Wit91} and
proven by Kontsevich~\cite{Kon92} (see \cite{OP1,Mir,KL07} for other proofs).
In case the dimension of the target $X$ is 1, the closed Gromov-Witten
theory has been solved completely in a series of papers
\cite{OP06,OPequiv,OPvir} more than
10 years ago. Fundamental to the solution is a special study
of the target $\CP^1$ via virtual localization
\cite{pandharipande-virtual-loc},
Hodge integrals \cite{FabPan}, and the ELSV formula \cite{ELSV}.
The complete solution for the closed Gromov-Witten theory of
all 1-dimensional targets
has two steps:
\begin{enumerate}[(i)]
  \item Virasoro constraints are used to reduce all descendent integrals to the {\em stationary sector} consisting of descendents of points,
  \item the stationary sector is solved via the Gromov-Witten/Hurwitz correspondence.
\end{enumerate}

More recently, there have been parallel developments in {\em open} Gromov-Witten
theory. The open case concerns integrals over the moduli spaces
of maps of Riemann surfaces with boundary $(\Sigma,\partial \Sigma)$ to a symplectic
target $(X,L)$, where the boundary $\partial \Sigma$ is required to map
to the Lagrangian
$$
L\subset X.
$$
The open Gromov-Witten theory of a point has been defined and solved in a series of papers \cite{BT17,PST14,ST1, Comb} resulting in open analogues of the KdV equations and Virasoro constraints.

Our goal here is to begin a systematic study of the open Gromov-Witten
theory of 1-dimensional targets in the basic case
$$ (X,L) = (\CP^1,\rr\pp^1)\, .$$
The open theory is technically much more difficult to define
than the closed theory because of nontrivial boundary, corner, and orientation
issues for the moduli spaces of maps. Our main results are the definition
and calculation of
the stationary descendent theory of disk maps to $(\CP^1,\rr\pp^1)$.

One of the main driving forces for our work was the localization formula of
the second author \cite{fp-loc-OGW}. We develop here
an analogous localization formula from first principles. Even though in our geometry
the fixed-point components never intersect the boundary, our formula
involves contributions from the corners of the moduli spaces (which contribute also to the formula of \cite{fp-loc-OGW}).

The construction of the
stationary descendent theory of $(\CP^1,\rr\pp^1)$ for open Riemann
surfaces of higher genus is not yet complete. However, using our solution of
the disk theory,
we conjecture a full solution for the open stationary descendent
invariants $(\CP^1,\rr\pp^1)$ in all genera.

\subsection{General notations}
Throughout the paper we consider closed Riemann surfaces as a special case of Riemann surfaces with boundary.
For a smooth or stable Riemann surface $\Sigma$ with non-empty boundary, let the \emph{double} $\Sigma_\C$ be the Riemann surface obtained by gluing the surface and the conjugate surface (same surface but with opposite complex structure) along the common boundary, using the Schwartz reflection principle. Whenever $\Sigma$ has non-empty boundary, its \emph{doubled genus}, which we will simply call the \emph{genus}, is the genus of $\Sigma_\C,$ which we denote by
$$
g=g(\Sigma).
$$
The Riemann surface $\Sigma_\C$ is endowed a canonical involution $b$ called \emph{conjugation}.
We denote by $h(\Sigma)$ the number of boundary components. The \emph{small genus} $g_s(\Sigma)$ is the genus of the closed surface obtained by capping each boundary component of $\Sigma$ with a disk. When $\Sigma$ is connected with $\partial\Sigma\neq\emptyset$, the numbers $g,~h$ and $g_s$ are related by
\begin{equation}\label{eq:small_g}
g_s=\frac{g-h+1}{2}\, ,
\end{equation}
and hence
\begin{equation}\label{eq:g,h}
1\leq h\leq g+1\,,\qquad h=g+1 \mod~2\, .
\end{equation}
When $\partial\Sigma=\emptyset$, we write $g_s=g,$ the usual genus and then $h=0.$
Two elements from the set $\{g,h,g_s\}$ determine, in the connected case, the third. When working with fixed point graphs, to be defined below, $g_s,h$ are more convenient parameters than $g,h$, but when considering dimensions of moduli, working with the parameter $g$ simplifies the equations. We shall therefore work interchangeably with different elements from $\{g,h,g_s\}$, according to the context.

%Let $\left(X,L\right)=\left(\cc\pp^{1},\rr\pp^{1}\right)$.
We view  $S^1$ as given by the unit circle in the complex plane,
$e^{\sqrt{-1}\theta}\in S^{1}$.
The circle $S^1$ acts on $\left(\cc\pp^{1},\rr\pp^{1}\right)$ by the rule
\[
e^{\sqrt{-1}\theta}\cdot \left[z_{0}:z_{1}\right]=\left[\cos\theta\cdot z_{0}-\sin\theta\cdot z_{1}\, :\, \sin\theta\cdot z_{0}+\cos\theta\cdot z_{1}\right],
\]
which preserves the Lagrangian $\rr\pp^1\subset \cc\pp^1$. The $S^1$-action has two fixed points
$$\left\{ p_{+},p_{-}\right\} \subset\cc\pp^{1}$$
labeled in such a
way that $e^{\sqrt{-1}\theta}\in S^{1}$ acts on $T_{p_{\pm}}\cc\pp^{1}$
by $z\mapsto e^{\pm2\sqrt{-1}\theta}z$.
The form $d\theta$ will denote the canonical normalized angular form on $\RP^1,$ which satisfies $\int_{\RP^1}d\theta=1,$ where $\RP^1$ is oriented as the boundary of the upper hemisphere.

We identify the relative homology group
\[
H_{2}\left(\cc\pp^{1},\rr\pp^{1}\right)=\zz_{+}\oplus\zz_{-}
\]
so that the two oriented hemispheres containing $p_{+}$ and $p_{-}$
represent the generators $\left(1,0\right)$ and $\left(0,1\right)$,
respectively.
The \emph{degree} of a map $u:(\Sigma,\partial\Sigma)\to(\CP^1,\RP^1)$ is the element
\[\vd=(d_{+},d_{-})=u_*[\Sigma]\in H_{2}\left(\cc\pp^{1},\rr\pp^{1}\right).\]
We write $\partial^{H}$ for the connecting map $H_{2}\left(\cc\pp^{1},\rr\pp^{1}\right)\to H_{1}\left(\rr\pp^{1}\right)$,
\begin{equation}\label{eq:connecting}
\partial^{H}\left(a,b\right)=a-b.
\end{equation}
We also set
\[
\left|\left(a,b\right)\right|=\max\left\{ a,b\right\}.
\]

A map $u:(\Sigma,\partial\Sigma)\to(\CP^1,\RP^1)$ induces a map $u_\C:\Sigma_\C\to\CP^1,$ which satisfies
$$
u_\C(b(z))=\overline{u(z)}\,,
$$
where, written in the standard projective coordinates, $\overline{[w:v]}=[\bar w: \bar v],$ and $z\to\bar{z}$ is the standard conjugation.

For a positive integer $n$, let
$[n]=\{1,\ldots,n\}$.
% Labels for marked points will come from a fixed set $\Univ$ which contains $\N$% and some other symbols, such as $\star,\sstar.$
We will use natural numbers to label usual marked points and
symbols such as $\{\star,\sstar\}$
to label marked points which are created from normalizing nodes in
nodal surfaces.
We  will consider stable disk maps%maps from stable marked disks
$$u:(\Sigma,\partial\Sigma,z_1,\ldots,z_n)\rightarrow (\CP^1,\RP^1)\, .$$
We will only consider the stationary sector, meaning that marked points are constrained to fixed points in the target.

In a sequel \cite{BPTZ2}, we will relate the stationary theory developed
here to the
enumeration of open Hurwitz covers.
We will only consider internal markings. Boundary markings play no interesting role in the stationary theory, since the divisor relation removes
boundary markings. In a future work, where we will consider the full equivariant theory, which includes also non-stationary insertions, boundary markings will play a non-trivial role.

\begin{nn}\label{nn:a_eps}
For a set $\lf$ of labels, we use the notation
$\vec{a}=(a_i)_{i\in \lf}\in \zz_{\geq 0}^\lf$ to denote a \emph{vector of descendents}, which indicates the number of descendents the $i^{th}$ marking carries, and
$$\vec\epsilon=(\epsilon_i)_{i\in\lf}\in\{\pm\}^\lf$$ to denote a \emph{vector of point constraints}, which indicates to which component of $\CP^1\backslash\RP^1$ the point constraining the $i^{th}$ marking belongs.
\end{nn}
%Thus, the $i^{th}$ interior marked point will always carry $\psi_{i}^{a_i}\cdot\ev_{i}^{*}\rho_{\epsilon_i}$ where $\rho_{\pm}$ is a form which represents the equivariant Poincar\'e dual to the north / south fixed-point, and whose support is contained in the corresponding hemisphere.

\subsection{Stationary open Gromov-Witten theory of $(\CP^1,\RP^1)$ in genus $0$}\label{subsec:intro-OGW_0}
The goal of the paper is to construct and study
the stationary open GW theory of maps to $(\CP^1,\RP^1)$ with descendents.
In Section \ref{sec:GW-defs}, we will present
the main geometric construction of this paper:
a geometric definition of the equivariant stationary descendent invariants
for the moduli of maps from disks to $(\CP^1,\RP^1)$,
\begin{equation}\label{integrand1st}
\<\prod_{i\in\lf}\tau^{\epsilon_i}_{a_i}\>_{0,\vd}=
\int_{\overline{\mm}_{0,\emptyset,\lf}\left(\vd\right)}\prod_{i\in\lf}\psi_{i}^{a_i}\ev_{i}^{*}\rho_{\epsilon_i}\, .
\end{equation}
Here, $\overline{\mm}_{0,\emptyset,\lf}\left(\vd\right)$ is the moduli space
of stable disk maps with internal markings labeled by $\lf$ and no boundary markings (indicated by $\emptyset$).
The moduli space
$\overline{\mm}_{0,\emptyset,\lf}\left(\vd\right)$
is an orbifold with corners which is oriented with the orientation defined
in  Section \ref{sec:or}, the integrand
$\psi_i$ is the equivariant Chern form associated with a connection on the $i^{th}$ cotangent line bundle $\CL_i$, and $\rho_{\pm}$ is the
equivariant Poincar\'e dual of the corresponding point (see Section \ref{subsec:coherent_from_psi}).

Because $\partial\overline{\mm}_{0,\emptyset,\lf}\neq\emptyset,$
the integral
\eqref{integrand1st} depends on a choice of boundary conditions. There is an infinite dimensional space of such choices which gives rise to different integrals. It is natural to ask whether there exist distinguished boundary conditions which are better than others in some sense.
Criteria for "good" boundary conditions might be that they can be
described in ways which interact with the stratification of the moduli, so
that the resulting intersection numbers are related to the closed Gromov-Witten
theory of $\CP^1,$ or that a connection with an interesting integrable hierarchy exists.

The construction we provide satisfies most of these criteria. The
main exception is that we do not know, at the moment, the precise relation to
the Toda hierarchy which governs the closed theory of $\CP^1$.
We define the integrals
\eqref{integrand1st}
in terms of equivariant forms. These forms are defined in a recursive way
with respect to the stratification of the moduli. Although the space of
these recursive integrands will also be infinite dimensional,
we will show (in the first part of Theorem~\ref{thm:int_nums_equal_tree_sum})
that different choices yield the same integrals.
We will compute these integrals and find closed formulas in special cases.
In the sequel \cite{BPTZ2},  we will show (in genus 0) a correspondence with
a Hurwitz theory with completed cycles which is the open analog
of the Gromov-Witten/Hurwitz correspondence of \cite{OP06}.

\subsection{Boundary conditions on sections}
We postpone the full definition via recursive equivariant forms
to Section \ref{sec:GW-defs}. Instead, we briefly describe here
an alternative definition in the language of sections of bundles.
The approach using sections is simpler to describe
and is useful for building intuition, but describes only the non-equivariant
limit of our theory. The equivalence of the two definitions can be proven along the lines of \cite[Corollary 15]{fp-loc-OGW}.
%at any rate the discussion of multisections is given only for motivation, and as an intuitive easily stated definition of the invariants, although later in this paper
To prove our main results, we will work exclusively with the full
definition in terms of equivariant forms.

We would like to describe the integral \eqref{integrand1st} as a weighted count of zeroes of a multisection \[s=\bigoplus_{i\in\lf\, , \, j\in[a_i]} s_{ij}\] of
$\bigoplus\CL_i^{\oplus a_i}$
restricted to the locus
\[
\bigcap \ev^{-1}_i(p_i)\subset \overline{\mm}_{0,\emptyset,\lf}\left(\vd\right)\, ,
\]
where $p_i\in\CP^1\backslash\RP^1$ are arbitrary different points
and
$$
\ev_i:\overline{\mm}_{0,\emptyset,\lf}\left(\vd\right)\to\CP^1
$$
are the evaluation maps. Since \eqref{integrand1st} depends on the boundary behaviour of $s,$ we must describe the boundary conditions for a multisection of the line bundle $\CL_i$.

We will concentrate only on real codimension $1$ boundary components.
There are two types of real codimension $1$ boundary strata. First, there are strata which parameterize stable disks with two disk components that share a common boundary node. The second type concerns
strata which parameterize stable disks whose boundary is contracted to a point. These are equivalent to spheres with one special point which lies in $\RP^1.$

In the first case, let $\bb$ be such a boundary stratum described by
\begin{equation*}
\bb=\oCM_{0,\sstar_{1},\lf_{1}}\left(\vd_{1}\right)\times_{L}\oCM_{0,\sstar_{2},\lf_{2}}\left(\vd_{2}\right)\xrightarrow{i_{\bb}}\partial\oCM_{0,\emptyset,\lf}\left(\vd\right)\, ,
\end{equation*}
where the fibered product $\times_L$ is taken with respect to the evaluation maps at $\sstar_1,\sstar_2,$ see the discussion which precedes Proposition \ref{prop:corners of moduli} for details and notations.
There is a map
\[f_\bb:\bb\to\oCM_1^{\DIAMOND}\times\oCM_2^{\DIAMOND},\] defined by \emph{forgetting the node}, where $$\oCM_i^{\DIAMOND}= \oCM_{0,\emptyset,\lf_{i}}\left(\vd_{i}\right)\, .$$
A section $s\in C^\infty(\CL_i\to\bb)$ is said to be \emph{coherent} if there exists a section
$$s'\in C^\infty(\CL_i\to\oCM_1^{\DIAMOND}\times\oCM_2^{\DIAMOND})$$
with $s = f_\bb^*s'.$
If $\lf_j=\{k\}$ and $\vd_{j}=\vec{0}$, then $\oCM_j^{\DIAMOND}$ is not stable and the line bundle $\CL_k$ is not pulled back from it. However, this boundary is not
relevant for the \emph{stationary} sector. Overcoming this issue in a theory which also involves non-stationary descendent insertions requires more work.

The second type of codimension $1$ boundaries is the exceptional boundary
\[
\ee=\ev_{\star}^{-1}\left(L\right)\subset{\oCM}_{0,\lf\sqcup\star}\left(d_{+}\right),
\]
which may appear only when  $d_{+}=d_{-}.$ We say that a section $s\in C^\infty(\CL_i\to\ee)$ is coherent if there exists a section $s'\in C^\infty(\CL_i\to{\oCM}_{0,\lf}\left(d_{+}\right))$ with $s=f^*s',$ where
$$f:\ee\to{\oCM}_{0,\lf}\left(d_{+}\right)$$ is the forgetful map.

A (multi)section is coherent if its restrictions to all codimension $1$ boundaries are coherent.
An example of a coherent section in the stationary theory is the following.
Consider the $i^{th}$ marked point $q_i$ which is mapped to $p_i.$
In a generic point of the moduli space
given by the stable map $u:\Sigma\to\CP^1$ and the markings data,  the map $u$ induces an isomorphism
near $p_i$
of the cotangent lines $T^*_{q_i}\Sigma$ and $T^*_{p_i}\CP^1\simeq \C$.
Choose some fixed $v\in \C$. Then the section $s\in\Gamma(\CL_i)$ given by $s_u=u^*v,$ whenever $u^*$ maps $T^*_{p_i}$ to $T^*_{q_i}$ isomorphically, and extended by $0$ otherwise is coherent. The zero locus of this section is exactly the subspace where the map $u$ is ramified at $q_i.$ This simple observation is the first hint of the open Gromov-Witten/Hurwitz correspondence to be discussed in the sequel \cite{BPTZ2}, and may also serve as a different path to proving it in genus $0$ (see also Proposition 1.1 in \cite{OP06}).

Consider now the closed setting and assume $d-1=\sum_{i\in \lf} a_i.$
We may calculate the closed Gromov-Witten invariant
\[
\left<\prod_{i\in \lf} \tau_i^{\pt}\right>_{0,d}
\] as
\[\#\left(Z(\bigoplus_{i\in\lf,j\in [a_i]} s_{ij})\cap \bigcap \ev^{-1}_i(p_i)\right),\]
where $p_i\in\CP^1$ are generic points, $ s_{ij}$ are generic
smooth multisections of $\CL_i$, and $\#Z( s)$ is the properly defined zero count of $ s.$

Consider next the open setting and assume $\sum_{i\in\lf} a_i=d_{+}+d_{-}-1.$
Suppose we have coherent multisections $s_{ij}$ of $\CL_i$ without a
common zero on the boundary.
Let $p_i\in\CP^1\backslash\RP^1,~i\in \lf$ be distinct points which
$p_i$ in the $\epsilon_i$ hemisphere ($+$ and $-$
stand for the upper and lower hemispheres, respectively).
We may then consider the number
\begin{equation}\label{eq:naive_int_num}\#\left(Z(\bigoplus_{i\in\lf,j\in [a_i]}\ s_{ij})\cap \bigcap \ev^{-1}_i(p_i)\right).\end{equation}
It turns out that such coherent multisections exist.
The remarkable fact, which is one of our main results here,
is that the count
\eqref{eq:naive_int_num} is independent of the generic choices.
The number \eqref{eq:naive_int_num}
will be denoted by
\[
\left<\prod \tau_{a_i}^{\epsilon_i}\right>_{0,\vd}.
\]

The main idea behind coherent sections is that over the boundary they are pulled back from moduli spaces of dimensions which are strictly smaller than the dimension of $\partial{\oCM}_{0,\emptyset,\lf}\left(\vd\right).$ This fact, and transversality arguments, show that if such multisections can be found, the intersection numbers that they define will be independent of the choice of the specific coherent multisections.
A similar idea has already appeared in \cite{BCT1,BCT2,PST14} and related papers in their definitions of open intersection numbers. In those papers, and here
as well, the scheme of boundary conditions gives rise to a beautiful enumerative theory.

In order to compute the invariants  \eqref{eq:naive_int_num}
using torus localization, we are lead to impose more refined conditions involving the recursive structure of the corners of the space, see Section
\ref{sub:coherent integrands}. Instead of coherent sections,
we will work with \emph{coherent equivariant forms}. Although there are no
fixed points on the boundary, we will see that the localization
formula expresses the invariants
%(including the ones at the non-equivariant limit)
as polynomials involving contributions from all the moduli spaces appearing at the corners.
% after having forgotten the special boundary points coming from the nodes.
A similar situation occurs when computing Welschinger's invariants in terms of fixed-point contributions, see \cite{fp-loc-OGW}.
\subsection{Equivariant localization formula for stationary descendents}
In order to write the equivariant localization formula for the open stationary descendents in $g=0,$ we begin with some definitions which are relevant also for the $g>0$ case. In Section \ref{sec:high_genus}, we will use these definitions to write the all-genera formula.
\subsubsection{Moduli specifications}\label{subsec:mod_spec}
The discrete data which classifies the connected components of moduli spaces of maps from marked surfaces with boundary, but without boundary markings, to $(\CP^1,\RP^1)$ can be encapsulated in the definition of a \emph{moduli specification}. Denote by $\Univ$ the set of possible labels of marked points (containing
$\N$ and other symbols such as $\star,\sstar$).
\begin{defn}\label{def:moduli_spec}
A moduli specification is a decorated graph $S=(V_b\sqcup V_w,E,g_s,\lf,\vd,d)$, where
\begin{enumerate}[(i)]
\item $V_b$ is the set of black vertices, $V_w$ is the set of white vertices, and $E$ is the set of edges,
\item $g_s$, $\lf$, and $\vd$ are functions,
  $g_s:V_b\to\zz_{\geq0}$ ,\ \
  $\lf:V_b\to 2^{\Univ}$ ,
  $$\vd=(d^+,d^-):V_b\to H_2(\CP^1,\RP^1;\zz)\simeq\zz\oplus\zz\, ,$$
\item $d$ is a function  $d:V_w\to \zz.$
\end{enumerate}
We require
\begin{enumerate}[(a)]
\item Each connected component of the graph contains exactly one black vertex.
  The graph is bipartite and each white vertex has a single (black) neighbor.
  We write $n(S)$ for the number of connected components. For a connected component $S'\subset S$, we denote the
  set of white vertices by $V_w(S')$ which we call
  the set of \emph{boundaries}. Let
  $$h(S')=|V_w(S')|$$ be the number of \emph{boundaries}.
  For the unique black vertex $v$ of $S'$, we define
  $$\lf(S')=\lf(v)\, ,\ \ \vd(S')=\vd(v)\, ,\ \ g_s(S')=g_s(v)$$
  which we call
  the \emph{boundary labels, total degree} and \emph{small genus}, respectively, of the component. The \emph{genus} of $S'$ is defined via \eqref{eq:small_g} to be $g(S')=2g_s(S')+h(S')-1.$ The \emph{boundary degrees} of $S'$ is the multiset $d|_{V_w(S')}.$
\item The sets $\lf(S')$ for different connected components are pairwise disjoint.
\item The stability condition for every connected component $S'$ holds:
\[
3\left(d^{+}(S')+d^{-}(S')\right)+2g_s(S')+2\left|\lf(S')\right|+h(S')> 2.
\]
\item For every connected component $S'$, we have
\[d^+(S') - \sum_{\{w\in V_w(S') | d(w) > 0\}}{d(w)} = d^-(S') + \sum_{\{w\in V_w(S') | d(w) < 0\}}{d(w)} \geq 0.\]
In particular,  $d^+(S')-d^-(S')=\sum_{w\in V_w(S')}{d(w)}$.
\end{enumerate}

An isomorphism between two moduli specifications $S_1$ and $S_2$
is a graph isomorphism which respects the decorations.
An automorphism is similarly defined, and we write $\Aut(S)$ for the automorphism group of $S$. A moduli specification
$S$ is \emph{connected} if $n(S)=1$.
It is \emph{closed} if for every connected component $S'\subset S$, we have
$h(S')=0$.
\end{defn}

There is an obvious way to read the data of
a moduli specification off of a holomorphic map
$$\left(\Sigma,\partial\Sigma\right)\to\left(\cc\pp^{1},\rr\pp^{1}\right)$$
from a Riemann surface with boundary $\Sigma$. We
associate a black vertex for each connected component of $\Sigma,$ a white vertex for each connected component of the boundary, and an edge from a white vertex to a black vertex if the corresponding boundary belongs to the corresponding connected component. The functions $\vd,\lf,g_s$ are the total degree, internal labels and small genus of the connected component which corresponds to the vertex, and $d$ is the degree of the corresponding boundary. The degree $d$ is with respect to the orientation on $\partial\Sigma$ induced from $\Sigma,$ and the chosen orientation of $\RP^1.$
%\begin{rmk}
%One can also define a moduli specification as a set whose elements, which correspond to connected components are themselves sets of genus, number of boundary components, degree and boundary degrees. The advantage in describing this data as a graph will become clear throughout the paper.
%\end{rmk}

A moduli specification $S$ is
\begin{enumerate}[\textbullet]
\item A \emph{disk moduli specification }if $n(S)=1,~g_s(S)=0$, and $h(S)=1$. In case we consider such a moduli specification, we denote it simply by $(\lf,\vd),$ omitting the rest of the notation.
\item \emph{sphere moduli specification }if $n(S)=1,~g_s(S)=0$, and $h(S)=0$. In case we consider such a moduli specification, we denote it simply by $(\lf,d),$ where $d=|\vd(S)|$ is the usual degree of a map from $\CP^1$ to $\CP^1.$
\item \emph{disk-cover moduli specification }if for every black vertex $v$, $\vd(v)$ is of the form $\left(d,0\right),$
or for every black vertex $v$, $\vd(v)$ is of the form $\left(0,d\right).$
\end{enumerate}
We denote the set of all disk moduli specifications with $\lf\subseteq\mathbb{N}$ by
$\mathcal{D}$
and the set of all sphere moduli specifications with $\lf\subseteq\mathbb{N}$ by
$\mathcal{S}$.

%We write $\mathcal{S}(\subseteq\lf),\mathcal{D}(\subseteq\lf)$ for the subsets of $\mathcal{S},\mathcal{D}$ made of pairs whose first coordinate is a subset of $\lf.$

\subsubsection{Fixed-point graphs}\label{subsubsec:fpg}

\begin{defn}\label{def:fpg}
A\emph{ fixed-point graph} $\Gamma$ is a tuple $\Gamma=\left(V\sqcup V_\circ,\FFF,\delta,\gamma,\mu,\lambda%,\kappa,\sigma
\right)$, where
\begin{enumerate}[(i)]
\item $V\sqcup V_\circ$ is a set of \emph{vertices}.
\item $\FFF$ is the set of \emph{hemisphere edges}.
\item $\delta:V_\circ\to\zz_{\geq0}^{2}\backslash\left\{ \left(0,0\right)\right\} $
is the \emph{degree function}.
\item $\gamma:V\to\zz_{\geq0}$ is the \emph{(vertex)} \emph{genus function}.
\item $\mu:V\to\left\{ \pm1\right\} $ is the \emph{(vertex)} \emph{map
function}.
\item $\lambda:V\to 2^{\Univ}$ is the association of \emph{vertex
labels}.
\end{enumerate}
The data is required to satisfy the following conditions:
\begin{enumerate}[(a)]
\item Every hemisphere edge is incident to one vertex of $V$ and one vertex of $V_\circ$.
\item Every vertex $v\in V_\circ$ is incident to either one hemisphere edge or two. In the former case, $v$ is called a \emph{boundary}, otherwise $v$ is called an \emph{equator}.
\item $\delta\left(v\right)=\left\{ \left(d,d\right)|d\in\zz_{>0}\right\} $ for an equator and
$\delta\left(v\right)\in\left\{ \left(d,0\right),\left(0,d\right)|d\in\zz_{>0}\right\} $
for a boundary.
\item $\mu\left(v_{1}\right)\neq\mu\left(v_{2}\right)$
whenever $v_{1},v_{2}$ are neighbors of the same equator.
\item For $v\neq u\in V$, the set $\lambda(v)$ is finite and $\lambda(u)\cap\lambda(v)=\emptyset.$
\end{enumerate}
There is an equivalence relation on $\FFF$ defined by declaring two hemisphere edges to be equivalent if they are incident to the same element of $V_\circ.$
Equivalence classes of this relation which are of size one are called \emph{disk edges}. We denote their set by $H$ and consider $H$ as a subset of $\FFF.$ The remaining equivalence classes are of size $2$ and are called \emph{sphere edges}. We denote their set by $E.$
The set of sphere edges is in bijection with equators, and the set of disk edges is in bijection with boundaries.

The \emph{contraction} of $\Gamma$ is the moduli specification $S$ obtained as follows. Associate a black vertex $v'$ to every connected component $\Gamma'$ of $\Gamma.$ Define
\[g_s(v')=h_1(\Gamma')+\sum_{v\in V(\Gamma')}\gamma(v),~\lf(v')=\bigsqcup_{v\in V(\Gamma')}\lambda(v),~\vd(v')=\sum_{u\in V_\circ(\Gamma')}\delta(u)\, .\]
Associate a white vertex  $v_h$ for any disk edge $h'$ of $\Gamma',$ connect it to $v'$, and set $d(v_h)=\partial^H(\delta(h)).$
A fixed-point graph $\Gamma$ is said to be of type $S,$ for a moduli specification $S,$ if its contraction yields~$S.$
\end{defn}
\begin{figure}[t]
\centering
\includegraphics[scale=.65]{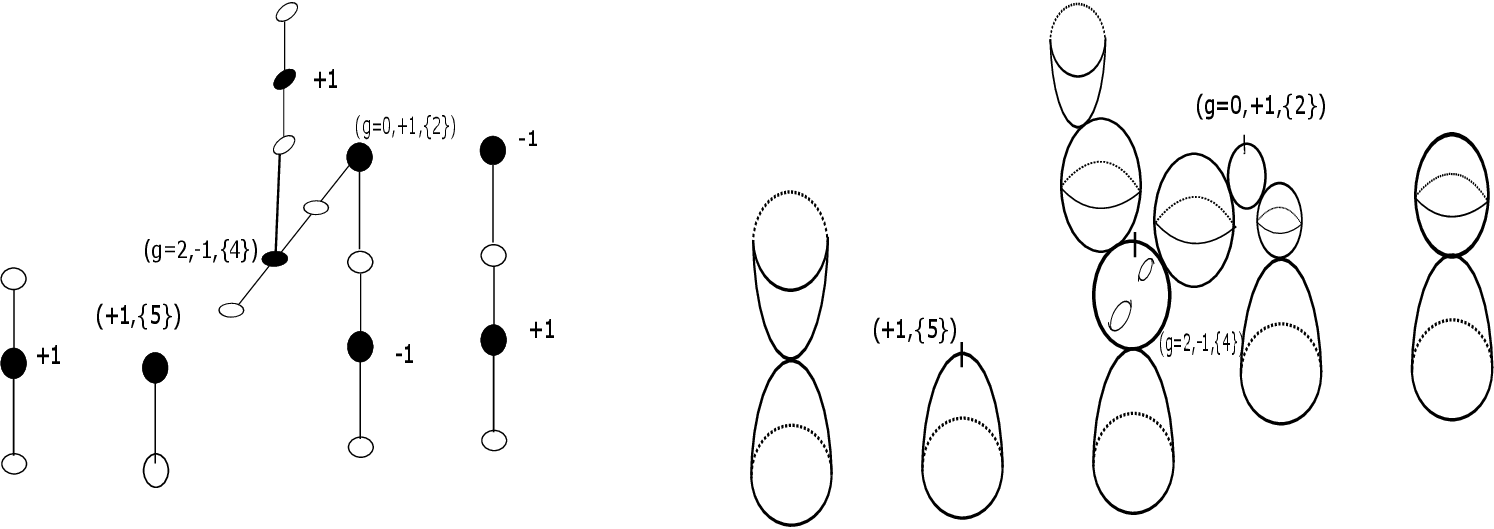}
\caption{Fixed points graphs and the corresponding geometric strata.
  For the purpose of illustration,
  we have drawn degree $0$ components as closed curves rather than points.
  Empty vertices stand for boundary and equator vertices. Next to the remaining vertices we write their genus, if they correspond to contracted component, $\mu$ and the markings.}
\label{fig:fixed_pts}
\end{figure}

\begin{rmk}
Every $f\in\FFF$ is incident to a unique $v\in V,$ which we denote by $v\left(f\right).$ We extend the domain of $\mu$ to include $\FFF$ by setting $\mu(f)=\mu\left(v\left(f\right)\right).$
We can also define degrees for elements of $\FFF\sqcup E$. For $e\in E$ we define $\delta(e)=\delta(v_\circ),$ where $v_\circ$ is the equator of $e.$ For $f\in\FFF$ if $\mu(f)=+1$ we put $\delta(f)=(|\delta(v_\circ)|,0),$ where $v_\circ$ is the unique vertex in $V_\circ$ incident to $f,$ otherwise we set $\delta(f)=(0,|\delta(v_\circ)|).$
\end{rmk}
Given an $S^{1}$-invariant stable open map $\left(\Sigma,\partial\Sigma\right)\xrightarrow{f}\left(\cc\pp^{1},\rr\pp^{1}\right)$,
we can read off its fixed-point graph $\Gamma$ as follows.
\begin{enumerate}[\textbullet]
\item A sphere edge $e\in E$ corresponds to a rational irreducible component
$\cc\pp^{1}\simeq C\subset\Sigma$ with $f|_{C}$ given by the degree $d$
cover of $\cc\pp^{1}$ branched at $p_{\pm}$  and  $\delta\left(e\right)=\left(d,d\right)$.
\item A disk edge $h\in H$ corresponds to a disk component $D^{2}\simeq C\subset\Sigma$
with $f|_{C}$ given by the standard degree $d$ cover of either the northern or
southern hemisphere (branched at the pole).  And $\delta\left(h\right)$
is $\left(d,0\right)$ or $\left(0,d\right)$, accordingly.
\item Vertices correspond either to a contracted component of the domain or to a single point of the
domain. Either way, $f$ maps
such components to $p_{\pm}$ according to $\mu\left(v\right)=\pm1$.
\end{enumerate}
\begin{rmk}
$E$ can be naturally thought of as the edge set of a graph with vertices $V,$ and $H$ can be thought of as half-edges, or tails of such a graph. From the data of $V,E,H,\delta|_{E\sqcup H}$, we  can reconstruct $\FFF,V_\circ,\delta.$

We include $\FFF,V_\circ$ in the definition because later, in Section \ref{sec:high_genus}, where we will have to work with more complicated graphs, the addition of $\FFF, V_\circ$ will simplify definitions. However,
we will sometimes denote a fixed-point graph by the
more standard notation $\left(V,E,H,\delta|_{E\sqcup H},\gamma,\mu,\lambda\right).$
\end{rmk}

An \emph{isomorphism} of fixed-point graphs
\[
\left(V,E,H,\delta,\gamma,\mu,\lambda\right)\to\left(V',E',H',\delta',\gamma',\mu',\lambda'%,\kappa',\sigma'
\right)
\]
is a triple of bijections $V\to V',E\to E',H\to H'$ that preserve
the rest of the structure.
Let $\Aut(\Gamma)$ be the
 finite group  automorphisms of the fixed-point graph
$\Gamma$.

We define the finite group of \emph{covering automorphisms},
\[
A^0_\Gamma=\prod_{e\in E}\zz/\left|\delta\left(e\right)\right|\times\prod_{h\in H}\zz/\left|\delta\left(h\right)\right|.
\]
We define the finite group $A_\Gamma$ of \emph{geometric automorphisms},
as the semidirect product
\[
1\to A^0_\Gamma\to A_{\Gamma}\to\Aut\left(\Gamma\right)\to1,
\]
where $\Aut(\Gamma)$ acts on $A^0_\Gamma$
by permuting the hemisphere edges in the obvious way.

%A \emph{flag} $f$ of $\Gamma$ is either a half edge $h\in H\left(\Gamma\right)$
%or an edge $e\in E\left(\Gamma\right)$ together with an incident
%vertex. Either way, there is a unique vertex $v\in V$ associated with every
%flag $f$ which we denote $v\left(f\right)$. We write $\mu\left(f\right)=\mu\left(v\left(f\right)\right)$,
% and $\delta(f)$ for the degree of the corresponding edge or half edge. The set of flags is denoted $F\left(\Gamma\right)$.

For $v\in V\left(\Gamma\right)$,
we let
\[
\val\left(v\right)=\left\{ f\in F\left(\Gamma\right)|v\left(f\right)=v\right\}
\]
denote the \emph{set} of incident hemisphere edges (so $\left|\val\left(v\right)\right|$
is the usual ``valency'' of $v$ in the graph). We write
\[
\val_{+}\left(v\right)=\val\left(v\right)\sqcup\lambda\left(v\right).
\]
We set

\emph{
\[
V_{-}\left(\Gamma\right)=\left\{ v\in V\left(\Gamma\right)|\left(\gamma\left(v\right),\left|\val_{+}\left(v\right)\right|\right)\in\left\{ \left(0,0\right),\left(0,1\right),\left(0,2\right),\left(1,0\right)\right\} \right\}
\]
}and $V_{+}\left(\Gamma\right)=V\left(\Gamma\right)\backslash V_{-}\left(\Gamma\right)$.
$v\in V_{+}\left(\Gamma\right)$ corresponds to a contracted component of the domain, while $v\in V_{-}\left(\Gamma\right)$ corresponds to a single point of the domain.

We let
\[
\overline{\mm}_{\Gamma}=\prod_{v\in V_{+}\left(\Gamma\right)}\overline{\mm}_{\gamma\left(v\right),\val_{+}\left(v\right)}.
\]

\begin{obs}
Given an $S^{1}$-invariant stable open map $\left(\Sigma,\partial\Sigma\right)\xrightarrow{f}\left(\cc\pp^{1},\rr\pp^{1}\right)$,
which corresponds
to a fixed-point graph $\Gamma$ as above, the isotropy group of $f$ can be identified with the group of geometric
automorphisms $A_{\Gamma}$. The space of all stable open maps corresponding
to $\Gamma$ is isomorphic to the stack quotient of $\overline{\mm}_{\Gamma}$
by the group $A_{\Gamma}$.
\end{obs}

\subsubsection{Contribution of a fixed-point graph}\label{subsec:fp_graph_cont}
Let $\Gamma$ be a fixed-point graph. We construct  vector bundles
on $\overline{\mm}_{\Gamma}$ by pulling back the Hodge and cotangent
bundles along the projections $\overline{\mm}_{\Gamma}\to\overline{\mm}_{\gamma\left(v\right),\val_{+}\left(v\right)}$.
First, for $v\in V_{+}\left(\Gamma\right)$, we let $\mathbb{E}_{v}^{\vee}$
denote the pullback of the dual of the Hodge bundle on $\overline{\mm}_{\gamma\left(v\right),\val_{+}\left(v\right)}$.
The corresponding Chern polynomial is denoted
\[
c\left(\mathbb{E}_{v}^{\vee}\right)\in H^{\bullet}\left(\overline{\mm}_{\Gamma}\right)\left[t\right].
\]
For $f\in F\left(\Gamma\right)$ with $v\left(f\right)=v\in V_{+}\left(\Gamma\right)$,
we can identify $f$ with a marking on $\overline{\mm}_{\gamma\left(v\right),\val_{+}\left(v\right)}$,
and we then let $\mathbb{L}_{f}$ denote the pullback of the corresponding
cotangent line bundle. We denote the first Chern class of
$\mathbb{L}_{f}$
by
\[
\psi_{f}=c_{1}\left(\mathbb{L}_{f}\right)\in H^{\bullet}\left(\overline{\mm}_{\Gamma}\right).
\]

Our invariants are (a priori) defined over the localized equivariant
cohomology of a point $\qq\left[u,u^{-1}\right]$, where $u$ is a
generator of degree $+2$. Let
\[
\omega_{f}=\mu\left(f\right)\frac{2u}{\left|\delta\left(f\right)\right|}.
\]

We define a class $e_{\Gamma}^{-1}\in H^{\bullet}\left(\overline{\mm}_{\Gamma};\qq\left[u,u^{-1}\right]\right)$
by

\begin{align}
\label{eq:euler^-1}e_{\Gamma}^{-1}=&\prod_{\left\{ f\in F\left(\Gamma\right)|v\left(f\right)\in V_{+}\left(\Gamma\right)\right\} }\frac{1}{\omega_{f}-\psi_{f}}\mu\left(f\right)\left(2u\right)
\prod_{v\in V_{+}\left(\Gamma\right)}%\prod_{\nu\neq\mu\left(v\right)}
c\left(\mathbb{E}_{v}^{\vee}\right)\left(t=\frac{\mu\left(v\right)}{2u}\right)\left(2u\right)^{\gamma\left(v\right)-1}\\
&\times
\prod_{\left\{ v\in V_{-}\left(\Gamma\right)|\val_{+}\left(v\right)=\left\{ f_{1},f_{2}\right\} \right\} }\frac{1}{\omega_{f_{1}}+\omega_{f_{2}}}\prod_{\left\{ v\in V_{-}\left(\Gamma\right)|\val_{+}\left(v\right)=\left\{f\right\}\right\} }\omega_{f}\notag\\
&\times
\prod_{e\in E}\frac{\left(-1\right)^{\left|\delta\left(e\right)\right|}\left|\delta\left(e\right)\right|{}^{2\left|\delta\left(e\right)\right|}}{\left(\left|\delta\left(e\right)\right|!\right)^{2}\left(2u\right)^{2\left|\delta\left(e\right)\right|}}
\prod_{h\in H}\mu(h)^{|\delta(h)|+1}\frac{\left(\left|\delta\left(h\right)\right|\right){}^{\left|\delta\left(h\right)\right|}}{\left|\delta\left(h\right)\right|!\left(2u\right)^{\left|\delta\left(h\right)\right|}}.\notag
\end{align}

We define a class $\alpha_{\Gamma}^{\vec{a},\vec{\epsilon}}\in H^{\bullet}\left(\overline{\mm}_{\Gamma};\qq\left[u,u^{-1}\right]\right)$ to be $0$ if $\epsilon_i\neq\mu(v),$ for some $v\in V_+\cup V_-$ and $i\in\lambda(v),$
and otherwise by

\begin{equation}\label{eq:alpha_term}
\alpha_{\Gamma}^{\vec{a},\vec{\epsilon}}=\left(\prod_{v\in V_{+}\left(\Gamma\right)}\prod_{i\in\lambda\left(v\right)}\psi_{i}^{a_{i}}\mu\left(v\right)\,2u\right)\left(
\prod_{\substack{v\in V_{-}\left(\Gamma\right),\\\lambda\left(v\right)=\left\{ i\right\},\val(v)=\left\{f\right\} }}
\left(-\omega_f\right)^{a_{i}}\mu\left(v\right)\,2u\right).
%\left(-\frac{\mu\left(v\right)\,2u}{|\delta(f)|}\right)^{a_{i}}\mu\left(v\right)\,2u\right).
\end{equation}

\begin{rem}
  In all genera,  $e_{\Gamma}^{-1}$ is conjectured to
  be the inverse to the equivariant Euler class
of the virtual normal bundle to the fixed-point component $\ff_{\Gamma}$
parameterized by $\Gamma$.
In genus $0$, which is what we consider in all sections except Section \ref{sec:high_genus}, $e_{\Gamma}^{-1}$ is as claimed. Moreover,
in genus 0,
the second term in \eqref{eq:euler^-1} is simplified to $\prod_{v\in V_{+}\left(\Gamma\right)}\left(2u\right)^{-1}=(2u)^{-|V_+(\Gamma)|}.$
The class
$\alpha_{\Gamma}^{\vec{a},\vec{\epsilon}}$ should be the pullback of the integrand, a product of stationary descendents as in \eqref{integrand1st}, to $\ff_{\Gamma}$.
\end{rem}
\begin{defn}\label{def:inhomterms}
  Let $S$ be a moduli specification, $\vec{a},\vec{\epsilon}$ vectors as in Notation \ref{nn:a_eps}.  The \emph{fixed-point contribution}
$I\left(S,\vec{a},\vec{\epsilon}\right)$
  {\em associated
to $S,\vec{a},\vec{\epsilon},$ } is defined by the equation
\begin{equation}\label{eq:I_S}
\frac{I\left(S,\vec{a},\vec{\epsilon}\right)}{|\Aut(S)|}=\sum_{\Gamma}\frac{1}{|\Aut(\Gamma)|}I(\Gamma,\vec{a},\vec{\epsilon}),
\end{equation}
where the sum ranges over all isomorphism types of fixed-point graphs $\Gamma$ corresponding
to $S$, and
\begin{equation}\label{eq:I_Gamma}
I(\Gamma,\vec{a},\vec{\epsilon})=\frac{1}{|A^0_\Gamma|}\int_{\overline{\mm}_{\Gamma}}e_{\Gamma}^{-1}\cdot\alpha_{\Gamma}^{\vec{a},\vec{\epsilon}}
\end{equation}
is the \emph{fixed-point contribution of $\Gamma$}. In case $S$ is a disk moduli specification $(\lf,\vd)$ or sphere moduli specification $(\lf,d),$ we write $I(\lf,\vd,\vec{a},\vec{\epsilon}),\, I(\lf,d,\vec{a},\vec{\epsilon})$ respectively for the associated fixed-point contributions.
\end{defn}
For later convenience we set $I(S,\vec{a},\vec{\epsilon})$ and $I(\Gamma,\vec{a},\vec{\epsilon})$ to be $0$ if some $a_i<0.$

\subsubsection{Genus $0$ equivariant localization formula}
We can now state our result in genus $0.$

\begin{defn}\label{def:non_lab_non_ord_trees_and_ampli}
For a disk moduli specification $(\lf,\vd)$ denote by $\TTT(\lf,\vd)$ the set of trees $T$ such that each vertex $v$ is decorated by a disk moduli specification $(\lf_v,\vd_v)$ such that
\[
\bigsqcup_{v\in V(T)}\lf_v=\lf,\qquad\sum_{v\in V(T)}\vd_v=\vd.
\]
An \emph{automorphism} of an element of $\TTT(\lf,\vd)$ is an automorphism of the underlying tree which preserves the moduli specifications of the vertices.
Let $\Aut(T)$ denote the automorphism group of~$T.$

The \emph{amplitude} of $T$ is defined by
\begin{equation}\label{eq:ampli_intro_g=0}
\AAA(T,\vec{a},\vec{\epsilon})=\left(\frac{(-2u)^{-|E(T)|}}{|\Aut(T)|}\prod_{v\in V(T)}({\vd^+(v)-\vd^-(v)})^{\val(v)}\right)I(S(T),\vec{a},\vec{\epsilon}),
\end{equation}
where $V(T),~E(T)$ are the vertex and edge sets of $T,~\val(v)$ is the number of neighbors of $v$ in $T,$ and $S(T)=\{(\lf_v,\vd_v)\}_{v\in V(T)}$ is the moduli specification associated to $T$.
\footnote{If we denote by $\vec{a}|_{\lf_v},\vec{\epsilon}_{\lf_v}$ the $\lf_v-$components of $\vec{a},\vec{\epsilon}$ respectively, then $I(S(T),\vec{a},\vec{\epsilon})=\prod_{v\in V(T)}I(\lf_v,\vd_v,\vec{a}|_{\lf_v},\vec{\epsilon}|_{\lf_v})$.}

Define\footnote{The dependence on $\lf$ is only notational. \eqref{eq:nice_tree_sum} depends only on the degree $\vd,$ vector of descendents $\vec{a}$ and the points constraints $\vec{\epsilon},$ and is by definition permutation invariant.}
\begin{equation}\label{eq:nice_tree_sum}
\OGW(\lf,\vd,\vec{a},\vec{\epsilon})=
\sum_{T\in\TTT(\lf,\vd)}\AAA(T,\vec{a},\vec{\epsilon})+\delta_{d^+-d^-}\frac{|\vd|}{2u}I(\lf,|\vd|,\vec{a},\vec{\epsilon}).\end{equation}
\end{defn}
The main result of the paper is the following theorem, see Theorem \ref{thm:int_nums_equal_tree_sum} for the precise statement of the boundary conditions.
\begin{thm*}
The open equivariant stationary descendent $\<\prod\tau_{a_i}^{\epsilon_i}\>_{0,\vd}$ defined via coherent boundary conditions equals $\OGW(\lf,\vd,\vec{a},\vec{\epsilon}).$
In particular, \eqref{eq:nice_tree_sum} vanishes if $1+\sum a_i<d^++d^-$.
\end{thm*}
\begin{figure}[t]
\centering
\includegraphics[scale=.55]{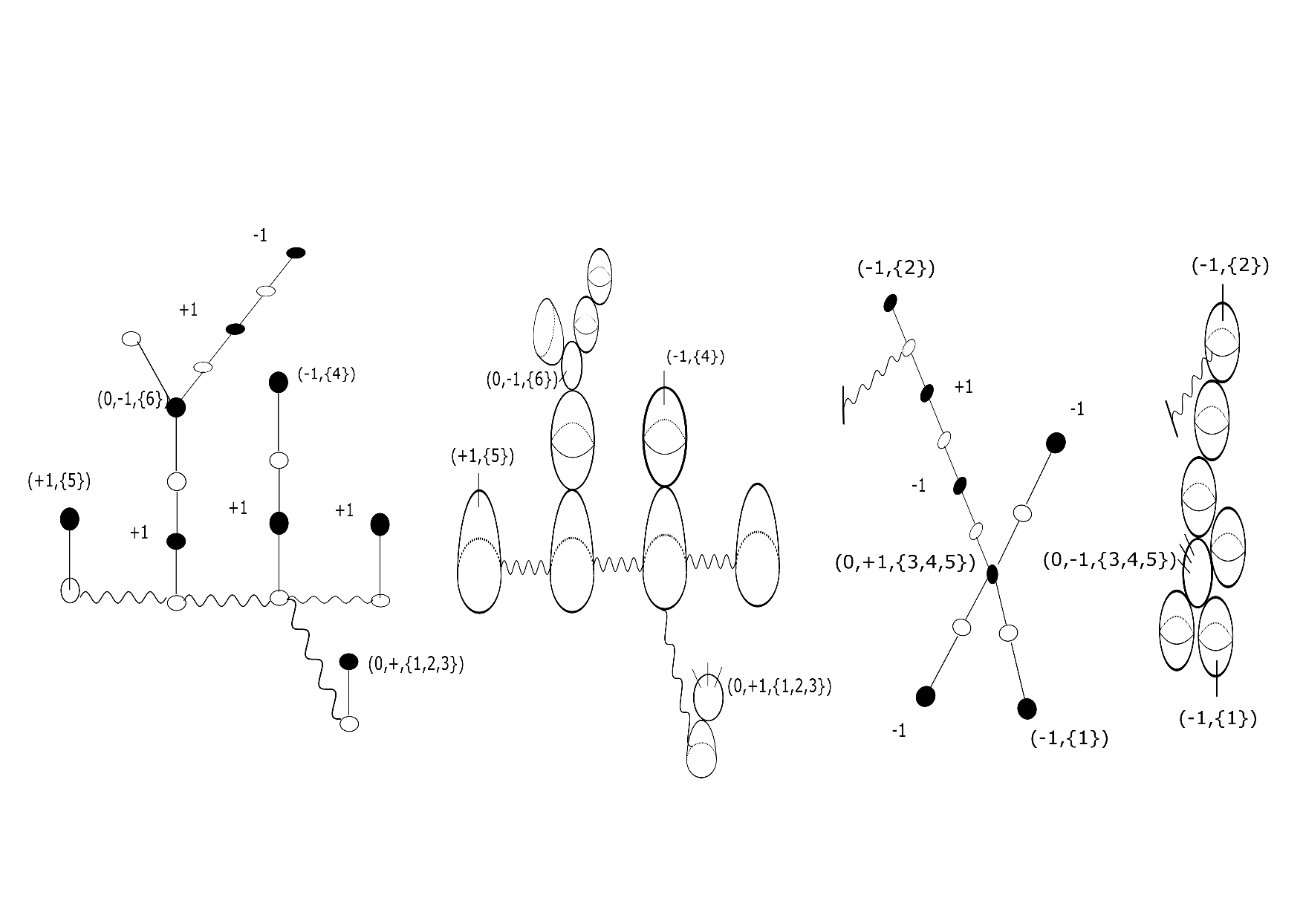}
\caption{Localization graphs which appear in the formula for genus $0$ invariants. We draw a graph and on its right the corresponding geometric picture. Again for illustration reasons we do not contract degree $0$ components, and again next to a contracted component vertex we write the genus, $\mu$ and the markings, while next to other vertices we write the labels and $\mu.$}
\label{fig:g_0_cont}
\end{figure}

\subsection{Explicit formulas and recursion}
In the case of disk covers, we have the following explicit formula for the open stationary descendents.
\begin{thm}\label{thm:disk_disk_covers_formula}
Suppose $a_i\geq 0$ and $1+\sum_{i=1}^n a_i=d$, then
\begin{equation}\label{eq:disk_disk_covers_formula}
\<\prod_{i=1}^n\tau^{+}_{a_i}\>_{0,(d,0)}%^{\circ}
=\frac{(1+\sum a_i)^{n-2}}{\prod a_i!}.
\end{equation}
%Similarly,
%\begin{equation}\label{eq:disk_disk_covers_formula_minus}
%\<\prod_{i=1}^n\tau^{-}_{a_i}\>_{0,(0,d)}%^{\circ}
%=-\frac{(1+\sum a_i)^{n-2}}{\prod a_i!}.
%\end{equation}
\end{thm}
An additional nice formula is the vanishing of intersection numbers for $\vd=(d,d),$ see Lemma~\ref{lem:(d,d)_vanishing} below.
These formulas will play a significant role in the sequel.

By analogy to the closed GW theory of $\CP^1,$
we have divisor equations and topological recursions (compare to the TRRs of \cite{PST14,BCT1,BCT2}) also in the open setting.
\begin{thm}\label{prop:divisor}
The equivariant divisor equation is
\[\<\tau_0^{\epsilon_1}\prod_{i=2}^l\tau_{a_i}^{\epsilon_i}\>_{0,\vd}
=d^{\epsilon_1}\<\prod_{i=2}^l\tau_{a_i}^{\epsilon_i}\>_{0,\vd}+
(2u)\sum_{\substack{j>1,\\\epsilon_j=\epsilon_1,~a_j>0}}\<\tau_{a_j-1}^{\epsilon_j}\prod_{i=2,i\neq j}^l\tau_{a_i}^{\epsilon_i}\>_{0,\vd}
.\]
In particular, if $1+\sum_{i=2}^l a_i=d_{+}+d_{-}>0$, then
\[\<\tau_0^{\epsilon_1}\prod_{i=2}^l\tau_{a_i}^{\epsilon_i}\>_{0,\vd}%^{\circ}
=d^{\epsilon_1}\<\prod_{i=2}^l\tau_{a_i}^{\epsilon_i}\>_{0,\vd}%^{\circ}
.\]
\end{thm}

\begin{thm}\label{thm:open_stat_trr_disk_cover}
Suppose $a_i\geq 0$, then
\begin{align}\label{eq:trr_disk_disk}
\<\tau^+_{a_i+1}\prod_{i=2}^l\tau^{+}_{a_i}\>_{0,(d,0)}=&\sum_{\substack{R\sqcup S = \{3,\ldots,l\}}}(2u)^{|R|}\<\tau_{a_1}\left(\prod_{i \in R} \tau_{a_i}\right)\tau_0\>_0^c\<\tau^+_0\tau^+_{a_2}\prod_{i \in S} \tau^+_{a_i}\>_{0,(d,0)}+\\
&+\sum_{\substack{R\sqcup S = \{3,\ldots,l\} \\ d_1 + d_2 = d}}d_2 \<\tau^+_{a_1} \prod_{i \in R} \tau^+_{a_i}\>_{0,(d_1,0)}\<\tau^+_{a_2}\prod_{i \in S} \tau^+_{a_i}\>_{0,(d_2,0)},\notag
\end{align}
where $\<\cdots\>^c_0$ stands for the closed descendent integrals over the Deligne-Mumford moduli space $\overline{\mm}_{0,n}.$
In particular, in the non-equivariant limit, $2+\sum_{i=1}^l a_i=d,$ we have
\begin{gather}\label{eq:no_internal_trr}
\<\tau^+_{a_i+1}\prod_{i=2}^l\tau^{+}_{a_i}\>_{0,(d,0)}%^{\circ}
=\sum_{\substack{R\sqcup S = \{3,\ldots,l\} \\ d_1 + d_2 = d}}d_2 \<\tau^+_{a_1} \prod_{i \in R} \tau^+_{a_i}\>_{0,(d_1,0)}%^{\circ}
\<\tau^+_{a_2}\prod_{i \in S} \tau^+_{a_i}\>_{0,(d_2,0)}%^{\circ}
,
\end{gather}
where in the sum only non-equivariant intersection numbers appear.
\end{thm}
\begin{rmk}There is also a recursion for sphere covers
  which involves non-stationary insertions and
  will be handled in a future paper.\end{rmk}

\subsection{The main conjecture}
In Section \ref{sec:high_genus},
we propose a formula, motivated by equivariant localization, for all
higher genus open stationary descendents
which generalizes in a non-trivial way the proven formula for disks.
Our main conjecture, Conjecture \ref{conj:high_genus_def_and_loc} of Section
\ref{tmconj},
is that our formula is geometric:
{\em the open invariants in all higher genus can be defined geometrically via
coherent boundary condition, and
the resulting numbers agree with our high genus proposal}.

\subsection{Structure of the paper}
In Section \ref{sec:GW-defs}, we define the relevant moduli spaces, tautological line bundles, and the key objects for our calculations and definitions -- coherent integrands. We also explain how to obtain coherent integrands from the tautological line bundles. In Section \ref{sec:genus0-calc}, we prove the fixed-point localization formula and prove an algorithm for integrating coherent integrands.
In Section~\ref{sec:proofs}, we prove the main results:
Theorems \ref{thm:int_nums_equal_tree_sum},~\ref{thm:open_stat_trr_disk_cover},
and \ref{thm:disk_disk_covers_formula}.
In Section \ref{sec:high_genus}, we define the higher genus intersection numbers via localization. In Section \ref{sec:or}, we construct and analyze
the canonical orientation for the moduli spaces.

\subsection{The sequel}
The sequel to this paper  will be devoted to the higher genus theory
and to the open Gromov-Witten/Hurwitz correspondence:

\begin{enumerate}[\textbullet]
  \item
We will sketch a proof, modulo missing foundations, of the higher
genus formula and provide more evidence for its correctness,
including a genus $1$ recursion.
In addition, we will prove a \emph{map decomposition} theorem which
will allow us, in particular, to write all closed Gromov-Witten
invariants of $\CP^1$ in terms of the (conjectural higher genus)
\emph{disk cover} invariants defined in Section \ref{sec:high_genus}.

\item
We will define  \emph{open Hurwitz theory}, which is a natural generalization of classical Hurwitz theory to surfaces with boundary, and
prove that the genus $0$ Gromov-Witten
theory of disk covers satisfies an open GW/H correspondence, using the same completed cycles as in~\cite{OP06}. In fact, assuming a weak form of Conjecture \ref{conj:high_genus_def_and_loc},
we will prove the open GW/H  result in all genera.
\end{enumerate}

\subsection{Acknowledgments}
We thank Penka Georgieva for discussions about orientations, Hurwitz numbers, and other topics. We thank
Yoel Groman and Pavel Giterman for several enlightening discussions regarding maps to $\CP^1$ and how to visualize them.

The work of A.~B. is supported by the Mathematical Center in Akademgorodok under agreement No. 075-15-2019-1675 with the Ministry of Science and Higher Education of the Russian Federation.

A.~N.~Z. was partially supported by NSF grant DMS-1638352 and ERC-2012-AdG-320368-MCSK.

R.~P. was partially supported by SNF-200020-182181, SwissMAP, and the
Einstein Stiftung. The project has received funding from the European Research Council (ERC)
under the European Union Horizon 2020 research and innovation program (grant
agreement No.  786580).

R.~T. (incumbent of the Lillian and George Lyttle Career Development Chair) was supported by the ISF (grant No. 335/19), by a research grant from the Center for New Scientists of Weizmann Institute, by Dr. Max R\"ossler, the Walter Haefner Foundation, and the ETH Z\"urich Foundation, and partially by ERC-2012-AdG-320368-MCSK.

\section{Moduli of stable disk-maps and stationary descendent integrals}\label{sec:GW-defs}
The goal of this section is to define the stationary descendent integrals over the moduli space of stable disk-maps. We will first review the basics of equivariant localization, and continue to define stable maps to $(\CP^1,\RP^1),$ the moduli of stable genus $0$ maps to $(\CP^1,\RP^1),$ the tautological line bundles and the natural $S^1$ action. We will then define the notion of \emph{coherent integrand} and show how to obtain a coherent integrand out of the tautological line bundles. We shall end by defining stationary descendent integrals over the moduli space of stable disk-maps, and stating the main theorem, Theorem \ref{thm:int_nums_equal_tree_sum}.

\subsection{$S^1-$orbifolds and equivariant forms}

All our orbifolds will have corners unless explicitly mentioned otherwise.
We refer the reader to \cite{mod2hom} Section $3$ for a precise definition of
this and related differential-geometric notions (differential forms,
vector bundles, various types of maps, etc.) which we will use without
further discussion.

An $S^{1}$-orbifold $\yy$ is an orbifold with corners which is equipped
with an $S^{1}$ action. Let $\xi$ denote the vector field generating
the action, and consider the differential graded algebra of \emph{equivariant
forms}
\[
\mathcal{A}_{S^{1}}\left(\yy\right)=\left(\Omega\left(\yy;\rr\left[u\right]\right)^{S^{1}},D=d-u\iota_{\xi},\cdot\right)
\]
consisting of $S^{1}$-invariant differential forms on $\yy$ with
values in the graded polynomial algebra $\rr\left[u\right],\deg1 u=2$.
The differential is a deformation of the exterior derivative by contraction
with $\xi$, and the product is the usual exterior product of forms
(taking values in an evenly-graded ring). The grading is the total
grading, coming from the De-Rham degree and the grading of $\rr\left[u\right]$.
The homology of this differential graded algebra is isomorphic to
the cohomology of the homotopy quotient of $\yy$ by the $S^{1}$
action, also known as the\emph{ equivariant cohomology} of $\yy$,
see \cite{AB}.

The algebra of \emph{Laurent equivariant forms} $\mathcal{A}_{S^{1}}\left(\yy\right)\left[u^{-1}\right]$
is the localization of $\mathcal{A}_{S^{1}}\left(\yy\right)$ by the
multiplicative subset $\left\{ 1,u,u^{2},...\right\} $.

If $f:\xx\to\yy$ is an equivariant map, the usual pullback of forms
restricts to a map $f^{*}:\aso\left(\yy\right)\to\aso\left(\xx\right)$.
If $f$ is an equivariant proper and oriented submersion then we can
consider the usual pushforward of forms (defined by integration over
the fiber) as a map
\[
f_{*}:\aso\left(\xx\right)\to\aso\left(\yy\right).
\]
Our convention for the pushforward is such that
\begin{equation}
f_{*}\left(f^{*}\alpha\cdot\beta\right)=\alpha\cdot f_{*}\beta\label{eq:pf convention}
\end{equation}
and
\[
g_{*}\circ f_{*}=\left(-1\right)^{\deg1 g\cdot\deg1 f}\left(g\circ f\right)_{*}.
\]
The induced orientation on $\partial\xx$ is defined so that the orientation
of $\xx$ is recovered by appending an outward normal vector to the
beginning of an oriented base for $\partial\xx$, so that
\[
\left(f_{*}D-Df_{*}\right)\alpha=\left(f_{*}d-df_{*}\right)\alpha=\left(-1\right)^{\deg1\alpha+\dim\yy-\dim\xx}\left(f|_{\partial\xx}\right)_{*}\alpha.
\]

In particular, if $\xx$ is a compact and oriented orbifold with corners, then the pushforward to a point gives an integration map
\[
\int:\aso\left(\xx\right)\to\rr\left[u\right],
\]
satisfying Stokes' theorem
\[
\int_{\xx}D\alpha=\int_{\partial\xx}\alpha.
\]

Since all of these operations are $\rr\left[u\right]$-linear they
extend to Laurent equivariant forms as well.
\begin{defn}
A Laurent equivariant form $Q\in\mathcal{A}_{S^{1}}\left(\yy\right)\left[u^{-1}\right]$
is called \emph{an equivariant primitive }(or just \emph{a primitive})
if $DQ=1$, the unit of $\mathcal{A}_{S^{1}}\left(\yy\right)\left[u^{-1}\right]$.
\end{defn}
If $\yy$ has no fixed points, $\yy^{S^{1}}=\emptyset$, then $Q=\frac{\eta}{D\eta}$
is a primitive where $\eta=g\left(\xi,-\right)$ is the dual to $\xi$
with respect to any $S^{1}$-invariant Riemannian metric $g$ on $\yy$.
\begin{defn}
An $S^{1}$-orbifold $\xx$ is said to have \emph{simple fixed points}if $\partial\xx\cap\xx^{S^{1}}=\emptyset$.
\end{defn}
If $\xx$ has simple fixed points, then $\xx^{S^{1}}$ is represented
by a disjoint union of orbifolds without boundary (not necessarily
of the same dimension) together with a closed embedding $\xx^{S^{1}}\to\xx$
(see \cite[ Definition 36]{fp-loc-OGW}) and for every $c\geq1$ the
orbifold with corners $\yy=\partial^{c}\xx$ is fixed point free.
All the orbifolds we will consider here have simple fixed points.

If $F\subset\xx^{S^{1}}$ is a connected component, we denote by $N_{F}$
the normal bundle to $F\to\xx$ and let
\[
e_{S^{1}}\left(N_{F}\right)\in\aso\left(F\right)^{\times}\subset\aso\left(F\right)
\]
denote an associated equivariant Euler form (see the proof of Lemma
\ref{lem:fp contrib by euler} below).

\subsection{Stable maps to $(\CP^1,\RP^1)$}
In this section we briefly recall the definition of stable maps to $(\CP^1,\RP^1),$ for a definition in the more general case of a target pair $(X,L), $ see \cite{Liu02}.

\begin{definition}
A \emph{Riemann surface with boundary} $\Sigma$ is the result of removing
finitely many disjoint extendably embedded open disks from a compact Riemann surface.
Given disjoint finite sets $\kf,\lf\subset \Univ,$
a $(\kf,\lf)-$\emph{smooth marked Riemann surface (with boundary)} is a triple
$(\Sigma, \{x_i\}_{i\in \kf}, \{z_i\}_{i\in \lf} )$
where
\begin{enumerate}[(a)]
\item $\Sigma$ is a Riemann surface (with boundary).
\item $\{x_i\}_{i\in \kf}$ are distinct boundary points.
\item $\{z_i\}_{i\in \lf}$ are distinct internal points.
\end{enumerate}
The points $x_i$ are called boundary markings, the points $z_i$ are the internal markings.
We will sometimes omit the marked points and $(\kf,\lf)$ from our notations; we will sometimes identify a marking with its label from $\kf\cup\lf.$

The \emph{genus} $g(\Sigma)$ of a connected marked smooth Riemann surface with boundary $\Sigma$ is defined as the genus of the doubled surface $\Sigma_\C,$ if $\partial\Sigma\neq\emptyset,$ and as the usual genus otherwise.
$\Sigma$ is \emph{stable} if its automorphism group is finite.
\end{definition}
When $\Sigma$ is connected and $\partial\Sigma=\emptyset$ stability is equivalent to $2g(\Sigma)+|\lf|\geq 3,$ while when $\partial\Sigma\neq\emptyset$ stability is equivalent to $2g(\Sigma)+2|\lf|+|\kf|\geq 3.$ When $\Sigma$ is not connected, then stability is equivalent to the stability of all the connected components.
\begin{definition}
A $(\kf,\lf)-$\emph{pre-stable marked surface with boundary} is a tuple
$\Sigma=((\Sigma_\alpha)_{\alpha\in C\sqcup O},\sim_B,\sim_I,\mathcal{CB})$, where
\begin{enumerate}[(a)]
\item $C,O$ are finite sets. For $\alpha\in C$ $\Sigma_\alpha$ is a smooth marked Riemann surface without boundary, for $\alpha\in O$ $\Sigma_\alpha$ is a smooth marked Riemann surface with boundary.
\item An equivalence relation $\sim_{B}$ on the set of all boundary marked
points, with equivalence classes of size at most $2.$ An equivalence relation $\sim_I$ on the set of all internal marked points, with
equivalence classes of size at most $2.$
\item A subset $\mathcal{CB}$ of the equivalence classes of size $1$ of $\sim_I.$
\end{enumerate}
We require that $\kf$ is the set of labels of points belonging to $\sim_B-$equivalence
classes of size $1.$ $I$ is the set of labels of points belonging to $\sim_I$
equivalence classes of size $1$ which do not lie in $\mathcal{CB}.$

The two equivalence relations $\sim_B,\sim_I$ taken together are denoted
by $\sim$. Equivalence classes of $\sim$ ($\sim_B,\sim_I$) of size $2$ are called \emph{nodes} (\emph{boundary nodes, internal nodes}). Elements of $\mathcal{CB}$ are called \emph{contracted boundaries}.

We identify $\Sigma$ with the topological space $\left(\bigsqcup_{\alpha}\Sigma_\alpha\right)/\sim.$

The \emph{doubling} $\Sigma_\C$ of $\Sigma$ is the pre-stable marked surface without boundary obtained by gluing $\Sigma$ and an isomorphic copy of it $\overline{\Sigma}$ with an opposite complex structure along the common boundary (using Schwartz reflection), and then taking the topological quotient which identifies a point $z\in\mathcal{CB}\subset\Sigma$ with its partner in $\overline\Sigma.$ The \emph{genus} of a connected pre-stable marked surface with non-empty boundary or with contracted boundaries is defined as the genus of the doubled (pre-stable) surface. If there are no boundaries or contracted boundaries, the genus is taken to be the usual genus of a pre-stable surface.

The \emph{normalization of $\Sigma$} is the pair $(\widehat{\Sigma},q)$, where $\widehat{\Sigma}=\bigsqcup_\alpha\Sigma_\alpha,$ and $q:\widehat{\Sigma}\to\Sigma$ is the quotient map.
$\Sigma$ is stable if each $\Sigma_\alpha$ is stable.
\end{definition}
Note that $\Sigma$ is smooth precisely if all the $\sim-$equivalence classes are of size $1$ and $\mathcal{CB}=\emptyset.$ We will usually omit $\sim_B,\sim_I$ and $\mathcal{CB}$ from the notations.
\begin{definition}
A \emph{stable map} $u:\Sigma\to(\CP^1,\RP^1),$ where $\Sigma=(\{\Sigma_\alpha\},\sim_B,\sim_I,\mathcal{CB})$ is a pre-stable marked surface with boundary, is a map $u$ which satisfies
\begin{enumerate}[(a)]
\item The restrictions $u|_{\Sigma_\alpha}$ are smooth, holomorphic on the interior of $\Sigma_\alpha,$ and map $\partial\Sigma_\alpha$ to $\RP^1.$
\item If $x\sim y$ then $u(x)=u(y);$ equivalently, $u$ descends to a (continuous) map from $\Sigma$ to $\CP^1.$
\item $u(\mathcal{CB})\subset \RP^1.$
\item If some $\Sigma_\alpha$ is mapped to a point, then $\Sigma_\alpha$ is stable.
\end{enumerate}
$u$ is said to be \emph{smooth} if $\Sigma$ is smooth.
It is said to be a \emph{disk map} if $g(\Sigma)=0$ and $\Sigma$ is connected with either $\partial\Sigma\neq\emptyset$ or $\mathcal{CB}\neq \emptyset.$

The \emph{degree} of $u$ is defined as $u_*[\Sigma,\partial\Sigma]\in H_2(\CP^1,\RP^1)\simeq\Z\oplus\Z.$ $u$ is a \emph{disk cover} if its degree is either of the form $(d,0)$ or of the form $(0,d).$
\end{definition}
The notions of isomorphisms and automorphisms between stable maps are the expected ones.

Throughout most of this paper we will consider only disk maps. In addition, we will consider mainly internal markings; boundary markings will appear through normalization of boundary nodes.

\subsection{Moduli space, its corners and the tautological line bundles}

In this section we discuss the moduli spaces of stable disk maps,
the structure of the corners, cotangent line bundles and maps between the
moduli spaces. Many of the constructions appeared in \cite{mod2hom}
and \cite{equiv-OGW-invts}. For those, we will allow ourselves to
be quite brief here, and the reader should consult these references
for more detail.

Consider finite disjoint subsets $\kf,\lf\subseteq\Univ$ and $\vd=\left(d^{+},d^{-}\right)\in\zz_{\geq0}^{2}\subset H_{2}\left(\cc\pp^{1},\rr\pp^{1}\right)$.
Write $\sum\vd=d^{+}+d^{-}$ and $\lf_{\cc}=\kf\sqcup\left(\lf\times\left\{ 1,2\right\} \right)$.
Suppose the stability condition
\begin{equation}
\left|\kf\right|+2\left|\lf\right|+3\sum\vd>2\label{eq:stability}
\end{equation}
holds.
\begin{prop}
\label{prop:The-moduli-space}The moduli space $\overline{\mm}_{0,\kf,\lf}\left(\vd\right)$
is an $S^{1}$-orbifold with corners admitting an $S^{1}$-equivariant
map
\begin{equation}
\overline{\mm}_{0,\kf,\lf}\left(\vd\right)\to\overline{\mm}_{0,\lf_{\cc}}\left(\sum\vd\right).\label{eq:open-closed relation}
\end{equation}
Moreover, $\overline{\mm}_{0,\kf,\lf}\left(\vd\right)$ has simple fixed points.
\end{prop}
\begin{proof}
The moduli map is induced from the doubling map $(\Sigma,u)\to(\Sigma_\C,u_\C)$, and the first statement of the proposition is a special case of \cite[Theorem 1]{mod2hom}
(see Example 3 \emph{ibid}.). Since the induced action on $\rr\pp^{1}\subset\cc\pp^{1}$
has no fixed points, the fixed points of $\overline{\mm}_{0,\kf,\lf}\left(\vd\right)$
have no special points on $\rr\pp^{1}$. In particular, $\partial\overline{\mm}_{0,\kf,\lf}\left(\vd\right)$ has no fixed points.
\end{proof}
Recall the underlying groupoid of $\overline{\mm}_{0,\kf,\lf}\left(\vd\right)$
is equivalent to the groupoid whose objects are \emph{fundamental
configurations }$\sigma=\left(\left(\Sigma,\nu,\lambda,w\right),b,\Sigma^{1/2}\right)$.
Here $\left(\Sigma,\nu,\lambda,w\right)$ represents a point in $\overline{\mm}_{0,\lf_{\cc}}\left(\sum\vd\right)$, so
\begin{enumerate}[\textbullet]
\item $\Sigma$ is a disjoint union of $\cc\pp^{1}$'s,
\item $\nu:\Sigma\to\Sigma$ is an involution whose size 2 orbits are \emph{the}
\emph{nodes},
\item $\lambda:\lf_{\cc}\to\Sigma^{\nu}$ is an injective map which specifies
the position of the marked points, and
\item $w:\Sigma\to X$ is a holomorphic map with $w_{*}\left(\left[\Sigma,\partial\Sigma\right]\right)=\sum\vd$.
\end{enumerate}
$b:\Sigma\to\Sigma$ is an anti-holomorphic involution, so that $\left(\left(\Sigma,\nu,\lambda,w\right),b\right)$
represents a point in the fixed-point stack $\overline{\mm}_{0,\lf_{\cc}}\left(\sum\vd\right)^{\zz/2}$
for a suitable $\zz/2$ action. Finally, $\Sigma^{1/2}\subset\Sigma$
is a fundamental domain for the action of $b$, subject to some conditions,
so $\Sigma^{1/2}/\nu$ is the domain of the stable disk-map.

Let us discuss the structure of the corners. Consider some fundamental
configuration $\sigma=\left(\left(\Sigma,\nu,\lambda,w\right),b,\Sigma^{1/2}\right)$
representing a point $\left[\sigma\right]\in\overline{\mm}_{0,\kf,\lf}$.
A \emph{real node }of $\sigma$ is an orbit $o=\left\{ o_{1},o_{2}\right\} $
of the involution $\nu$ such that $b\left(o\right)=o$. We call a
node \emph{standard }if $b\left(o_{1}\right)=o_{1}$ and \emph{exceptional
}otherwise (these are called E-type and H-type nodes in \cite{mod2hom}, following \cite{Liu02}).
Let $\mathcal{R}\left(\sigma\right)$ denote the set of real nodes
of $\sigma$. It is not hard to see that $\mathcal{R}\left(\sigma\right)$
either consists of a single exceptional node, or some number of standard
nodes.

The set $\mathcal{R}\left(\sigma\right)$ is naturally identified
with the set of local boundary components at $\left[\sigma\right]\in\overline{\mm}_{0,\kf,\lf}\left(\vd\right)$,
so a point $q\in\partial^{c}\overline{\mm}_{0,\kf,\lf}\left(\vd\right)$
is represented by $\left(\sigma,\rho\right)$ where $\sigma$ is a
fundamental configuration, and $\rho:\left\{ 1,...,c\right\} \hookrightarrow\mathcal{R}\left(\sigma\right)$
is an injective map. Our numbering convention is such that $\left(\sigma,\rho|_{\left\{ 2,...,c\right\} }\right)$
represents the image of $q$ in $\partial^{c-1}\overline{\mm}_{0,\kf,\lf}\left(\vd\right)$.

Consider $\rho(c)$. If $\rho(c)$ is standard, we can write
\[
\left(\left(\Sigma,\nu,\lambda,w\right),b,\Sigma^{1/2}\right)=\left(\left(\Sigma_{1}\sqcup\Sigma_{2},\nu_{1}\sqcup\nu_{2},\lambda_{1}\sqcup\lambda_{2},w_{1}\sqcup w_{2}\right),b_{1}\sqcup b_{2},\Sigma_{1}^{1/2}\cup\Sigma_{2}^{1/2}\right)
\]
so that
\[
\sigma_{i}=\left(\left(\Sigma_{i},\nu_{i},\lambda_{i},w_{i}\right),b_{i},\Sigma_{i}^{1/2}\right)
\]
specifies a point $p_{i}\in\overline{\mm}_{0,\kf_{i},\lf_{i}}\left(\vd_{i}\right)$
for $i=1,2$. This decomposition is unique up to a $\zz/2$ action
which swaps the labels $i=1$ and $i=2$. We have $\mathcal{R}\left(\sigma\right)=\mathcal{R}\left(\sigma_{1}\right)\sqcup\mathcal{R}\left(\sigma_{2}\right)\sqcup\left\{ \rho(c)\right\} $.
Set $\kappa_{i}=\rho^{-1}\left(\mathcal{R}\left(\sigma_{i}\right)\right)$,
$c_{i}=\left|\kappa_{i}\right|$, and let $\rho_{i}=\rho|_{\kappa_{i}}\circ\gamma_{i}$
where $\gamma_{i}:\left\{ 1,...,c_{i}\right\} \to\kappa_{i}$ is the
unique order-preserving bijection. We can interpret $\left(\sigma_{i},\rho_{i}\right)$
as specifying a point of $\partial^{c_{i}}\overline{\mm}_{0,\kf_{i},\lf_{i}}\left(\vd_{i}\right)$.
Note that in order to reconstruct $\sigma$, we must keep track of
the partition $\left\{ 1,...,c-1\right\} =\kappa_{1}\sqcup\kappa_{2}$;
henceforth if $\kappa\subset\N$ is any finite subset, we write $\partial^{\kappa}\xx$
for a copy of $\partial^{c}\xx$, $c=\left|\kappa\right|$, where
the local boundary components are labeled by $\kappa$ instead of
$\left[c\right]$. This allows us to write
\[
\partial^{c}\left(\xx\times\yy\right)=\bigsqcup\partial^{\kappa_{1}}\xx\times\partial^{\kappa_{2}}\yy,
\]
where the disjoint union ranges over partitions $\kappa_{1}\sqcup\kappa_{2}=\left[c\right]$,
which specify how to merge the local boundary components of $\xx$
and of $\yy$ to form a corner of the product. A similar Leibnitz
rule holds for transverse fibered product of orbifolds over the manifold
without boundary $L=\rr\pp^{1}$ (cf. \cite[Proposition 6.7]{joyce-fibered}).

If $\rho(c)$ is an exceptional node, then $\Sigma^{b}=\emptyset$
and $\Sigma=\Sigma^{1/2}\sqcup b\left(\Sigma^{1/2}\right)$, $\kf=\emptyset$
and $d_{+}=d_{-}$.

\begin{prop}
\label{prop:corners of moduli}We have $\partial^{c}\overline{\mm}_{0,\kf,\lf}\left(\vd\right)=\cl\sqcup\ee$. \\
The \emph{standard corners}
\[
\cl=\cl^{c}\left(\kf,\lf,\vd\right)=\left(\bigsqcup\partial^{\kappa_{1}}\overline{\mm}_{0,\kf_{1}\sqcup\sstar_{1},\lf_{1}}\left(\vd_{1}\right)\times_{L}\partial^{\kappa_{2}}\overline{\mm}_{0,\kf_{2}\sqcup\sstar_{2},\lf_{2}}\left(\vd_{2}\right)\right)_{\zz/2}
\]
are the $\zz/2$ quotient of the disjoint union over $\kf_{1}\sqcup\kf_{2}=\kf$,
$\lf_{1}\sqcup\lf_{2}=\lf$, $\vd_{1}+\vd_{2}=\vd$ and
$\kappa_{1}\sqcup\kappa_{2}=\left\{ 1,...,c-1\right\} $ such that
$\left(k_{i}+1\right)+2l_{i}+3\sum\vd_{i}>2$ for $i=1,2$, where
$k_{i}=\left|\kf_{i}\right|$ and $l_{i}=\left|\lf_{i}\right|$. The
fibered product is over the evaluation maps $\ev_{\sstar_{1}},\ev_{\sstar_{2}}$
to $L=\rr\pp^{1}$. $\zz/2$ acts by swapping the two fibered factors.\\
The \emph{exceptional boundary} is
\[
\ee=\begin{cases}
\ev_{\star}^{-1}\left(L\right)\subset\overline{\mm}_{0,\lf\sqcup\star}\left(d_{+}\right), & \text{if }d_{+}=d_{-},\kf=\emptyset,\text{ and }c=1,\\
\emptyset, & \text{otherwise},
\end{cases}
\]
where $\overline{\mm}_{0,\lf\sqcup\star}\left(d_{+}\right)$ is the moduli space of closed stable maps.
\end{prop}
\begin{proof}
The discussion preceding the proposition defines a bijection of groupoids
of sets. Using the real gluing maps this is seen to be an equivalence
of orbifolds.
\end{proof}
The standard corners $\cl^{c}\left(\kf,\lf,\vd\right)$ correspond
to the familiar disk bubbling in Floer theory. The exceptional boundary
$\ee$ represents points where the boundary of a disk with no special
points shrinks to a point on $\rr\pp^{1}$. See also the illustrations in \ref{fig:moduli}.

The product of symmetric groups $\Sym\left(\kf\right)\times\Sym\left(\lf\right)$
acts on $\overline{\mm}_{0,\kf,\lf}\left(\vd\right)$ by permuting
the labels, and there is an $S^{1}\times\Sym\left(\kf\right)\times\Sym\left(\lf\right)$
equivariant evaluation map $\overline{\mm}_{0,\kf,\lf}\left(\vd\right)\to L^{\kf}\times X^{\lf}$.

If $\overline{\mm}_{\kf_{i}\sqcup\sstar_{i},\lf_{i}}\left(\vd_{i}\right)$
are moduli spaces of disks, $i=1,2$, we have a gluing map
\[
\overline{\mm}_{\kf_{1}\sqcup\sstar_{1},\lf_{1}}\left(\vd_{1}\right)\fibp{\ev_{\sstar_{1}}}{\ev_{\sstar_{2}}}\overline{\mm}_{\kf_{2}\sqcup\sstar_{2},\lf_{2}}\left(\vd_{2}\right)\to\overline{\mm}_{\kf_{1}\sqcup\kf_{2},\lf_{1}\sqcup\lf_{2}}\left(\vd_{1}+\vd_{2}\right)
\]
obtained by taking the $\zz/2$-invariants of the corresponding closed gluing map. Gluing maps at a complex node can be defined similarly.

Suppose \eqref{eq:stability} holds. We have a forgetful map
\[
\overline{\mm}_{0,\kf\sqcup\left\{ y\right\} ,\lf}\left(\vd\right)\xrightarrow{f_{y}}\overline{\mm}_{0,\kf,\lf}\left(\vd\right)
\]
corresponding to forgetting a boundary marked point.
\begin{rem}
In case $\kf=\emptyset$ and $\vd=\left(d,d\right)$ this map is only
\emph{weakly smooth} over points of $\overline{\mm}_{0,\kf,\lf}\left(\vd\right)$
with an exceptional node, see the proof of Lemma 8 in \cite{mod2hom}.
It is straightforward to extend the category of orbifolds with corners
as defined there to allow such maps. Some fibered products (cf. Lemma
26 \emph{ibid.}) may only satisfy the universal property with respect
to the smaller subcategory of smooth maps, but this is sufficient
for our needs. In particular, if $\sstar'\in\kf'$ and $\sstar''\in\kf''$,
we can use the universal property to define maps such as
\[
\overline{\mm}_{0,\kf'\sqcup\left\{ y\right\} ,\lf'}\left(\vd'\right)\fibp{\ev_{\sstar'}}{\ev_{\sstar''}}\overline{\mm}_{0,\kf'',\lf''}\left(\vd''\right)\xrightarrow{\left(f_{y},\id\right)}\overline{\mm}_{0,\kf',\lf'}\left(\vd'\right)\fibp{\ev_{\sstar'}}{\ev_{\sstar''}}\overline{\mm}_{0,\kf'',\lf''}\left(\vd''\right),
\]
since $\kf'\neq\emptyset$ in this case.
\end{rem}
The usual relations (e.g., associativity of gluing, and the various
compatibilities of the evaluation maps, forgetful maps, gluing maps,
and the group actions) are shown to hold using the corresponding results
for the closed curves.

For $i\in\lf$ we have a cotangent line bundle $\mathbb{L}_{i}=\mathbb{L}_{i}^{\kf,\lf,\vd}$
on $\overline{\mm}_{0,\kf,\lf}\left(\vd\right)$ which is the pullback
of the cotangent line bundle $\mathbb{L}_{i\times\left\{ 1\right\} }^{\lf_{\cc},\sum\vd}$
on $\overline{\mm}_{0,\lf_{\cc}}\left(\sum\vd\right)$ along \eqref{eq:open-closed relation}.

\begin{lem}
\label{lem:d-forget cartesian}Suppose the stability condition \eqref{eq:stability} holds. Then the map
\[
df_{j}:\left(\mathbb{L}_{i}^{\kf\sqcup\left\{ j\right\} ,\lf,\vd}\right)^{\vee}\to\left(\mathbb{L}_{i}^{\kf,\lf,\vd}\right)^{\vee}
\]
is \emph{cartesian}, that is, it induces an isomorphism $\left(\mathbb{L}_{i}^{\kf\sqcup\left\{ j\right\} ,\lf,\vd}\right)^{\vee}\simeq f_j^* \left(\mathbb{L}_{i}^{\kf,\lf,\vd}\right)^{\vee}$.
\end{lem}
\begin{proof}
Consider some fundamental configuration $\sigma=\left(\left(\Sigma,\nu,\lambda,w\right),b,\Sigma^{1/2}\right)$ representing a point of $\overline{\mm}_{0,\kf\sqcup\left\{ j\right\} ,\lf}\left(\vd\right)$.
Forgetting the point $j$ never makes an irreducible component of
$\Sigma$ which contains $i\times\left\{ 1\right\} $ unstable (if
an irreducible component contains both $j$ and $i\times\left\{ 1\right\} $,
it must also contain $i\times\left\{ 2\right\} $, and either have
positive degree or contain at least one other special point). The
claim follows.
\end{proof}
We shall sometimes omit the superscript $\kf,\lf,\vd$ from the notation of $\mathbb{L}_{i}.$

\subsection{Coherent integrands}\label{sub:coherent integrands}
Let $\alpha\in\aso\left(\xx\right)$
satisfy $D\alpha=0$.
Because $\partial\overline{\mm}_{0,\kf,\lf}\left(\vd\right)\neq\emptyset$,
in order to define
\[
\int_{\overline{\mm}_{0,\emptyset,\lf}\left(\vd\right)}\alpha
\]
we must introduce boundary conditions on $\alpha$. In fact, we will
use the recursive structure of the boundary of the moduli spaces to
impose a more detailed condition which will place additional constraints
on $\alpha$ at the corners. This is done for two reasons. First,
it allows us to construct $\alpha$ recursively. Second, we need this
detailed condition to compute the contributions of the corners $\partial^{c}\overline{\mm}_{0,\emptyset,\lf}\left(\vd\right)$
in Proposition \ref{prop:simple fp formula} below. In this section we
formulate these conditions precisely, and in the next two sections we will use
this to define the open descendent integrals.

Consider disk specifications $\left(\lf_{1},\vd_{1}\right)$, $\left(\lf_{2},\vd_{2}\right)$
and $\left(\lf,\vd\right)$ in $\mathcal{D}$. The equation
\begin{equation}
\left(\lf_{1},\vd_{1}\right)\#\left(\lf_{2},\vd_{2}\right)=\left(\lf,\vd\right)\label{eq:gluing specifications}
\end{equation}
will be taken to mean that $\lf_{1}\sqcup\lf_{2}=\lf$ and $\vd_{1}+\vd_{2}=\vd$.
In this case we have a boundary component $\bb\subset\partial\overline{\mm}_{0,\emptyset,\lf}\left(\vd\right)$
given by
\begin{equation}
\bb=\bb\left(\lf_{1},\vd_{1},\lf_{2},\vd_{2}\right)=\overline{\mm}_{0,\sstar_{1},\lf_{1}}\left(\vd_{1}\right)\times_{L}\overline{\mm}_{0,\sstar_{2},\lf_{2}}\left(\vd_{2}\right).\label{eq:bdry comp}
\end{equation}
For $\left(\lf,\left(d,d\right)\right)\in\mathcal{D}$ we also have
the exceptional boundary component
\[
\ee=\ee\left(\lf,\left(d,d\right)\right)\subset\partial\overline{\mm}_{0,\emptyset,\lf}\left(\left(d,d\right)\right).
\]
We write $\bb\xrightarrow{i_{\bb}}\overline{\mm}_{0,\emptyset,\lf}\left(\vd\right)$
and $\ee\xrightarrow{i_{\ee}}\overline{\mm}_{0,\emptyset,\lf}\left(\left(d,d\right)\right)$
for the maps that forget the local boundary component. We may write
$E|_{\bb},E|_{\ee}$ instead of $i_{\bb}^{*}E,i_{\ee}^{*}E$.
\begin{defn}
\label{def:coherent integrand}A \emph{coherent integrand }is a collection
of equivariant forms
\[
\left\{ \alpha_{\lf,\vd}\in\aso\left(\overline{\mm}_{0,\emptyset,\lf}\left(\vd\right)\right)\right\} _{\left(\lf,\vd\right)\in\mathcal{D}}\sqcup\left\{ \beta_{\lf,d}\in\aso\left(\overline{\mm}_{0,\lf}\left(d\right)\right)\right\} _{\left(\lf,d\right)\in\mathcal{S}}.
\]
We assume the degrees of all forms is even, and that the following
conditions hold.
\begin{enumerate}[(a)]
\item\label{it:1} $D\alpha_{\lf,\vd}=0,D\beta_{\lf,d}=0$.
\item\label{it:2} $\alpha_{\lf,\vec{0}}=0,\beta_{\lf,0}=0$.
\item\label{it:3} For any $\left(\lf_{1},\vd_{1}\right),\left(\lf_{2},\vd_{2}\right)$
and $\left(\lf,\vd\right)$ satisfying \eqref{eq:gluing specifications}
we have
\begin{equation}\label{eq:bcond}
i_{\bb}^{*}\alpha_{\lf,\vd}=\left
(\Prr_{1}'\right)^{*}\alpha_{\lf_{1},\vd_{1}}\wedge\left(\Prr_{2}'\right)^{*}\alpha_{\lf_{2},\vd_{2}},
\end{equation}
for $\bb=\bb\left(\lf_{1},\vd_{1},\lf_{2},\vd_{2}\right)$ and where
$\Prr_{i}'$ denotes the composition
\[
\bb\xrightarrow{\Prr_{i}}\overline{\mm}_{0,\sstar_{i},\lf_{i}}\left(\vd_{i}\right)\to\overline{\mm}_{0,\emptyset,\lf_{i}}\left(\vd_{i}\right)
\]
of the projection with the forgetful map, if it is defined. In case
the $\overline{\mm}_{0,\emptyset,\lf_{i}}\left(\vd_{i}\right)$ is
unstable, we set the right-hand side of \eqref{eq:bcond} to be zero.
\item\label{it:4} For any $\left(\lf,\left(d,d\right)\right)\in\mathcal{D}$ and $\ee\subset\partial\overline{\mm}_{0,\emptyset,\lf}\left(\left(d,d\right)\right)$
the exceptional boundary component, we have a natural identification
\[
\ee=\ev_{\star}^{-1}\left(L\right)\subset\overline{\mm}_{0,\lf\sqcup\star}\left(d\right)
\]
and we require
\begin{equation}\label{eq:alpha-beta coherence}
\alpha_{\lf,\vd}|_{\ee}=f^{*}\beta_{\lf,d},
\end{equation}
where $f$ is the composition
\[
\ev_{\star}^{-1}\left(L\right)\hookrightarrow\overline{\mm}_{0,\lf\sqcup\star}\left(d\right)\xrightarrow{}\overline{\mm}_{0,\lf}\left(d\right)
\]
of the inclusion followed by the forgetful map.
\end{enumerate}
%A collection $\{\alpha_{\lf',\vd}\}_{(\lf',d)\in\mathcal{D}(\subseteq\lf)},\{\beta_{\lf',d}\}_{(\lf',d)\in\mathcal{S}(\subseteq\lf)}$  ($\mathcal{D}(\subseteq\lf),\mathcal{S}(\subseteq\lf)$ were defined in the end of Subsection \ref{subsec:mod_spec}) which satisfies the above requirements, will be called a \emph{coherent integrand bounded by} $\lf.$
\end{defn}

The map
\[
\bigsqcup\overline{\mm}_{0,\sstar_{1},\lf_{1}}\left(\vd_{1}\right)\times_{L}\overline{\mm}_{0,\sstar_{2},\lf_{2}}\left(\vd_{2}\right)\to\cl^{1}\left(\emptyset,\lf,\vd\right)
\]
is $2:1$; the right hand side of \eqref{eq:bcond} defines a form
in the image of the pullback map since it is invariant under the $\zz/2$
action that swaps the two moduli factors. Here we use the assumption
that $\deg1\alpha_{\lf,d}$ is even).

\begin{rem}
There is an obvious non-equivariant version of these conditions; the
equivariant forms are replaced by ordinary differential forms; conditions
\eqref{it:1}-\eqref{it:4} are formally the same (with $d$ used in place of $D$ in
\eqref{it:1}).
\end{rem}

\subsection{Coherent integrands coming from the tautological line bundles}\label{subsec:coherent_from_psi}
The main step towards the definition of the stationary
descendent integrals \eqref{integrand1st} is constructing
a suitable coherent integrand. First, we recall the construction of
the equivariant Chern form, following Atiyah-Bott \cite[\S 8]{AB},
focusing on the necessary changes needed to accommodate orbifolds.

Throughout this section we fix vectors $\vec{a},\vec{\epsilon}$ as in Notation \ref{nn:a_eps}, so we may omit them from the notation.

Let $\mm$ be an $S^{1}$-orbifold, and let $E\to\mm$ be an equivariant
complex line bundle. Let $\Gamma\left(E\right)$ denote the \emph{sheaf}
of sections of $E$. The $S^{1}$ action defines a map of sheaves
\[
X_{E}:\Gamma\left(E\right)\to\Gamma\left(E\right).
\]
A connection is specified by a map of sheaves
\[
\nabla:\Gamma\left(E\right)\to\Gamma\left(E\otimes T^{*}\mm\right),
\]
and we say the connection is \emph{equivariant} if $\nabla\circ X_{E}=X_{E\otimes T^{*}\mm}\circ\nabla$.
If $\nabla$ is equivariant, the \emph{equivariant Chern form }$c_{1}\left(E,\nabla\right)\in\aso\left(\mm\right)$
can be defined locally using a generating section (but is independent
of the choice of such a section), cf. \cite[equation~(8.8)]{AB}. We have
$Dc_{1}\left(E,\nabla\right)=0$.

We let $\rho_{\pm}\in\aso\left(\cc\pp^{1}\right)$ be equivariant
forms representing the Poincar\'e dual to $p_{\pm}$. We assume that
the support of $\rho_{\pm}$ is contained in the connected component
of $\cc\pp^{1}\backslash\rr\pp^{1}$ containing $p_{\pm}$.

For every $\left(\lf,d\right)\in\mathcal{S}$ and every $i\in\lf$,
we fix once and for all an $S^{1}$-equivariant connection $\nabla_{\text{cl},i}^{\lf,d}$
for the cotangent line bundle $\mathbb{L}_{\text{cl},i}^{\lf,d}$ on $\overline{\mm}_{0,\lf}\left(d\right)$.
Set
\begin{align}
&\psi_{\text{cl},i}^{\lf,d}=c_{1}\left(\mathbb{L}_{\text{cl},i}^{\lf,d},\nabla_{\text{cl},i}^{\lf,d}\right),\notag\\
&\beta_{\lf,d}=\begin{cases}\beta_{\lf,d}^{\vec{a},\vec{\epsilon}}=\prod_{i\in\lf}\left(\psi_{\text{cl},i}^{\lf,d}\right)^{a_i}\ev_{i}^{*}\rho_{\epsilon_i},&\text{if $d\neq 0$},\\
\beta_{\lf,0}=0,& \text{otherwise}.
\end{cases}
\label{eq:beta integrand def}
\end{align}
\begin{prop}
\label{prop:integrand constn}There exist $S^{1}$-equivariant connections
${\left\{ \nabla_{i}^{\lf,\vd}\right\} _{\left(\lf,\vd\right)\in\mathcal{D},i\in\lf}}$
for the complex line bundles $\mathbb{L}_{i}^{\lf,\vd}$ on $\overline{\mm}_{0,\emptyset,\lf}\left(\vd\right)$
such that, if we define
\begin{align}
&\psi_{i}^{\lf,\vd}=c_{1}\left(\mathbb{L}_{i}^{\lf,\vd},\nabla_{i}^{\lf,\vd}\right),\notag\\
&\alpha_{\lf,\vd}=\alpha_{\lf,\vd}^{\vec{a},\vec{\epsilon}}=\prod_{i\in\lf}\left(\psi_{i}^{\lf,\vd}\right)^{a_i}\ev_{i}^{*}\rho_{\epsilon_i},\label{eq:integrand defn}
\end{align}
then the collection $\left\{ \alpha_{\lf,\vd}\right\} ,\left\{ \beta_{\lf,d}\right\} $
is a coherent integrand.
\end{prop}
The proof of this proposition occupies the remainder of this section.

Let $\xx$ be an orbifold with corners. We denote by $i:\partial\xx\to\xx$
and by $i_{j}:\partial^{2}\xx\to\partial\xx$, for $j=1,2$, the structure
maps forgetting a local boundary component.
\begin{lem}
\label{lem:extending connections}Let $E$ be an $S^{1}$-equivariant
complex line over an $S^{1}$-orbifold with corners $\xx$. Let $K\subset\partial\xx$
be a clopen (closed and open) component of the boundary, and let $K_{2}=i_{1}^{-1}K\cap i_{2}^{-1}K$.
Write $i^{K}:K\to\xx$ and $i_{j}^{K}:K_{2}\to K$ for the restrictions
of $i,i_{j}$. Let $\nabla_{K}$ be an $S^{1}$-equivariant connection on $\left(i^{K}\right)^{*}E$,
such that $\left(i_{1}^{K}\right)^{*}\nabla_{K}=\left(i_{2}^{K}\right)^{*}\nabla_{K}$.
Then there exists an $S^{1}$-equivariant connection $\nabla$ with
$\left(i^{K}\right)^{*}\nabla=\nabla_{K}$.
\end{lem}
\begin{proof}
Fix a reference $S^{1}$-equivariant connection $\nabla_{0}$ for
$E$, so that ${\nabla_{K}-\left(i^{K}\right)^{*}\nabla_{0}}$ corresponds
to a 1-form on $K$, and the problem becomes: given an $S^{1}$-invariant
1-form $\omega_{K}\in\Omega^{1}\left(K;\cc\right)^{S^{1}}$ such that
$\left(i_{1}^{K}\right)^{*}\omega_{K}=\left(i_{2}^{K}\right)^{*}\omega_{K}$,
find an $S^{1}$-invariant 1-form $\omega\in\Omega^{1}\left(\xx;\cc\right)^{S^{1}}$
with $\left(i^{K}\right)^{*}\omega=\omega_{K}$. This is an immediate
generalization of \cite[Lemma 82]{fp-loc-OGW} (there, the claim is
proved for the case $K=\partial\xx$, but the proof extends \emph{mutatis
mutandis} to $K\subset\partial\xx$ clopen).
\end{proof}
Fix some $i\in\mathbb{N}$. We denote
\[
\mathcal{D}\left(i\right)=\left\{ \left(\lf,\vd\right)\in\mathcal{D}|i\in\lf,\vd\neq\left(0,0\right)\right\} .
\]
The equation

\begin{equation}
\left(\lf,\vd\right)=\left(\vd_{1},\lf_{1}\right)\#_{i}\left(\vd_{2},\lf_{2}\right)\label{eq:i-guling}
\end{equation}
will be taken to mean $\left(\lf,\vd\right)=\left(\vd_{1},\lf_{1}\right)\#\left(\vd_{2},\lf_{2}\right)$,
$\vd_{1},\vd_{2}\neq\left(0,0\right)$ and $i\in\lf_{1}$. Take a triple $\left(\lf,\vd\right),\left(\vd_{1},\lf_{1}\right),\left(\vd_{2},\lf_{2}\right)$ which satisfies \eqref{eq:i-guling}, and consider
\begin{equation}
\bb=\overline{\mm}_{0,\sstar_{1},\lf_{1}}\left(\vd_{1}\right)\times_{L}\overline{\mm}_{0,\sstar_{2},\lf_{2}}\left(\vd_{2}\right)\xrightarrow{\Prr_{1}'}\overline{\mm}_{0,\emptyset,\lf_{1}}\left(\vd_{1}\right).\label{eq:known boundary}
\end{equation}
It follows from Lemma \ref{lem:d-forget cartesian} that $d\Prr_{1}'$ induces an
isomorphism
\begin{equation}
\mathbb{L}_{i}^{\lf,\vd}|_{\bb}\simeq\left(\Prr_{1}'\right)^{*}\mathbb{L}_{i}^{\lf_{1},\vd_{1}}.\label{eq:coherence for L_i}
\end{equation}

Next we consider the exceptional boundary $\ee\subset\partial\overline{\mm}_{0,\lf}\left(\left(d,d\right)\right)$.
We fix some precompact open subset
\[
\operatorname{supp}\rho_{\pm}\subset U_{\pm}\Subset\CP^1\backslash\RP^{1}
\]
containing the support of $\rho_{\pm}$. Write $\ee'=\ee\cap\ev_{i}^{-1}\left(U_{\epsilon_i}\right)$,
and let $\ee'\xrightarrow{f'}\overline{\mm}_{0,\lf}\left(d\right)$
denote the pullback of $\ev_{\star}^{-1}\left(L\right)\xrightarrow{f}\overline{\mm}_{0,\lf}\left(d\right)$
(see Definition \ref{def:coherent integrand}).
\begin{obs}
$df'$ induces an isomorphism %\ref{eq:exceptional coherence}
\begin{equation}
\mathbb{L}_{i}^{\lf,\vd}|_{\ee'}\simeq\left(f'\right)^{*}\mathbb{L}_{\text{cl},i}^{\lf,d}.\label{eq:L-exceptional coherence}
\end{equation}
\end{obs}
\begin{proof}
The image of $\ee'$ in $\overline{\mm}_{0,\lf}\left(d\right)$ does
not contain any configurations where $i,\star$ are on a degree zero
component.
\end{proof}
\begin{defn}
A collection of connections $\nabla_{i}^{\lf,\vd}$ for $\mathbb{L}_{i}^{\lf,\vd}$,
defined for all $\left(\lf,\vd\right)$ in some subset $\mathcal{D}'\subset\mathcal{D}\left(i\right)$,
will be called \emph{coherent }if:
\begin{enumerate}[(a)]
\item For all $\left(\lf,\vd\right),\left(\lf_{1},\vd_{1}\right)\in\mathcal{D}'$
and $\left(\lf_{2},\vd_{2}\right)\in\mathcal{D}$ which satisfy \eqref{eq:i-guling},
we have
\begin{equation}
\left(i^{*}\nabla_{i}^{\lf,\vd}\right)|_{\bb}=\left(\Prr_{1}'\right)^{*}\nabla_{i}^{\lf_{1},\vd_{1}}.\label{eq:coherence for connection}
\end{equation}
\item For all $\left(\lf,\left(d,d\right)\right)\in\mathcal{D}'$, we have
\begin{equation}
\nabla_{i}^{\lf,\left(d,d\right)}|_{\ee'}=\left(f'\right)^{*}\nabla_{\text{cl},i}^{\lf,d}.\label{eq:exceptional coherence}
\end{equation}
\end{enumerate}
\end{defn}

\begin{lem}\label{lem:inductive_construction_of_connection}
There exists a coherent collection of connections $\left\{ \nabla_{i}^{\lf,\vd}\right\} $
defined for all $\left(\lf,\vd\right)\in\mathcal{D}\left(i\right)$.
\end{lem}
\begin{proof}
Fix some linear order on the set $\mathcal{D}\left(i\right)$ with the property
that \eqref{eq:i-guling} implies $\left(\lf_{1},\vd_{1}\right)<\left(\lf,\vd\right)$.
For $\left(\lf,\vd\right)\in\mathcal{D}\left(i\right)$ we denote the corresponding
prefix of $\mathcal{D}\left(i\right)$ by
\[
\mathcal{D}_{\leq\left(\lf,\vd\right)}=\left\{ \left(\lf_{1},\vd_{1}\right)\in\mathcal{D}\left(i\right)|\left(\lf_{1},\vd_{1}\right)\leq\left(\lf,\vd\right)\right\} .
\]
We define $\mathcal{D}_{<\left(\lf,\vd\right)}$ similarly. We construct,
by induction on $(\lf,\vd)\in\mathcal{D}\left(i\right),$
a collection of coherent connections on $\mathcal{D}_{\leq\left(\lf,\vd\right)}.$
Suppose we have a coherent collection on $\mathcal{D}_{<\left(\lf,\vd\right)}$.
Let
\[
K\subset\partial\overline{\mm}_{0,\emptyset,\lf}\left(\vd\right)
\]
be the clopen component of the boundary corresponding to $\bb$ of
the form \eqref{eq:known boundary} for $\left(\lf_{1},\vd_{1}\right),\left(\lf_{2},\vd_{2}\right)$
satisfying \eqref{eq:i-guling}, as well as the exceptional boundary
component (which is non-empty only if $\vd=\left(d,d\right)$). Define
$\nabla_{K}$ on $\left(i^{K}\right)^{*}\mathbb{L}_{i}^{\lf,\vd}$
so that \eqref{eq:coherence for connection} and \eqref{eq:exceptional coherence}
hold (for the latter, this involves using a partition of unity to
extend the pulled-back connection from $\ee'$ to $\ee$).

We check that $\left(i_{1}^{K}\right)^{*}\nabla_{K}=\left(i_{2}^{K}\right)^{*}\nabla_{K}$.
Indeed, $\partial\ee=\emptyset$ so there is nothing to check at the
exceptional boundary. The corner components in $i_{1}^{-1}K\cap i_{2}^{-1}K$
can be written as fibered products
\[
\mm_{1}\times_{L}\mm_{2}\times_{L}\mm_{3}.
\]
By symmetry we may assume the $i^{th}$ marking is on the disk parameterized
by $\mm_{j}$ for $j=1$ or $j=2$. Consider first the case $j=1$.
Let $\mm_{1}^{\DIAMOND}=\overline{\mm}_{0,\emptyset,\lf_{1}}\left(\vd_{1}\right)$
denote the moduli obtained from $\mm_{1}$ by forgetting the boundary
node, and let $\mm_{2}^{\DIAMOND}$ denote the moduli obtained from $\mm_{2}$
by forgetting the boundary node connecting it to $\mm_{3}$. The triangle
\[
\xymatrix{ & \mm_{1}\times_{L}\mm_{2}\times_{L}\mm_{3}\ar[dl]\ar[dr]\\
\mm_{1}\times_{L}\mm_{2}^{\DIAMOND}\ar[rr] &  & \mm_{1}^{\DIAMOND}
}
\]
commutes, and so does its linearization, so the cartesian lifts of
the maps to the cotangent line bundles $\mathbb{L}_{i}$ also form a commutative
triangle. Thus we can use the inductive hypothesis to conclude that
$i_{1}^{*}\nabla_{K}=i_{2}^{*}\nabla_{K}$, since both are equal to
the pullback of $\nabla_{i}^{\lf_{1},\vd_{1}}$ from $\mm_{1}^{\DIAMOND}$.
The case $j=2$ is similar (though the triangle needs to be replaced
by a square). Now apply Lemma \ref{lem:extending connections} to
construct $\nabla_{i}^{\lf,\vd}$ compatible with $\nabla_{K}$, completing
the inductive step and the proof of the lemma.
\end{proof}

\begin{proof}
[Proof of Proposition \ref{prop:integrand constn}.]We define $\alpha_{\lf,\vd}$
by \eqref{eq:integrand defn} applied to the connections constructed in Lemma \ref{lem:inductive_construction_of_connection}. Let us check the conditions of Definition
\ref{def:coherent integrand}. Condition \eqref{it:1} holds since the Chern
form and $\rho_{\pm}$ are $D$-closed. Condition \eqref{it:2} holds for $\beta_{\lf,0}$
by definition, and for $\alpha_{\lf,\left(0,0\right)}$ since $\lf\neq\emptyset$
and $\rho_{\pm}|_{\rr\pp^{1}}=0$. Condition \eqref{it:3} holds by \eqref{eq:coherence for connection} and the fact the evaluation maps commute with $\Prr'_{i}$. Let us check that
condition \eqref{it:4} holds. If $\lf=\emptyset$, then $\alpha_{\lf,\vd}=1$ and $\beta_{\lf,d}=1,$ since \eqref{eq:beta integrand def},\eqref{eq:integrand defn} involve empty products. Therefore \eqref{eq:alpha-beta coherence} holds trivially. Assume $\lf\neq\emptyset.$ By construction,
the evaluation maps commute with \eqref{eq:open-closed relation}.
At
\[
\ee''=\ee\cap\bigcap_{i\in\lf}\left(\ev_{i}^{\lf,\left(d,d\right)}\right)^{-1}\left(U_{\epsilon_i}\right)
\]
\eqref{eq:alpha-beta coherence} holds by \eqref{eq:exceptional coherence},
while at $\ee\backslash\ee''$ both sides of \eqref{eq:alpha-beta coherence}
vanish.
\end{proof}

\subsection{Stationary descendent integrals}
\begin{defn}\label{def:int_bracket}
Fix $\vec{a},\vec{\epsilon}.$ A \emph{coherent integrand} for the tautological line bundles is a coherent integrand $\left\{ \alpha_{\lf,\vd}\right\} ,\left\{ \beta_{\lf,d}\right\} $ defined via \eqref{eq:beta integrand def},\eqref{eq:integrand defn}.

For given $\lf,\vd,\vec{a},\vec{\epsilon},$ we define the \emph{open stationary descendent integral} or \emph{open stationary intersection number} by
\[
\<\prod_{i\in\lf}\tau^{\epsilon_i}_{a_i}\>_{0,\vd}=\int_{\overline{\mm}_{0,\emptyset,\lf}\left(\vd\right)}\alpha^{\vec{a},\vec{\epsilon}}_{\lf,\vd},
\]
where $\alpha^{\vec{a},\vec{\epsilon}}_{\lf,\vd}$ comes from a family of coherent integrands for the tautological line bundles, and $\int_{\overline{\mm}_{0,\emptyset,\lf}\left(\vd\right)}$ is oriented via the canonical orientation defined in Section \ref{sec:or}.
\end{defn}
Definition \ref{def:int_bracket} gives a rigourous definition for the integral on the right-hand side of~\eqref{integrand1st}. The next theorem says that the open stationary intersection number is independent of the specific coherent integrand used in its definition, and gives a formula for this number.

\begin{thm}\label{thm:int_nums_equal_tree_sum}
For given $\lf,\vd,\vec{a},\vec{\epsilon},~\<\prod_{i\in\lf}\tau^{\epsilon_i}_{a_i}\>_{0,\vd}$ is independent of the coherent integrand used to define it.
Moreover, it holds that
\[
\<\prod_{i\in\lf}\tau^{\epsilon_i}_{a_i}\>_{0,\vd}=
\OGW(\lf,\vd,\vec{a},\vec{\epsilon}),
\]
where \[\OGW(\lf,\vd,\vec{a},\vec{\epsilon})=\sum_{T\in\TTT(\lf,\vd)}\AAA(T,\vec{a},\vec{\epsilon})+\delta_{d^+-d^-}\frac{d^+}{2u}I(\lf,d,\vec{a},\vec{\epsilon})\]
was defined in \eqref{eq:nice_tree_sum}, and $\AAA(-,-,-),\TTT(-,-)$ and $I(-,-,-,-)$ are defined in Definitions \ref{def:non_lab_non_ord_trees_and_ampli} and \ref{def:inhomterms}.
In particular, \eqref{eq:nice_tree_sum} vanishes if $1+\sum a_i<d^++d^-$.
\end{thm}

\begin{rmk}\label{rmk:why_not_naive}
Several comments are in place.
\begin{enumerate}[(a)]
\item
First, in the introduction another definition for open intersection numbers was sketched, as the weighted cardinality of the intersection of the zero loci of multisections of $\mathbb{L}_i,$ given some boundary conditions. That approach, which is similar to the approach taken in \cite{PST14,BCT1,BCT2}, can be made fully rigourous by a recursive construction analogous to those given in these papers. The equivalence of that definition with the definition in terms of coherent integrands (in the non-equivariant limit) can be proven along the lines of \cite[Corollary 15]{fp-loc-OGW}.

\item\label{it:graphical}
Second, the first sum in \eqref{eq:nice_tree_sum} can be represented in a less compact, but perhaps more suggestive way, by expanding the expressions \eqref{eq:I_S} in the definition of the amplitude in terms of fixed-point graph contributions \eqref{eq:I_Gamma}. Doing that, we get a sum over tuples $(T,(\Gamma_v)_{v\in V(T)}),$ in which $T\in\TTT$ are trees which have no vertex $v$ with $d^+(v)=d^{-}(v)$ (otherwise the amplitude clearly vanishes), and $\Gamma_v$ are fixed point graphs for the specifications $(\lf_v,\vd_v)$ of the vertices of $T.$ Each tuple can be graphically presented as follows:
\begin{enumerate}[(i)]
\item Replace each vertex $v$ by $\Gamma_v.$ Since $d^+(v)\neq d^-(v),~$the graph $\Gamma_v$ has a single disk edge $H_v$ with
    \[\partial^H(\delta(H_v))=(d^+(v)-d^-(v)),\]
    where $\partial^H$ is the connecting map of \eqref{eq:connecting}.
\item Draw the edge $e=\{u,v\}$ as a wavy edge which connects the boundary vertices of $H_v,~H_u$ and associate this wavy edge $e$ the weight
\[\frac{\partial^H(\delta(H_v))\partial^H(\delta(H_u))}{-2u}=\frac{(d^+(v)-d^-(v))(d^+(u)-d^-(u))}{-2u}.\]
\item Associate the vertex $v$ of $T$ the weight \[\frac{1}{|A_{\Gamma_v}|}\int_{\overline{\mm}_{\Gamma_v}}e_{\Gamma_v}^{-1}\cdot\alpha_{\Gamma_v}^{\vec{a},\vec{\epsilon}}.\]
\item Take the product of weights of vertices and wavy edges, and divide by $|\Aut(T)|$.
\end{enumerate}

Similarly, the second term in \eqref{eq:nice_tree_sum} can be graphically presented as a sum over all isomorphism types of pairs $(\Gamma,e)$, where $\Gamma$ is a fixed point graph for $(\lf,d),$ the sphere moduli specifications obtained from shrinking the boundary of a disk map of type $(\lf,\vd),$ and $e$ is a sphere edge of $\Gamma.$ An isomorphism between $(\Gamma,e)$ to $(\Gamma',e')$ is an isomorphism between $\Gamma$ and $\Gamma'$ which takes $e$ to $e'.$  For such a pair one draws a wavy half-edge emanating from the equator of $e.$ This wavy half-edge is associated the weight $\frac{-|\delta(e)|}{-2u}$. The contribution of $(\Gamma,e)$ to \eqref{eq:nice_tree_sum} is the product
\[\frac{-|\delta(e)|}{-2u}\frac{1}{|A'_{(\Gamma,e)}|}\int_{\overline{\mm}_{\Gamma}}e_{\Gamma}^{-1}\cdot\alpha_{\Gamma}^{\vec{a},\vec{\epsilon}}\]
of the graph contribution and the wavy half-edge weight
\[|A'_{(\Gamma,e)}|=|A^0_\Gamma||\Aut_{(\Gamma,e)}|,\]
where $\Aut_{(\Gamma,e)}$ is the automorphism group of $\Gamma$ preserving $e.$
See Figure \ref{fig:g_0_cont} for the graphical description described above (where we draw sphere edges of fixed point graphs as spheres and disk edges as disks).

The proof that the two representations are the same is straightforward (see also Lemma \ref{lem:OGW_loc_graphs} for a similar statement in higher genus).

\item
One could have hoped that the simpler formula \eqref{eq:I_S} will give rise to the stationary open descendent integrals. This expression does not correspond to the intersection theory we define, and moreover, there is no geometric intersection theory that would give rise to this expression. Indeed, a formula which calculates an intersection number must vanish in underdetermined cases (when the dimension of the moduli is bigger than the dimension of the constraints). This is not the case with $I(S,\vec{a},\vec{\epsilon}).$
\end{enumerate}
\end{rmk}

\section{Fixed-point localization for coherent integrands}\label{sec:genus0-calc}
In this section we first prove a general localization formula for $S^1-$orbifolds with corners and simple fixed points, Proposition \ref{prop:simple fp formula}. We then prove a localization expression for the integration of a coherent integrand, Theorem \ref{thm:loc for CP1,RP1}.

\subsection{Fixed-point formula for orbifolds with corners}
The following fixed-point formula is the main computational tool that
we will use.
\begin{prop}
\label{prop:simple fp formula}Let $\xx$ be a compact oriented $S^{1}$-orbifold
with simple fixed points, and for $c\geq1$ let $Q_{c}\in\mathcal{A}_{S^{1}}\left(\partial^{c}\xx\right)\left[u^{-1}\right]$
be an equivariant primitive which is invariant under the $\Sym\left(c\right)$
action on $\partial^{c}\xx$, permuting local boundary components.
Then for any $\alpha\in\mathcal{A}_{S^{1}}\left(\xx\right)$ such
that $D\alpha=0$ we have
\[
\int_{\xx}\alpha=\sum_{F\subset\xx^{S^{1}}}\int_{F}\frac{\alpha|_{F}}{e\left(N_{F}\right)}+\sum_{c\geq1}\left(-1\right)^{c-1}\int_{\partial^{c}\xx}Q_{c}\cdots Q_{2}Q_{1}\,\alpha|_{\partial^{c}\xx},
\]
where the first sum ranges over connected components $F\subset\xx^{S^{1}}$.
\end{prop}
\begin{rem}
Note that as a special case of Proposition \ref{prop:simple fp formula} we obtain the well-known Atiyah-Bott fixed-point formula \cite{AB},
\begin{equation}
\int_{\xx}\alpha=\sum\int_{F}\frac{\alpha|_{F}}{e_{S^{1}}\left(N_{F}\right)},\quad\text{if}\quad\partial\xx=\emptyset.\label{eq:classical fp formula}
\end{equation}
There are easy counterexamples where \eqref{eq:classical fp formula} does not hold in case $\partial\xx\neq\emptyset$.
\end{rem}
\begin{rem}
It is possible to obtain a fixed-point formula even when the fixed
points of $\xx$ intersect the boundary, see \cite{fp-loc-OGW}.
\end{rem}
The proof of Proposition \ref{prop:simple fp formula} will occupy
the remainder of this section. Let $\widetilde{\xx}$ be an oriented
orbifold with corners on which $S^{1}$ acts with \emph{no} fixed
points, and suppose $Q_{c}\in\aso\left(\partial^{c}\widetilde{\xx}\right)\left[u^{-1}\right]$
are $\Sym(c)$-invariant primitives as above. We write
$Q_{\leq c}=Q_{c}\cdots Q_{1}$.
\begin{lem}
\label{lem:Q rewriting}For any equivariant form $\alpha$ with $D\alpha=0$ we have
$\int_{\widetilde{\xx}}\alpha=\sum_{c\geq1}\left(-1\right)^{c-1}\int_{\partial^{c}\widetilde{\xx}}Q_{\leq c}\alpha$.
\end{lem}
\begin{proof}
Since $\widetilde{\xx}$ has no fixed points, we may take $Q$ to be an equivariant primitive on $\widetilde{\xx}$. We will show by induction
on $m$ that
\[
\int_{\widetilde{\xx}}\alpha=\sum_{1\leq c\leq m}\left(-1\right)^{c-1}\int_{\partial^{c}\widetilde{\xx}}Q_{\leq c}\alpha+\left(-1\right)^{m-1}\int_{\partial^{m}\widetilde{\xx}}\left(Q-Q_{m}\right)Q_{\leq\left(m-1\right)}\alpha,
\]
from which the result follows by taking $m$ so large that $\partial^{m}\widetilde{\xx}=\emptyset$.
Let us prove the base case $m=1$. By Stokes' theorem, we have
\[
\int_{\widetilde{\xx}}\alpha=\int_{\widetilde{\xx}}D\left(Q\alpha\right)=\int_{\partial\widetilde{\xx}}Q\alpha=\int_{\partial\widetilde{\xx}}Q_{1}\alpha+\int_{\partial\widetilde{\xx}}\left(Q-Q_{1}\right)\alpha.
\]
Suppose the result holds for $m$, let us prove it for $m+1$:
\begin{align*}
&\left(-1\right)^{m-1}\int_{\partial^{m}\widetilde{\xx}}\left(Q-Q_{m}\right)Q_{\leq\left(m-1\right)}\alpha=\left(-1\right)^{m}\int_{\partial^{m}\widetilde{\xx}}D\left(QQ_{m}\right)Q_{\leq\left(m-1\right)}\alpha\overset{\left(\star\right)}{=}\left(-1\right)^{m}\int_{\partial^{m}\widetilde{\xx}}D\left(QQ_{\leq m}\alpha\right)=\\
=&\left(-1\right)^{m}\int_{\partial\partial^{m}\widetilde{\xx}}QQ_{\leq m}\alpha=\left(-1\right)^{m}\int_{\partial^{m+1}\widetilde{\xx}}Q_{m+1}Q_{\leq m}\alpha+\left(-1\right)^{m}\int_{\partial^{m+1}\widetilde{\xx}}\left(Q-Q_{m+1}\right)Q_{\leq m}\alpha.
\end{align*}
To establish the equality marked with $\left(\star\right)$ it suffices
to show that
\[
\int_{\partial^{m}\widetilde{\xx}}QQ_{m}\cdots\widehat{Q_{j}}\cdots Q_{1}\alpha=0
\]
for any $1\leq j\leq m-1$. Consider the orientation-reversing involution
$\sigma:\partial^{m}\widetilde{\xx}\to\partial^{m}\widetilde{\xx}$
which swaps the $j$ and $j+1$ boundary components. We have $\sigma^{*}Q=Q$,
$\sigma^{*}\alpha=\alpha$ and $\sigma^{*}Q_{i}=Q_{i}$ for all $i<j$
since $\sigma$ commutes with the structure map $\partial^{j}\widetilde{\xx}\to\widetilde{\xx}$.
We have $\sigma^{*}Q_{i}=Q_{i}$ for $i\geq j+1$ by the invariance
assumption. In short, we have
\[
\int_{\partial^{m}\widetilde{\xx}}Q_{m}\cdots\widehat{Q_{j}}\cdots Q_{1}Q\alpha=-\int_{\partial^{m}\widetilde{\xx}}\sigma^{*}\left(Q_{m}\cdots\widehat{Q_{j}}\cdots Q_{1}Q\alpha\right)=-\int_{\partial^{m}\widetilde{\xx}}Q_{m}\cdots\widehat{Q_{j}}\cdots Q_{1}Q\alpha,
\]
from which it follows that $\int_{\partial^{m}\widetilde{\xx}}Q_{m}\cdots\widehat{Q_{j}}\cdots Q_{1}Q\alpha=0$. This completes the proof.
\end{proof}
Now suppose that $\xx$ has simple fixed points, and fix some $S^{1}$-invariant
Riemannian metric on $\xx$. For each connected component $F\subset\xx^{S^{1}}$,
let $S\left(N_{F}\right)\xrightarrow{\pi}F$ denote the sphere bundle
associated with the normal bundle $N_{F}$. We take $\epsilon>0$
so small so that the geodesic flow
\[
\bigsqcup_{F\subset\xx^{S^{1}}}\left(S\left(N_{F}\right)\times\left(0,\epsilon\right)\right)\to\xx
\]
is well-defined and injective, and use it to construct the blow up
\[
\widetilde{\xx}=\left(\xx\backslash\xx^{S^{1}}\right)\cup\left(\bigsqcup_{F\subset\xx^{S^{1}}}\left(S\left(N_{F}\right)\times[0,\epsilon)\right)\right)
\]
of the fixed points of $\xx$. It is not hard to see $\widetilde{\xx}$
is a compact oriented orbifold with corners with no fixed points and
that
\[
\partial^{c}\widetilde{\xx}=\begin{cases}
\partial\xx\sqcup\bigsqcup_{F\subset\xx}S\left(N_{F}\right), & \text{if }c=1,\\
\partial^{c}\xx, & \text{if }c\neq 1.
\end{cases}
\]

Note that the induced outward normal orientation on $S\left(N_{F}\right)$
is opposite to the standard orientation induced from that of the total
space $N_{F}$. Thus, Proposition \ref{prop:simple fp formula} follows
from Lemma \ref{lem:Q rewriting} and the following.
\begin{lem}
\label{lem:fp contrib by euler}If $Q$ is any equivariant primitive
on $S\left(N_{F}\right)$, then we have
\begin{equation}
-\int_{S\left(N_{F}\right)}Q\pi^{*}\left(\alpha|_{F}\right)=\int_{F}\frac{\alpha|_{F}}{e_{S^{1}}\left(N_{F}\right)}.\label{eq:fp contrib}
\end{equation}
\end{lem}
\begin{proof}
We need to compute
\[
-\int_{S\left(N_{F}\right)}Q\pi^{*}\left(\alpha|_{F}\right)=\int_{F}\left(-\pi_{*}Q\right)\alpha|_{F}.
\]
Since $\pi$ is an $S^{1}$-equivariant map from an orbifold without
boundary to an orbifold with trivial $S^{1}$ action, we have
\[
\pi_{*}D=D\pi_{*}=d\pi_{*}.
\]
Now if $Q'$ is any other equivariant primitive, then
\[
0=\left(\pi_{*}D-d\pi_{*}\right)\left(QQ'\right)=\pi_{*}\left(Q'\right)-\pi_{*}\left(Q\right)+d\epsilon,
\]
so the cohomology class $\left[-\pi_{*}Q\right]$, and the left hand
side of \eqref{eq:fp contrib}, is independent of $Q$. To complete
the proof, we identify $\left[-\pi_{*}Q\right]$ with the inverse
of the equivariant Euler form. In \cite{twA8} the second author proves the existence
of an \emph{equivariant angular form}
\[
\phi\in\Omega\left(S\left(N\right);\Or\left(N\right)\otimes\rr\left[u\right]\right)^{S^{1}}
\]
for $S\left(N\right)\xrightarrow{\pi}F$. This is a form such that
(a) $D\phi\in\im\left(\pi^{*}\right)$ and (b) $\pi_{*}\phi=1$. By
definition\emph{ }the equivariant Euler form associated with $\phi$
is specified by the equation
\[
D\phi=-\pi^{*}e_{S^{1}}\left(N_{F}\right).
\]
On the other hand,
\[
0=\left(\pi_{*}D-d\pi_{*}\right)\left(Q\phi\right)=1+\left(\pi_{*}Q\right)\cdot e_{S^{1}}\left(N_{F}\right)+d\epsilon,
\]
which shows that $e_{S^{1}}\left(N_{F}\right)$ is invertible (since
$d\epsilon$ is nilpotent), and $\left[-\pi_{*}Q\right]\cdot\left[e_{S^{1}}\left(N_{F}\right)\right]=1$.
The proof of the lemma, hence also of Proposition \ref{prop:simple fp formula}, is now complete.
\end{proof}
\begin{rem}
In the proof of Lemma \ref{lem:fp contrib by euler}, we can take $Q=-\left(\pi^{*}e_{S^{1}}\left(N_{F}\right)\right)^{-1}\phi$
as an equivariant primitive, so $\pi_{*}Q=-e_{S^{1}}\left(N_{F}\right)^{-1}$
and $\epsilon=0$.
\end{rem}

\subsection{Statement of fixed-point formula for coherent integrands}\label{subsec:localization formula}
\begin{defn}\label{def:trees}
Fix a disk moduli specification $\left(\lf,\vd\right)\in\mathcal{D}$.
Write
\[
\mathcal{D}{}_{\neq0}=\left\{ \left(\lf,\vd\right)\in\mathcal{D}|\vd\neq\left(0,0\right)\right\} .
\]
For $r\geq0$ a non-negative integer, an $\left(r,\lf,\vd\right)$\emph{-labeled
tree} $T$ consists of
\begin{enumerate}[(a)]
\item A tree with $r$ oriented edges labeled $1,...,r$. We denote by $V\left(T\right)$
the set of vertices of $T$.
\item A function
\[
\delta:V\left(T\right)\to\mathcal{D}_{\neq0},\qquad\delta\left(v\right)=\left(\lf\left(v\right),\vd\left(v\right)\right)=\left(\lf\left(v\right),\left(d^{+}\left(v\right),d^{-}\left(v\right)\right)\right),
\]
such that $\bigsqcup\lf\left(v\right)=\lf$ and $\sum\vd\left(v\right)=\vd$.
\end{enumerate}
There is an obvious notion of isomorphism of such trees, and we fix
a set of representatives $\hts\left(r,\lf,\vd\right)$ for the isomorphism
classes of $\left(r,\lf,\vd\right)$-labeled trees.
\end{defn}
\begin{nn}\label{nn:amplitude,F,E}
Consider some coherent integrand $\left\{ \alpha_{\lf,\vd}\right\} _{\left(\lf,\vd\right)\in\mathcal{D}},\left\{ \beta_{\lf,d}\right\} _{\left(\lf,d\right)\in\mathcal{S}}$
as in Definition \ref{def:coherent integrand}. For $\left(\lf,\vd\right)\in\mathcal{D}_{\neq0}$
we set
\[
F\left(\lf,\vd\right)=F^{\left\{ \alpha_{\delta}\right\} _{\delta\in\mathcal{D}},\left\{ \beta_{\delta}\right\} _{\delta\in\mathcal{S}}}\left(\lf,\vd\right)
=\sum_{F\subset\overline{\mm}_{0,\emptyset,\lf}\left(\vd\right)}\int_{F}\frac{\alpha_{\lf,\vd}|_{F}}{e^{S^{1}}\left(N_{F}\right)},
\]
where $F$ ranges over connected components of the fixed-point stack
$\overline{\mm}_{0,\emptyset,\lf}\left(\vd\right)^{S^{1}}$, and
\begin{align*}
&E\left(\lf,\vd\right)=E^{\left\{ \alpha_{\delta}\right\} _{\delta\in\mathcal{D}},\left\{ \beta_{\delta}\right\} _{\delta\in\mathcal{S}}}\left(\lf,\vd\right)=\begin{cases}
\frac{-d}{\left(-2u\right)}\int_{\overline{\mm}_{0,\lf}\left(d\right)}\beta_{\lf,d}, & \text{if }\vd=\left(d,d\right),\\
0, & \text{otherwise},
\end{cases}\\
&G\left(\lf,\vd\right)=G^{\left\{ \alpha_{\delta}\right\} _{\delta\in\mathcal{D}},\left\{ \beta_{\delta}\right\} _{\delta\in\mathcal{S}}}\left(\lf,\vd\right)
=F\left(\lf,\vd\right)+E\left(\lf,\vd\right).
\end{align*}

We define the \emph{amplitude }$\AAH\left(T\right)=A^{\left\{ \alpha_{\delta}\right\} _{\delta\in\mathcal{D}},\left\{ \beta_{\delta}\right\} _{\delta\in\mathcal{S}}}\left(T\right)
\in\rr\left[u,u^{-1}\right]$ of a labeled tree $T\in\hts\left(r,\lf,\vd\right)$~by
\[
\AAH\left(T\right)=\frac{1}{2^{r}r!}\left(\frac{-1}{2u}\right)^{r}\prod_{v\in V}\left(d^{+}\left(v\right)-d^{-}\left(v\right)\right)^{\val\left(v\right)}\cdot G\left(\delta\left(v\right)\right).
\]
Observe that if $T$ has more than one vertex, and if $E\left(\delta(v)\right)\neq 0,$ then
$\left(d^{+}\left(v\right)-d^{-}\left(v\right)\right)^{\val\left(v\right)}=0,$ hence also $\AAH\left(T\right)=0.$
Thus, we can also write
\[
\AAH\left(T\right)=
\begin{cases}
\frac{1}{2^{r}r!}\left(\frac{-1}{2u}\right)^{r}\prod_{v\in V}\left(d^{+}\left(v\right)-d^{-}\left(v\right)\right)^{\val\left(v\right)}\cdot F\left(\delta\left(v\right)\right),& \text{if }r>1,\\
G\left(\delta\left(v\right)\right), & \text{if }V(T)=\{v\}.
\end{cases}
\]
%and moreover, if $r>1$ for each vertex $d_{+}(v)-d_{-}(v)\neq0.$
\end{nn}

\begin{thm}
\label{thm:loc for CP1,RP1}
For all $\left(\lf,\vd\right)\in\mathcal{D}$ we have \begin{equation}
\int_{\overline{\mm}_{0,\emptyset,\lf}\left(\vd\right)}\alpha_{\lf,\vd}=\sum_{r\geq0}\sum_{T\in\hts\left(r,\lf,\vd\right)}\AAH\left(T\right).
\label{eq:fp formula}
\end{equation}

\end{thm}
The proof of this theorem appears at the end of $\S$\ref{subsec:Solution-to-Recursion}
below.

\subsection{Equivariant primitives}

To apply Proposition \ref{prop:simple fp formula} we need to construct
equivariant primitives $Q_{c}^{\lf,\vd}\in\aso\left(\partial^{c}\overline{\mm}_{0,\kf,\lf}\left(\vd\right)\right)\left[u^{-1}\right]$.

If $\vd=\left(d,d\right)$, $\kf=\emptyset$, and $\ee\subset\overline{\mm}_{0,\lf\sqcup\star}\left(d\right)$
denotes the exceptional boundary component, we take
\[
Q_{1}^{\lf,\vd}|_{\ee}=\left(-\frac{1}{2u}\right)\cdot\left(\ev_{\star}\right)^{*}d\theta,
\]
where $d\theta$ is the canonical angular form on $\RP^1,$ which satisfies $\iota_\xi d\theta=2.$

Recall the standard codimension $c$ corners, $C^{c}\left(\emptyset,\lf,\vd\right)$,
parameterize nodal configurations of $c+1$ disks connected by $c$
real nodes, which are numbered $i=1,...,c$. We define
\[
Q_{c}^{\lf,\vd}=\frac{-1}{2u}\,\frac{1}{c}\,\sum_{i=1}^{c}d\theta_{i},
\]
where $d\theta_{i}$ is the pullback of $d\theta$ along the evaluation
map corresponding to the $i^{th}$ node.

It follows that for $s\leq c$ we have
\[
Q_{s}|_{C^{c}\left(\emptyset,\lf,\vd\right)}=\frac{-1}{2u}\frac{1}{s}\left(d\theta_{s}+\cdots+d\theta_{c}\right)
\]
so
\[
Q_{\leq c}^{\lf,\vd}=Q_{c}^{\lf,\vd}\cdots Q_{1}^{\lf,\vd}=\frac{1}{c!}\,\frac{1}{\left(-2u\right)^{c}}\,d\theta_{1}\cdots d\theta_{c}.
\]
We say that a connected component of standard corners $\bb\subset\cl^{c}\left(\emptyset,\lf,\vd\right)$ is \emph{degenerate }if an interior point of $\bb$ has a degree zero
disk with two or more nodes attached, and \emph{non-degenerate} otherwise.
Note that if $\bb$ is degenerate, then $Q_{\leq c}^{\lf,\vd}|_{\bb}=0$,
since $d\theta_{i}=d\theta_{j}$ for some $i\neq j$.

\subsection{Recursion for corner contributions}

The following recursion is the basis for the proof of the fixed-point formula in Theorem \ref{thm:loc for CP1,RP1}.
\begin{prop}
\label{prop:recursion}We have
\begin{gather*}
\int_{\cl^{c}\left(\emptyset,\lf,\vd\right)}Q_{\leq c}^{\lf,\vd}\alpha_{\lf,\vd}=\frac{1}{2}\frac{\left(d_{1}^{+}-d_{1}^{-}\right)\left(d_{2}^{+}-d_{2}^{-}\right)}{\left(-2u\right)}\,\frac{1}{c}\,\sum\int_{\cl^{c_{1}}\left(\emptyset,\lf_{1},\vd_{1}\right)}Q_{\leq c_{1}}^{\lf_{1},\vd_{1}}\alpha_{\lf_{1},\vd_{1}}\times
\int_{\cl^{c_{2}}\left(\emptyset,\lf_{2},\vd_{2}\right)}Q_{\leq c_{2}}^{\lf_{2},\vd_{2}}\alpha_{\lf_{2},\vd_{2}},
\end{gather*}
where the sum ranges over all $\left(\lf_{1},\vd_{1}\right)\#\left(\lf_{2},\vd_{2}\right)=\left(\lf,\vd\right)$ and $c_{1}+c_{2}=c-1$.
\end{prop}
The proof of this proposition occupies the remainder of this section.
Since $Q_{\leq c}$ vanishes on degenerate corners, and $\alpha_{i}^{\lf_{i},\vd_{i}}$
vanishes if $\vd_{i}=\left(0,0\right)$, we have
\[
\int_{\cl^{c}\left(\emptyset,\lf,\vd\right)}Q_{\leq c}^{\lf,\vd}\alpha_{\lf,\vd}=\frac{1}{2}\int_{\widetilde{\cl}^{c}\left(\lf,\vd\right)}Q_{\leq c}^{\lf,\vd}\alpha_{\lf,\vd},
\]
where
\[
\widetilde{\cl}^{c}\left(\lf,\vd\right)=\bigsqcup\partial_{nd}^{\sigma_{1}}\overline{\mm}_{0,\sstar_{1},\lf_{1}}\left(\vd_{1}\right)\times_{L}\partial_{nd}^{\sigma_{2}}\overline{\mm}_{0,\sstar_{2},\lf_{2}}\left(\vd_{2}\right),
\]
is the disjoint union ranging over all $\sigma_{j},\lf_{j},\vd_{j}$, $j=1,2$, as in Proposition \ref{prop:corners of moduli}, where $\vd_{1},\vd_{2}$ are both not equal to $\left(0,0\right)$, and
\[
\partial_{nd}^{\sigma_{j}}\overline{\mm}_{0,\sstar_{j},\lf_{j}}\left(\vd_{j}\right)\subset\partial^{\sigma_{j}}\overline{\mm}_{0,\sstar_{j},\lf_{j}}\left(\vd_{j}\right)
\]
are the non-degenerate corners.

Now fix some $\sigma_{1},\sigma_{2}$, $\sigma_{1}\sqcup\sigma_{2}=\left\{ 1,...,c-1\right\} $
and write $C=\partial_{nd}^{\sigma_{1}}\overline{\mm}_{0,\sstar_{1},\lf_{1}}\left(\vd_{1}\right)\times_{L}\partial_{nd}^{\sigma_{2}}\overline{\mm}_{0,\sstar_{2},\lf_{2}}\left(\vd_{2}\right)$.
Non-degenerate corners are vertical in the sense that the forgetful
map ${\overline{\mm}_{0,\sstar_{j},\lf_{j}}\left(\vd_{j}\right)\to\overline{\mm}_{0,\emptyset,\lf_{j}}\left(\vd_{j}\right)}$
induces a map
\begin{equation}\label{eq:f_i}
\mm_{j}=\partial_{nd}^{\sigma_{j}}\overline{\mm}_{0,\sstar_{j},\lf_{j}}\left(\vd_{j}\right)\xrightarrow{f_{j}}\check{\mm}_{j}=\partial_{nd}^{\sigma_{j}}\overline{\mm}_{0,\emptyset,\lf_{j}}.
\end{equation}
Indeed, when we forget $\sstar_{j}$ the \emph{depth} of a point (i.e.,
the number of local boundary components incident to that point) decreases
only if a disk component becomes unstable; for non-degenerate boundaries
this is never the case and so the depth is preserved. See \cite[\S 2.7]{mod2hom}
for more details. We will also need to consider the vertical/horizontal
boundary decomposition of $f_{j}$ itself:
\[
\partial\mm_{j}=\partial^{-}\mm_{j}\sqcup\partial^{+}\mm_{j}.
\]
The \emph{vertical} clopen component $\partial^{-}\mm_{j}$ admits
a map $\partial^{-}\mm_{j}\to\partial\check{\mm}_{j}$. The \emph{horizontal}
clopen component $\partial^{+}\mm_{j}\subset\partial\mm_{j}$ parameterizes
configurations with a \emph{ghost disk}: an irreducible component
which is a degree $\left(0,0\right)$ disk with no special points
on the interior and three special points on the boundary: two boundary
nodes labeled by $\left\{ a_{1},a_{2}\right\} \subset\sigma_{j}$,
and the boundary marking $\sstar{}_{j}$ (this is the only type of
disk that becomes unstable when we forget $\sstar_{j}$). We can
define an involution
\[
\inv_{j}:\partial^{+}\mm_{j}\to\partial^{+}\mm_{j}
\]
that switches the labels $\left\{ a_{1},a_{2}\right\} $ of the nodes. %\footnote{Amitai: A sanity check for me (Ran) - this involution changes the cyclic order of the components which meet at the ghost, or does it only changes the order of the local boundary components, but is "geometrically" meaningless?-----Yes}
This is an orientation-\emph{reversing} involution and we have
\begin{equation}
f_{j}\circ i\circ\inv_{j}=f_{j}\circ i\text{ and }\ev_{j}\circ i\circ\inv_{j}=\ev_{j}\circ i,\label{eq:involution invariance}
\end{equation}
where $i:\partial^{+}\mm_{j}\to\mm_{j}$.

Set $f=f_{1}\times f_{2}$ and write $C\overset{g}{\hookrightarrow}\mm_{1}\times\mm_{2}$
for the structure map. Set $\Gamma\in\aso\left(\check{\mm}_{1}\times\check{\mm}_{2}\right)$
to be
\[
\Gamma=\Prr_{1}^{*}Q_{\leq c_{1}}\alpha_{\lf_{1},\vd_{1}}\wedge\Prr_{2}^{*}Q_{\leq c_{2}}\alpha_{\lf_{2},\vd_{2}}=:\Prr_{1}^{*}\Gamma_{1}\wedge\Prr_{2}^{*}\Gamma_{2}.
\]
We have the following version of the projection formula.
\begin{lem}
\label{lem:projection formula}For $\vd^{j}=\left(d_{+}^{j},d_{-}^{j}\right)\neq\left(0,0\right)$,
if we orient $C$ as a corner of $\overline{\mm}_{0,\emptyset,\lf_{1}\sqcup\lf_{2}}\left(\vd^{1}+\vd^{2}\right)$
and $\check{\mm}_{1}\times\check{\mm}_{2}$ as a product of corners,
then
\begin{gather}
\int_{C}d\theta_c\,g^{*}f^{*}\Gamma=\left(-1\right)^{c-1}\left(-1\right)^{\sigma}\left(d_{+}^{1}-d_{-}^{1}\right)\left(d_{+}^{2}-d_{-}^{2}\right)\int_{\check{\mm}_{1}}\Gamma_{1}\times\int_{\check{\mm}_{2}}\Gamma_{2},\label{eq:proj formula}
\end{gather}
where, writing the elements of $\sigma_{j}$ in order $\sigma_{j}^{1}<\sigma_{j}^{2}<\cdots<\sigma_{j}^{c_{j}}$, we define
\[
\left(-1\right)^{\sigma}=\operatorname{sign}\begin{pmatrix}1 & \cdots &  &  & \cdots & c-1\\
\sigma_{1}^{1} & \cdots & \sigma_{1}^{c_{1}} & \sigma_{2}^{1} & \cdots & \sigma_{2}^{c_{2}}
\end{pmatrix}
\]
to be the shuffle sign.
\end{lem}
\begin{proof}
Let $\ev$ be the evaluation at the $c^{th}$ node, so that $d\theta_c=\ev^*d\theta.$
Consider the diagram
\[
\xymatrix{C\aru{r}{g}\ar[d]_{\ev} & \mm_{1}\times\mm_{2}\ar[r]^{f}\ar[d]^{\ev_{12}} & \check{\mm}_{1}\times\check{\mm}_{2}\\
L\ar[r]_{\delta} & L^{2}
}
\]
Here $L=\rr\pp^{1}$ and $\delta$ denotes the diagonal map. Using
the $S^1$-action one checks that the map $\ev_{12}=\ev_{1}\times\ev_{2}$
is a b-submersion (this means the differential is onto, as well as
all of its restrictions to the corners of the domain, see \cite{equiv-OGW-invts})
and so the square is a b-transverse cartesian square.

We let $\gamma\in\Omega^{\dim\left(\check{\mm}_{1}\times\check{\mm}_{2}\right)}\left(\check{\mm}_{1}\times\check{\mm}_{2}\right)$ be the \emph{top de Rham degree} part of $\Gamma$:
\[
\Gamma=\gamma\,u^{m}\mod u^{m+1},
\]
where $m=\frac{\deg1\Gamma-\dim\left(\check{\mm}_{1}\times\check{\mm}_{2}\right)}{2}$.
We define $\gamma_{1},\gamma_{2}$ similarly, so $\gamma=\gamma_{1}\wedge\gamma_{2}$.

The lemma follows immediately from the following computation:
\begin{align}
\left(-1\right)^{\sigma}\int_{C}\ev^{*}d\theta\,g^{*}f^{*}\gamma\overset{\left(\text{I}\right)}{=}&\left(-1\right)^{c_{1}}\int_{\mm_{1}\times\mm_{2}}\left(g_{*}\ev^{*}d\theta\right)\,f^{*}\gamma=\label{eq:projection formula computation}\\
\overset{\left(\text{II}\right)}{=}&\left(-1\right)^{c_{1}}\int_{\mm_{1}\times\mm_{2}}\left(\ev_{1}\times\ev_{2}\right)^{*}\left(d\theta_{1}\wedge d\theta_{2}\right)\,f^{*}\gamma=\notag\\
\overset{\left(\text{III}\right)}{=}&\left(\int_{\mm_{1}}\ev_{1}^{*}d\theta_{1}\,f_{1}^{*}\gamma_{1}\right)\times\left(\int_{\mm_{2}}\ev_{2}^{*}d\theta_{2}\,f_{2}^{*}\gamma_{2}\right)=\notag\\
\overset{\left(\text{IV}\right)}{=}&\left(-1\right)^{c-1}\left(\left(d_{1}^{+}-d_{1}^{-}\right)\int_{\check{\mm}_{1}}\gamma_{1}\right)\times\left(\left(d_{2}^{+}-d_{2}^{-}\right)\int_{\check{\mm}_{2}}\gamma_{2}\right).\notag
\end{align}
Let us justify this computation step by step. The discussion of orientations
and signs appears at the end of the proof.

First, we define $g_{*}$. Let $\partial_{1}-\partial_{2}$ denote
the vector field on $L\times L$ which is normal to the diagonal $L\xrightarrow{\delta}L\times L$.
By \cite[Lemma 55]{fp-loc-OGW} there exists a vector field $\eta$
on an open neighbourhood $U\subset\mm_{1}\times\mm_{2}$ of $C$,
such that (i) $\eta$ is \emph{b-tangent}, that is, for any $p\in U$
it is tangent to all the boundary faces incident to $p,$ and (ii)
$\left(\ev_{1}\times\ev_{2}\right)_{*}\eta=\partial_{1}-\partial_{2}$.
Flowing along $\eta$ and $\partial_{1}-\partial_{2}$ we construct
a pair of compatible tubular neighbourhoods $C\subset V\subset\mm_{1}\times\mm_{2}$
and $L\subset V_{0}\subset L\times L$. We denote by $V\xrightarrow{\pi}C$
and $V_{0}\xrightarrow{\pi_{0}}L$ the associated projections, which
satisfy $\ev\circ\pi=\pi_{0}\circ\ev_{12}$.

Let $r_{0}:V_{0}\to[0,\epsilon)$ be the distance function from the
diagonal associated with the flow. Let $\widetilde{V}_{0}\xrightarrow{B_{0}}V_{0}$
denote the hyperplane blowup of the diagonal (see \cite[\S 3.3]{mod2hom}); it is isomorphic to
$(-1,0]\sqcup[0,1)\times L\to\left(-1,1\right)\times L$, and extends
to a blow up $\widetilde{L\times L}\to L\times L$. We construct a
Thom form $\tau_{0}$ for $\delta$ supported in $V_{0}$, following
Bott and Tu \cite{bott+tu}. We define $\tau_{0}$ by
\[
B_{0}^{*}\tau_{0}=d\left(\sigma\left(r_{0}\right)\phi_{0}\right),
\]
where $\sigma:[0,\epsilon)\to\left[-1,0\right]$ is a smooth compactly
supported monotone function with $\left.\frac{\partial^{a}\sigma}{\partial r^{a}}\right|_{r=0}=0$
for all $a\geq1$, and $\phi_{0}=\pm\frac{1}{2}$ is a locally constant
function obtaining opposite values on the two connected components
of $\widetilde{V}_{0}$ ($\phi_{0}$ is the pullback of the angular
form for $S^{0}=\pt\sqcup\pt$ to $S\left(N_{\delta}\right)=S^{0}\times L$;
see below for the orientation convention which fixes the trivialization).
Let $\tau=\ev_{12}^{*}\tau_{0}$ be the corresponding Thom form for
$C\subset\mm_{1}\times\mm_{2}$. We define
\[
\delta_{*}\omega=\tau_{0}\wedge\pi_{0}^{*}\omega,\qquad g_{*}\omega=\tau\wedge\pi^{*}\omega,
\]
the expressions on the right-hand side should be interpreted as the extension by zero of the indicated forms, which have compact support in tubular neighbourhoods of the diagonals, to $L\times L$ and to $\mm_{1}\times\mm_{2}$, respectively.
%the expressions on the right-hand side extend by zero to $L\times L$
%and to $\mm_{1}\times\mm_{2}$, respectively.

Let us justify equation (I). Let $\widetilde{V}\xrightarrow{B_{V}}V$
denote the hyperplane blowup of $V$ along $C$ (more directly, it
can be defined as the pullback of $B_{0}$ along $\ev_{1}\times\ev_{2}$).
We have $\partial\widetilde{V}=\left(S^{0}\times C\right)\sqcup\left(\partial V\right){}^{\sim}$,
where $\left(\partial V\right)^{\sim}$ is an open substack of $\partial\left(\mm_{1}\times\mm_{2}\right)^{\sim}$, the hyperplane blowup of $\partial\left(\mm_{1}\times\mm_{2}\right)$
along $\partial C$. We denote the map associated with this blowup
by
\[
\partial\left(\mm_{1}\times\mm_{2}\right)^{\sim}\xrightarrow{B_{\partial}}\partial\left(\mm_{1}\times\mm_{2}\right).
\]
There is a lift $\partial\left(\mm_{1}\times\mm_{2}\right)^{\sim}\xrightarrow{\tilde{\ev}_{12}}\widetilde{L\times L}$
of $\ev_{12}$,
\[
B_{\partial}\circ\tilde{\ev}_{12}=\ev_{12}\circ B_{\partial}.
\]
We have
\begin{align}
\int_{\mm_{1}\times\mm_{2}}\left(g_{*}\lambda\right)\,\rho=&\int_{\widetilde{V}}d\left(\sigma\left(r_{0}\right)\phi_{0}\left(\pi^{*}\lambda\right)\,\rho\right)=\int_{\partial\widetilde{V}}\sigma\left(r_{0}\right)\phi_{0}\left(\pi^{*}\lambda\right)\,\rho=\label{eq:adjunction}\\
=&\left(-1\right)^{\sigma}\left(-1\right)^{c_{1}}\int_{S^{0}\times C}\phi_{0}\left(g\circ\Prr_{C}\right)^{*}\left(\left(\pi^{*}\lambda\right)\,\rho\right)+\EEE=\left(-1\right)^{\sigma}\left(-1\right)^{c_{1}}\int_{C}\lambda\,g^{*}\rho+\EEE,\notag
\end{align}
where
\begin{equation}
\EEE=\int_{\partial\left(\mm_{1}\times\mm_{2}\right)^{\sim}}\widetilde{\ev}_{12}^{*}\left(\sigma\left(r_{0}\right)\phi_{0}\right)\,B^{*}\left(\left(\pi^{*}\lambda\right)\,\rho\right).\label{eq:bdry term}
\end{equation}
In general, $\EEE$ doesn't have to vanish, but for $\lambda=\ev^{*}d\theta$
and $\rho=f^{*}\gamma$ we claim that it does, completing the justification
of step (I). Define
\[
\partial_{1}^{+}=\left(\partial^{+}\mm_{1}\right)\times\mm_{2}\subset\partial\left(\mm_{1}\times\mm_{2}\right),\qquad \tilde{\partial}_{1}^{+}=\left(\partial^{+}\mm_{1}\times\mm_{2}\right)^{\sim}\subset\left(\partial\left(\mm_{1}\times\mm_{2}\right)\right)^{\sim},
\]
and define $\partial_{j}^{\pm},\tilde{\partial}_{j}^{\pm}$ similarly,
so
\[
\partial\left(\mm_{1}\times\mm_{2}\right)^{\sim}=\tilde{\partial}_{1}^{+}\sqcup\tilde{\partial}_{2}^{+}\sqcup\tilde{\partial}_{1}^{-}\sqcup\tilde{\partial}_{2}^{-}.
\]
The involution $\inv_{j}$ lifts to an orientation-reversing involution $\widetilde{\inv}_{j}:\tilde{\partial}_{j}^{+}\to\tilde{\partial}_{j}^{+}$
with $B_{\partial}\circ\widetilde{\inv}_{j}=\inv_{j}\circ B_{\partial}$.
We have $\pi^{*}\ev^{*}d\theta=\ev_{12}^{*}\pi_{0}^{*}d\theta$. Using
this and \eqref{eq:involution invariance} it follows that
\[
\left[\tilde{\ev}_{12}^{*}\left(\sigma\left(r_{0}\right)\phi_{0}\right)\wedge B^{*}\left(\pi^{*}\ev^{*}d\theta\wedge f^{*}\gamma\right)\right]|_{\tilde{\partial}_{j}^{+}}
\]
is invariant under $\widetilde{\inv}_{j}^{*}$. Since $\widetilde{\text{inv}}_{j}$
is orientation-reversing, the contribution of $\int_{\tilde{\partial}_{j}^{+}}\cdots$
to $\EEE$ vanishes. On $\tilde{\partial}_{1}^{-}\sqcup\tilde{\partial}_{2}^{-}$,
$f^{*}\gamma$ vanishes identically. Indeed, there exists a map $f^{-}$
making the following square commute
\[
\xymatrix{\partial_{1}^{-}\sqcup\partial_{2}^{-}\ar[r]\ar[d]_{f^{-}} & \mm_{1}\times\mm_{2}\ar[d]^{f}\\
\partial\left(\check{\mm}_{1}\times\check{\mm}_{2}\right)\ar[r] & \check{\mm}_{1}\times\check{\mm}_{2}
}
\]
so $\left(f^{*}\gamma\right)|_{\partial_{1}^{-}\sqcup\partial_{2}^{-}}=\left(f^{-}\right)^{*}\left(\gamma|_{\partial\left(\check{\mm}_{1}\times\check{\mm}_{2}\right)}\right)=0$,
since $\gamma$ is a top form.

Let us justify equation (II). We have
\[
g_{*}\ev^{*}d\theta=\tau\wedge\pi^{*}\ev^{*}d\theta=\ev_{12}^{*}\left(\tau_{0}\wedge\pi_{0}^{*}d\theta\right)=\ev_{12}^{*}\delta_{*}d\theta.
\]
Since $d\theta$ is Poincar\'e dual to a point in $L$, $\delta_{*}d\theta$
is Poincar\'e dual to a point in $L\times L$, so there exists some
$\epsilon$ with
\[
\ev_{12}^{*}d\epsilon=g_{*}\ev^{*}d\theta-\ev_{12}^{*}\left(d\theta_{1}\wedge d\theta_{2}\right).
\]
It follows that
\[
\int_{\mm_{1}\times\mm_{2}}\left(g_{*}\ev^{*}d\theta\right)\,f^{*}\gamma=\int_{\mm_{1}\times\mm_{2}}\ev_{12}^{*}\left(d\theta_{1}\wedge d\theta_{2}\right)\,f^{*}\gamma+\int_{\partial\left(\mm_{1}\times\mm_{2}\right)}\ev_{12}^{*}\epsilon\,f^{*}\gamma.
\]
Again, the boundary term on the right-hand side vanishes, the argument is essentially the same as for \eqref{eq:bdry term}.

Step (III) follows from $\deg1\gamma_{1}=c_{1}$ and Fubini's theorem.
To prove step (IV) we need to show
\[
\int_{\mm_{1}}\ev_{1}^{*}d\theta_{1}\wedge f_{1}^{*}\gamma_{1}=\left(-1\right)^{c_{1}}\left(d_{1}^{+}-d_{1}^{-}\right)\int_{\check{\mm}_{1}}\gamma_{1}.
\]
Consider some point $p\in\check{\mm}_{1}$, represented by a fundamental
configuration $\sigma=\left(\left(\Sigma,\nu,\lambda,w\right),b,\Sigma^{1/2}\right)$.
Let $y_{p}\in L$ be a regular value for $w|_{\partial\Sigma^{1/2}}$. It
is not hard to see that for a small open interval $y_{p}\in V_{p}\subset L$
and a small open neighbourhood $p\in U_{p}\subset\check{\mm}_{1}$\textbf{
$f_{1}^{-1}\left(U_{p}\right)\cap\ev_{1}^{-1}\left(V_{p}\right)\xrightarrow{f_{1}\times\ev_{1}}U_{p}\times V_{p}$}
is a (possibly empty) finite covering map, whose sheets are in bijection
with $w|_{\partial}^{-1}\left(y_{p}\right)$. Choose a finite open
subcover $\left\{ U_{\alpha}=U_{p_{\alpha}}\right\} $ with a subordinate
partition of unity $\rho_{\alpha}$. Let $\lambda_{\alpha}\in\Omega^{1}\left(L\right)$
be Poincar\'e dual to a point with compact support in $V_{\alpha}=V_{p_{\alpha}}$,
and let $\epsilon_{\alpha}\in\Omega^{0}\left(L\right)$ satisfy $d\epsilon_{\alpha}=d\theta-\lambda_{\alpha}$.
We compute
\begin{multline*}
\int_{\mm_{1}}\ev_{1}^{*}d\theta_{1}\wedge f_{1}^{*}\gamma_{1}=\sum_{\alpha}\int_{\mm_{1}}\ev_{1}^{*}d\theta_{1}\wedge f_{1}^{*}\left(\rho_{\alpha}\gamma_{1}\right)=\\
=\sum_{\alpha}\int_{\mm_{1}}\ev_{1}^{*}\lambda_{\alpha}\wedge f_{1}^{*}\left(\rho_{\alpha}\gamma_{1}\right)+\sum_{\alpha}\int_{\partial\mm_{1}}\ev_{1}^{*}\epsilon_{\alpha}\wedge f_{1}^{*}\left(\rho_{\alpha}\gamma_{1}\right).
\end{multline*}
We have used the fact $d\left(\rho_{\alpha}\gamma_{1}\right)=0$, since
$\rho_{\alpha}\gamma_{1}$ is a top form. The boundary terms vanish
as before (by using $\inv_{1}$ to show that the contribution of $\partial^{+}\mm_{1}$
vanishes, and arguing that $f^{*}\rho_{\alpha}\gamma_{1}|_{\partial^{-}\mm_{1}}\equiv0$).
To evaluate $\int_{\mm_{1}}\ev_{1}^{*}\lambda_{\alpha}\wedge f_{1}^{*}\left(\rho_{\alpha}\gamma_{1}\right)$,
use $f_{1}\times\ev_{1}$ as local coordinates. The contribution of
each sheet is $\pm\int_{U_{\alpha}}\rho_{\alpha}\gamma_{1}$, where
the sign is $\pm1$ according to \eqref{eq:forgetful fiber or}. This signed count satisfies
\begin{equation}\label{eq:w_signed_count}
\#\,w|_{\partial}^{-1}\left(y_{p}\right)=\left(-1\right)^{c_{1}}\left(d_{1}^{+}-d_{1}^{-}\right),
\end{equation}
where only the sign needs some explanation, see below. So we get
\[
\sum_{\alpha}\int_{\mm_{1}}\ev_{1}^{*}\lambda_{\alpha}\wedge f_{1}^{*}\left(\rho_{\alpha}\gamma_{1}\right)=\left(-1\right)^{c_{1}}\left(d_{1}^{+}-d_{1}^{-}\right)
\int_{\check{\mm}_{1}}\left(\sum\rho_{\alpha}\right)\gamma=\left(-1\right)^{c_{1}}\left(d_{1}^{+}-d_{1}^{-}\right)\int_{\check{\mm}_{1}}\gamma,
\]
completing the justification of step (IV).

It remains to explain our orientation conventions and discuss signs.
A generic point $p\in C$ satisfies $g\left(p\right)\not\in i\left(\partial_{1}^{+}\sqcup\partial_{2}^{+}\right)$,
that $d\ev|_{g\left(p\right)},df|_{g\left(p\right)}$ are submersions,
and that the isotropy at $p$ is trivial. Fix such a point. For the remainder of the
proof we will pull back all of the vector bundles to $p$, and work
exclusively with the vector spaces we obtain in this way and with
their oriented bases. Henceforth, if the pullback
to $p$ is unambiguously defined we omit it from the notation. Let $[X]$ denote an oriented base for $T_qX,$ for an orbifold $X$ and a point $q$ that will be clear from the context.
We write $\ev^{*}TL,\ev_{1}^{*}TL,\ev_{2}^{*}TL$ for the positive vectors
(dual to $d\theta$) on the corresponding vector spaces.

Recall
\[
\check{\mm}_{j}=\partial^{\sigma_{j}}\overline{\mm}_{0,\emptyset,\lf_{j}}\left(\vd_{j}\right),
\]
and $\sigma_{1}\sqcup\sigma_{2}=\left\{ 1,...,c-1\right\} $. Let
\[
O_{j}=o\left(\sigma_{j}^{c_{j}}\right)o\left(\sigma_{j}^{c_{j-1}}\right)\cdots o\left(\sigma_{j}^{1}\right)
\]
denote the sequence of outward normal vectors in $\overline{\mm}_{0,\emptyset,\lf_{j}}\left(\vd_{j}\right),$
with $o\left(i\right)$ corresponding to the $i^{th}$ local boundary
component. Note the elements of $\sigma_{j}$ appear in reverse order;
this is because, iterating the definition of the boundary orientation,
we find that there is an equality of orientations
\[
O_{j}\,[\check{\mm}_{j}]=[\overline{\mm}_{0,\emptyset,\lf_{j}}\left(\vd_{j}\right)].
\]
Let $[\overline{\mm}]$ denote an oriented basis for $\overline{\mm}_{0,\emptyset,\lf_{1}\sqcup\lf_{2}}\left(\vd_{1}+\vd_{2}\right)$,
and let $O=o(c)$ denote an outward normal vector corresponding
to smoothing the $\left(\sstar_{1},\sstar_{2}\right)$-node. Since
$C$ is oriented as a corner of $\overline{\mm}$, we find that
\[
\left(-1\right)^{\sigma+c_{1}c_{2}}O\,O_{1}\,O_{2}\,[C]=[\overline{\mm}],
\]
where
\[
\left(-1\right)^{\sigma+c_{1}c_{2}}=\operatorname{sign}\begin{pmatrix}c-1 &  &  &  & \cdots &  &  & 2 & 1\\
\sigma_{1}^{c_{1}} & \sigma_{1}^{c_{1}-1} & \cdots & \sigma_{1}^{2} & \sigma_{1}^{1} & \sigma_{2}^{c_{2}} & \sigma_{2}^{c_{2}-1} & \cdots & \sigma_{2}^{1}
\end{pmatrix}.
\]

Let $F_{f\circ g}$ denote a vector spanning $\ker d\left(f\circ g\right)$,
oriented so that the following equality holds:
\begin{equation}
\left(-1\right)^{\sigma+c_{1}c_{2}}O_{1}\,O_{2}\,[C]=F_{f\circ g}\,O_{1}[\check{\mm}_{1}]\,O_{2}[\check{\mm}_{2}].\label{eq:fg fib or}
\end{equation}
To understand this equation, note that the standard convention for
orienting the fiber of $f\circ g$ would lead to an equation of the
form
\[
C\overset{!}{=}F_{f\circ g}^{!}\,\check{\mm}_{1}\,\check{\mm}_{2}.
\]
Instead, \eqref{eq:fg fib or} comes from identifying $F_{f\circ g}$
with the tangent space to the fiber of the map
\[
\overline{\mm}_{0,\sstar_{1},\lf_{1}}\left(\vd_{1}\right)\times_{L}\overline{\mm}_{0,\sstar_{2},\lf_{2}}\left(\vd_{2}\right)\to\overline{\mm}_{0,\emptyset,\lf_{1}}\left(\vd_{1}\right)\times\overline{\mm}_{0,\emptyset,\lf_{2}}\left(\vd_{2}\right),
\]
where the domain is oriented as a boundary of $\overline{\mm}$.
We use this orientation for $F_{f\circ g}$ because in Lemma \ref{lem:induced for bdry node} it is computed to be
\[
F_{f\circ g}=\left(-1\right)^{s_{1}+s_{2}}\ev^{*}TL,
\]
where $s_{j}$ is the sign of the region incident to $\sstar_{j}$.

Plugging this in \eqref{eq:fg fib or} and using that $\dim\check{\mm}_{1}=c_{1}\mod 2$
we get
\begin{equation}
[C]=\left(-1\right)^{\sigma}\left(-1\right)^{s_{1}+s_{2}}\left(-1\right)^{c-1}\left(\ev^{*}TL\right)\,[\check{\mm}_{1}]\,[\check{\mm}_{2}].\label{eq:C orientation}
\end{equation}

We turn to the construction of the Thom forms. We orient the normal
lines $N_{g},N_{\delta}$ associated with $g,\delta$ by
\begin{align}
&[\overline{\mm}_{0,\sstar_1,\lf_{1}}\left(\vd_{1}\right)][\overline{\mm}_{0,\sstar_2,\lf_{2}}\left(\vd_{2}\right)]=N_{g}\,[C]=\left(-1\right)^{\sigma}\left(-1\right)^{s_{1}+s_{2}}\left(-1\right)^{c-1}N_{g}\left(\ev^{*}TL\right)\,[\check{\mm}_{1}]\,[\check{\mm}_{2}],\label{eq:Ng or}\\
&(\ev_{1}^{*}TL)\,(\ev_{2}^{*}TL)=N_{\delta}\ev^{*}TL.\label{eq:Ndelta or}
\end{align}

Since $g$ is the pullback of $\delta$, $N_{g},N_{\delta}$ are naturally
identified as vector spaces. We claim that we have an equality of
oriented vector spaces
\begin{equation}
N_{g}=\left(-1\right)^{\sigma}\left(-1\right)^{c_{1}}N_{\delta}.\label{eq:normal gap}
\end{equation}
To see this, let $F_{f_{j}}$ be a vector spanning $\ker df_{j}$
oriented by the fiber convention
\begin{equation}
[\overline{\mm}_{0,\sstar_j,\lf_{j}}\left(\vd_{j}\right)]=F_{f_{j}}\,[\check{\mm}_{j}].\label{eq:def of forgetful fiber or}
\end{equation}
We have
\[
O_{j}\,F_{f_{j}}\,[\check{\mm}_{j}]=[\overline{\mm}_{0,\sstar_{j},\lf_{j}}\left(\vd_{j}\right)]=
\left(-1\right)^{s_{j}}(\ev_{j}^{*}TL)\,[\overline{\mm}_{0,\emptyset,\lf_{j}}\left(\vd_{j}\right)]=
\left(-1\right)^{s_{j}}(\ev_{j}^{*}TL)\,O_{j}[\check{\mm}_{j}],
\]
where the second equality follows from the orientation formula \eqref{eq:or}, so
\begin{equation}
F_{f_{j}}=\left(-1\right)^{s_{j}+c_{j}}\ev_{j}^{*}TL\label{eq:forgetful fiber or}
\end{equation}
Plugging \eqref{eq:forgetful fiber or} and \eqref{eq:def of forgetful fiber or}
in \eqref{eq:Ng or} we find
\[
(\ev_{1}^{*}TL)\,(\ev_{2}^{*}TL)=\left(-1\right)^{\sigma}\left(-1\right)^{c_{1}}N_{g}\,\ev^{*}TL.
\]
Comparing this to \eqref{eq:Ndelta or} gives \eqref{eq:normal gap}.

We use this to explain the sign appearing between the first and second
line in \eqref{eq:adjunction}. Indeed it is not hard to see that,
had we used a Thom form $\tau^{!}$ constructed using the orientation
of $N_{g}$, there would be no sign. Instead, we defined $\tau=\ev_{12}^{*}\tau_{0}$,
so there is a sign from \eqref{eq:normal gap}.

Step (II) is an equality of forms. The sign of step (III) requires no additional explanation. The sign of step (IV) comes from the calculation \eqref{eq:w_signed_count}, which follows from \eqref{eq:forgetful fiber or}. This concludes the proof of the
lemma.
\end{proof}

\begin{proof}
[Proof of Proposition \ref{prop:recursion}]By definition,
\begin{gather*}
\frac{1}{2}\int_{\widetilde{\cl}^{c}\left(\emptyset,\lf,\vd\right)}Q_{\leq c}^{\lf,\vd}\alpha_{\lf,\vd}=\frac{1}{2\cdot c!\cdot\left(-2u\right)^{c}}\sum_{\sigma_{1},\sigma_{2},\lf_{1},\lf_{2},\vd_{1},\vd_{2}}\int_{C}d\theta_{1}\cdots d\theta_{c}\,g^{*}f^{*}\,\left(\alpha_{\lf_{1},\vd_{1}}\,\alpha_{\lf_{2},\vd_{2}}\right),
\end{gather*}
where
\[
C=C\left(\sigma_{1},\sigma_{2},\lf_{1},\lf_{2},\vd_{1},\vd_{2}\right)=\partial_{nd}^{\sigma_{1}}\overline{\mm}_{0,\sstar_{1},\lf_{1}}\left(\vd_{1}\right)\times_{L}\partial_{nd}^{\sigma_{2}}\overline{\mm}_{0,\sstar_{2},\lf_{2}}\left(\vd_{2}\right).
\]

We have
\begin{align*}
\frac{1}{2\cdot c!\cdot\left(-2u\right)^{c}}&\int_{C}d\theta_{1}\cdots d\theta_{c}g^{*}f^{*}\,\left(\alpha_{\lf_{1},\vd_{1}}\,\alpha_{\lf_{2},\vd_{2}}\right)=\\
=&\frac{\left(-1\right)^{c-1}\left(-1\right)^{\sigma}}{2\cdot c!\cdot\left(-2u\right)^{c}}\int_{C}d\theta_{c}\,g^{*}f^{*}\left(d\theta_{\sigma_{1}^{1}}\cdots d\theta_{\sigma_{1}^{c_{1}}}\alpha_{\lf_{1},\vd_{1}}\,d\theta_{\sigma_{2}^{1}}\cdots d\theta_{\sigma_{2}^{c_{2}}}\alpha_{\lf_{2},\vd_{2}}\right)=\\
=&\frac{\left(-1\right)^{c-1}\left(-1\right)^{\sigma}}{2\cdot c\cdot\left(-2u\right)}\cdot\frac{c_{1}!\,c_{2}!}{\left(c-1\right)!}\int_{C}\ev^{*}d\theta\,g^{*}f^{*}\Gamma=\\
=&\frac{\left(d_{+}^{1}-d_{-}^{1}\right)\left(d_{+}^{2}-d_{-}^{2}\right)}{2\cdot\left(-2u\right)}\cdot\frac{1}{c}\cdot\frac{c_{1}!\,c_{2}!}{\left(c-1\right)!}\int_{\partial_{nd}^{\sigma_{1}}\overline{\mm}_{0,\emptyset,\lf_{1}}\left(\vd^{1}\right)}\Gamma_{1}\cdot\int_{\partial_{nd}^{\sigma_{2}}\overline{\mm}_{0,\emptyset,\lf_{2}}\left(\vd^{2}\right)}\Gamma_{2},
\end{align*}
where we used Lemma \ref{lem:projection formula} in the last step.

Summing over all $\sigma_{1}\sqcup\sigma_{2}=\left\{ 1,...,c-1\right\} $
with $\left|\sigma_{j}\right|=c_{j}$, then over all $c_{1}+c_{2}=c-1$,
$\lf_{1}\sqcup\lf_{2}=\lf$ and $\vd_{1}+\vd_{2}=\vd$, we find that
\begin{gather*}
\int_{\cl^{c}\left(\emptyset,\lf,\vd\right)}Q_{\leq c}^{\lf,\vd}\alpha_{\lf,\vd}=\frac{1}{2}\cdot\frac{\left(d_{1}^{+}-d_{1}^{-}\right)\left(d_{2}^{+}-d_{2}^{-}\right)}{\left(-2u\right)}\,\frac{1}{c}\,\sum\int_{\partial_{nd}^{c_{1}}\overline{\mm}_{0,\emptyset,\lf_{1}}\left(\vd_{1}\right)}Q_{\leq c_{1}}^{\lf_{1},\vd_{1}}\alpha_{\lf_{1},\vd_{1}}\times\int_{\partial_{nd}^{c_{2}}\overline{\mm}_{0,\emptyset,\lf_{2}}\left(\vd_{2}\right)}Q_{\leq c_{2}}^{\lf_{2},\vd_{2}}\alpha_{\lf_{2},\vd_{2}}.
\end{gather*}
The statement of the proposition follows, since we can enlarge the
domains of integration to include the degenerate corners
\[
\cl^{c_{j}}\left(\emptyset,\lf_{j},\vd_{j}\right)\supset\partial_{nd}^{c_{j}}\overline{\mm}_{0,\emptyset,\lf_{j}}\left(\vd_{j}\right)
\]
without changing the value of the integral.
\end{proof}

\subsection{Contribution of the exceptional boundary}
\begin{lemma}\label{lem:exceptional_cont}
The contribution of the exceptional boundary is
\[\int_{\ee}\left.Q_{1}^{\lf,\vd}\right|_{\ee}\alpha_{\lf,(d,d)}=
\int_{\ee}\left(-\frac{1}{2u}\right)\cdot\left(\ev_{\star}\right)^{*}d\theta f^*\beta_{\lf,d}
=\frac{-d}{\left(-2u\right)}\int_{\overline{\mm}_{0,\lf}\left(d\right)}\beta_{\lf,d},\]
where $f$ is the map which forgets $\star.$
\end{lemma}
\begin{proof}
By part \eqref{it:4} of Definition \ref{def:coherent integrand},
\[
\int_{\ee}\left.Q_{1}^{\lf,\vd}\right|_{\ee}\alpha_{\lf,(d,d)}=
\int_{\ee}\left(-\frac{1}{2u}\right)\cdot\left(\ev_{\star}\right)^{*}d\theta f^*\beta_{\lf,d}.
\]
Let $(\mm_\varepsilon)_{\varepsilon\in(0,1]}$ be a decreasing family of compact, $S^1-$invariant suborbifolds with boundary of $\overline{\mm}_{0,\lf}\left(d\right)$ whose union is ${\mm}_{0,\lf}(d).$ Let $\ee_\varepsilon = f^{-1}(\mm_\varepsilon).$ We have
\[
\int_{\ee}\left(-\frac{1}{2u}\right)\cdot\left(\ev_{\star}\right)^{*}d\theta f^*\beta_{\lf,d}=\lim_{\varepsilon\to 0}
\int_{\ee_\varepsilon}\left(-\frac{1}{2u}\right)\cdot\left(\ev_{\star}\right)^{*}d\theta f^*\beta_{\lf,d}.
\]
Indeed, $\overline{\mm}_{0,\lf}(d)\backslash {\mm}_{0,\lf}(d),$ and hence also $\overline{\ee}_{0,\lf}(d)\backslash {\ee}_{0,\lf}(d),$ are unions of strata with codimension at least $1$ to which the forms extend. Similarly,
\[\int_{\overline{\mm}_{0,\lf}\left(d\right)}\beta_{\lf,d} = \lim_{\varepsilon\to 0}\int_{{\mm_\varepsilon}}\beta_{\lf,d}.
\]
It is therefore enough to show for all $\varepsilon>0$
\[\int_{\ee_\varepsilon}\left(-\frac{1}{2u}\right)\cdot\left(\ev_{\star}\right)^{*}d\theta f^*\beta_{\lf,d}=
(d/2u)\int_{{\mm_\varepsilon}}\beta_{\lf,d}.\]

Now, $f|_{\ee_\varepsilon}:\ee_\varepsilon\to\mm_\varepsilon$ is a submersion. We can therefore calculate the left-hand side of the last equation by integration of $(\ev_\star)^*d\theta$ along the fibers of $f$. The integral along any fiber equals $-d,$ by Lemma \ref{lem:contracted boundary}, which completes the proof of the claim.
\end{proof}

\subsection{Solution to recursion}\label{subsec:Solution-to-Recursion}
For $\delta=\left(\lf,\vd\right)\in\mathcal{S}\left(\tilde{\lf},\tilde{\vd}\right)$
we write $\mm_{\delta}=\overline{\mm}_{0,\emptyset,\lf}\left(\vd\right)$,
and define $P_{c}^{\delta}\in\rr\left[u\right]$ for $c\geq0$
by
\[
P_{c}^{\delta}=\begin{cases}
\int_{\partial^{c}\mm_{\delta}}Q_{\leq c}^{\delta}\alpha_{\delta},&\text{if $c\neq 1$},\\
\int_{\partial\mm_{\delta}\backslash\ee}Q_{1}^{\delta}\alpha_{\delta},&\text{if $c=1$}.
\end{cases}
\]
Recall that $\ee$ denotes the exceptional boundary and that we use the definitions of $G(\delta),\AAH(T)$ from Notation \ref{nn:amplitude,F,E}.

We can combine Propositions \ref{prop:simple fp formula}, \ref{prop:recursion}, Lemma \ref{lem:exceptional_cont}
and the description of the boundary and corners, Proposition \ref{prop:corners of moduli}, to obtain the following pair of equations:
\begin{align}
&\sum_{c\geq0}\left(-1\right)^{c}P_{c}^{\delta}=G\left(\delta\right)\quad\text{for}\quad\delta\in\mathcal{S}\left(\tilde{\lf},\tilde{\vd}\right),\label{eq:alternating sum P}\\
&P_{c}^{\delta}=\sum_{\delta_{1}\#\delta_{2}=\delta}\frac{1}{2}\frac{\left(d_{1}^{+}-d_{1}^{-}\right)\left(d_{2}^{+}-d_{2}^{-}\right)}{-2u}\,\frac{1}{c}\,\sum_{c_{1}+c_{2}=c-1}P_{c_{1}}^{\delta_{1}}P_{c_{2}}^{\delta_{2}}\quad\text{for}\quad c\geq 1\quad\text{and}\quad\delta\in\mathcal{S}\left(\tilde{\lf},\tilde{\vd}\right).\label{eq:quadratic P}
\end{align}

\begin{lem}
Treating $\left\{ P_{c}^{\delta}\right\} $ as unknowns, there exists a unique solution to \eqref{eq:alternating sum P} and \eqref{eq:quadratic P}.
\end{lem}
\begin{proof}
Define a partial order
\[
\left(c_{1},\lf_{1},\vd_{1}\right)\leq\left(c_{2},\lf_{2},\vd_{2}\right)
\]
iff $c_{1}\leq c_{2}$ and $\lf_{1}\subseteq\lf_{2}$ and $d_{1}^{+}\leq d_{2}^{+}$
and $d_{1}^{-}\leq d_{2}^{-}$. We see that \eqref{eq:alternating sum P} and \eqref{eq:quadratic P} are upper-triangular; more precisely, if $c=0$ we can use \eqref{eq:alternating sum P}
to express $P_{c}^{\delta}$ in terms of $\left\{ P_{c'}^{\delta'}|\left(c',\delta'\right)<\left(c,\delta\right)\right\} $
and if $c\geq1$ we can use \eqref{eq:quadratic P} to do so. This
proves existence and uniqueness of a solution $\left\{ P_{c}^{\delta}\right\} $.
\end{proof}
We now describe an anzats for this.
\begin{lem}
\label{lem:anzats}The unique solution to \eqref{eq:alternating sum P}-\eqref{eq:quadratic P}
is given by
\[
P_{c}^{\lf,\vd}=\sum_{n\geq0}\binom{n}{c}\sum_{T\in\hts\left(n,\lf,\vd\right)}\AAH\left(T\right).
\]
\end{lem}
\begin{proof}
Let us check that this satisfies equation \eqref{eq:alternating sum P}.
Fix a tree $T\in\hts\left(n,\lf,\vd\right)$. If $n>0$, the contribution
of $\AAH\left(T\right)$ to $\sum_{c\geq0}\left(-1\right)^{c}P_{c}^{\lf,\vd}$
is given by
\[
\left(\sum_{c\geq0}\left(-1\right)^{c}\binom{n}{c}\right)\AAH\left(T\right)=0.
\]
If $n=0$, $T$ has a single vertex labeled $\left(\lf,\vd\right)$
and the contribution of $\AAH\left(T\right)$ to $\sum_{c\geq0}\left(-1\right)^{c}P_{c}^{\lf,\vd}$
is $G\left(\lf,\vd\right)$.

Let us check that this satisfies equation \eqref{eq:quadratic P}. Rewrite
the right-hand side of \eqref{eq:quadratic P} as a sum over tuples
$\left(\sigma,T_{1},v_{1},T_{2},v_{2}\right)$, where $\sigma=\left(\sigma_{1},\sigma_{2}\right)$
is a partition $\sigma_{1}\sqcup\sigma_{2}=\left[n-1\right]$ with
$\left|\sigma_{i}\right|=n_{i}$ and for $i=1,2$ we have $T_{i}\in\hts\left(n_{i},\lf_{i},\vd_{i}\right)$
and $v_{i}\in V\left(T_{i}\right)$, and where the contribution of
each such tuple is
\begin{align}
&\binom{n-1}{n_{1},n_{2}}^{-1}\frac{1}{2^{n}n!}\left(\frac{-1}{2u}\right)^{n}\sum_{c_{1}+c_{2}=c-1}\frac{1}{c}\frac{2^{n-1}n!}{2^{n_{1}}n_{1}!2^{n_{2}}n_{2}!}\times\label{eq:5-tuple contribution}\\
&\times\left(d^{+}\left(v_{1}\right)-d^{-}\left(v_{1}\right)\right)\binom{n_{1}}{c_{1}}\prod_{v\in V\left(T_{1}\right)}\left(d^{+}\left(v\right)-d^{-}\left(v\right)\right)^{\val\left(v\right)}\cdot G\left(\delta\left(v\right)\right)\times\notag\\
&\times\left(d^{+}\left(v_{2}\right)-d^{-}\left(v_{2}\right)\right)\binom{n_{2}}{c_{2}}\prod_{v\in V\left(T_{2}\right)}\left(d^{+}\left(v\right)-d^{-}\left(v\right)\right)^{\val\left(v\right)}\cdot G\left(\delta\left(v\right)\right).\notag
\end{align}
Now glue $v_{1}$ and $v_{2}$ by an oriented edge to obtain a tree
$T$, so that
\[
\prod_{j=1,2}\left(d^{+}\left(v_{j}\right)-d^{-}\left(v_{j}\right)\right)\prod_{v\in V\left(T_{j}\right)}\left(d^{+}\left(v\right)-d^{-}\left(v\right)\right)^{\val\left(v\right)}=\prod_{v\in V\left(T\right)}\left(d^{+}\left(v\right)-d^{-}\left(v\right)\right)^{\val\left(v\right)}.
\]
We use the partition $\sigma$ and the orders on the edges of $T_{1}$
and of $T_{2}$ to number the edges of $T$, labeling the new edge
$\left(v_{1},v_{2}\right)$ by $n=n_{1}+n_{2}+1$. This defines a
bijection
\[
\left\{ \left(\sigma,v_{1},T_{1},v_{2},T_{2}\right)\right\} \simeq\left\{ T\in\hts\left(n,\lf,\vd\right)\right\}.
\]
Plugging in the identity
\[
\sum_{c_{1}+c_{2}=c-1}\frac{1}{c}\binom{n_{1}}{c_{1}}\binom{n_{2}}{c_{2}}\frac{n!}{n_{1}!n_{2}!}=\binom{n}{c}\,\binom{n-1}{n_{1},n_{2}}
\]
completes the proof of \eqref{eq:quadratic P}.
\end{proof}
\begin{proof}
[Proof of Theorem \ref{thm:loc for CP1,RP1}]Theorem \ref{thm:loc for CP1,RP1} is the case $c=0$
of Lemma \ref{lem:anzats}.
\end{proof}

\section{Proofs of the main theorems}\label{sec:proofs}
\subsection{Fixed-point contributions for stationary descendent integrals}
Recall Notation \ref{nn:amplitude,F,E}.
Let $\left\{ \alpha_{\lf,\vd}\right\} ,\left\{ \beta_{\lf,d}\right\} $ be a coherent integrand for the tautological line bundles, we set
\begin{equation}\label{eq:F,E}
F\left(\lf,\vd,\vec{a},\vec{\epsilon}\right)=F^{\left\{ \alpha_{\delta}\right\} _{\delta\in\mathcal{D}},\left\{ \beta_{\delta}\right\} _{\delta\in\mathcal{S}}}\left(\lf,\vd\right),\qquad
E\left(\lf,\vd,\vec{a},\vec{\epsilon}\right) =E^{\left\{ \alpha_{\delta}\right\} _{\delta\in\mathcal{D}},\left\{ \beta_{\delta}\right\} _{\delta\in\mathcal{S}}}\left(\lf,\vd\right).
\end{equation}

The following lemma shows that for given $\vec{a},\vec{\epsilon},~E,F$ of \eqref{eq:F,E} are independent of the choices made in the definition of the coherent integrand, and calculates them.
\begin{lem}\label{lem:lem}
(a) For a sphere moduli specification $\left(\lf,d\right)$, we have
\[
\frac{2u}{d}E(\lf,d,\vec{a},\vec{\epsilon})=\int_{\overline{\mm}_{0,\lf}\left(d\right)}\beta_{\lf,d}=I\left(\lf,d,\vec{a},\vec{\epsilon}\right),
\]
where $I(S,\vec{a},\vec{\epsilon})$ is given in Definition \ref{def:inhomterms}. In particular, for fixed $\lf,d,\vec{a},\vec{\epsilon},$ $E(\lf,d,\vec{a},\vec{\epsilon})$ is independent of choices.

(b) For a disk moduli specification $\left(\lf,\vd\right)$, we have
\[
F\left(\lf,\vd,\vec{a},\vec{\epsilon}\right)=I\left(\lf,\vd,\vec{a},\vec{\epsilon}\right).
\]
In particular, for fixed $\lf,\vd,\vec{a},\vec{\epsilon}$ it is independent of choices.

(c) Let $S$ be a moduli specification with several connected components. Then we have
\[
I\left(S,\vec{a},\vec{\epsilon}\right)=\prod I\left(S',\vec{a}|_{S'},\vec{\epsilon}|_{S'}\right),
\]
where the product is taken over the connected components and $\vec{a}|_{S'},\vec{\epsilon}|_{S'}$ are the restrictions of $\vec{a},\vec{\epsilon}$ to the labels of $S'.$
\end{lem}
\begin{proof}
Part (a) is standard \cite{original-loc,kontsevich-torus-loc,pandharipande-virtual-loc},
and part (c) is immediate, so we prove the formula of part~(b). The independence of choices is clear from the formula.

The map \eqref{eq:open-closed relation} induces a map of fixed-point stacks
\[
\overline{\mm}_{0,\emptyset,\lf}\left(\vd\right)^{S^{1}}\to\left(\overline{\mm}_{0,\lf\times\left\{ 1,2\right\} }\left(\sum\vd\right)^{S^{1}}\right).
\]
We find that this map factors through
\[
\overline{\mm}_{0,\emptyset,\lf}\left(\vd\right)^{S^{1}}\to\left(\overline{\mm}_{0,\lf\times\left\{ 1,2\right\} }\left(\sum\vd\right)^{S^{1}}\right)^{\zz/2},
\]
which, assuming $\lf\neq\emptyset$, is the inclusion of a clopen
component (see \cite[\S 3.2]{fp-loc-OGW} for a similar argument).

Now use the description of the fixed points of the moduli of stable
\emph{closed }maps to deduce that
\[
\overline{\mm}_{0,\emptyset,\lf}\left(\vd\right)\simeq\bigsqcup_{\Gamma}\left(\overline{\mm}_{\Gamma}\right)_{A_{\Gamma}},
\]
where $\Gamma$ ranges over isomorphism classes of fixed-point graphs
$\Gamma$ for $\left(\lf,\vd\right)_{\mathcal{D}}$.

To compute the normal bundle of $\overline{\mm}_{\Gamma}\to\overline{\mm}_{0,\emptyset,\lf}\left(\vd\right)$,
we proceed as in \cite[\S3.2]{fp-loc-OGW}, and use the orientation results of Section \ref{sec:or}. Using the short exact sequence
\eqref{eq:fibered product-internal node}, the short exact sequence in Lemma \ref{lem:induced for internal nodes}, and the well-known formula for the normal bundle to the fixed points in the closed case \cite{pandharipande-virtual-loc}, we see that we only need to compute the normal bundle to an $S^1$-invariant smooth disk-map, justifying the contribution of the disk edges in the last factor of \eqref{eq:euler^-1}.

More precisely, write $\mm_{+}=\overline{\mm}_{0,\emptyset,\emptyset}\left(\left(d,0\right)\right)$
and $\mm_{-}=\overline{\mm}_{0,\emptyset,\emptyset}\left(\left(0,d\right)\right)$.
We denote $\mu=\mu\left(h\right)\in\left\{ \pm\right\} $. There is
a unique $S^{1}$~fixed point $q_{\mu}\in\mm_{\mu}$, represented
by a disk mapping to a hemisphere with one branch point at $p_{\mu}$
with ramification profile a single cycle of length $d$. We need to show that
\begin{equation}
e_{\mu}=e^{S^{1}}\left(T_{q_{\pm}}\mm_{\pm}\right)=\frac{\mu}{d!}\left(\mu\frac{2u}{d}\right)^{d-1}.\label{eq:half edge euler}
\end{equation}

Let $q\in\overline{\mm}_{0,\emptyset}\left(d\right)$ denote the double
of $q_{\mu}$, which is the unique fixed point of the moduli of stable
disk-maps with smooth domain. The involution that conjugates the map
acts on $T_{q}\overline{\mm}_{0,\emptyset}\left(d\right)$ as an anti-holomorphic
involution and we have
\[
T_{q_{\mu}}\mm_{\pm}=\left(T_{q}\overline{\mm}_{0,\emptyset}\left(d\right)\right)^{\zz/2}
\]
as $S^{1}$-representations, so
\[
e_{\mu}^{2}=e^{S^{1}}\left(T_{q}\overline{\mm}_{0,\emptyset}\left(d\right)\right)=\left(-1\right)^{d}\frac{\left(d!\right)^{2}\left(2u\right)^{2d-2}}{d^{2d-2}},
\]
which gives the value of $e$ up to a sign:
\[
e=s\cdot\frac{d^{d-1}}{d!\left(2u\right)^{d-1}},\quad s=s\left(\mu\right)\in\left\{ \pm1\right\} .
\]
Let us compute $s$. We have a diagram
\[
\mm_{\pm}\to\left(D^{2}\right)^{d-1}/\Sym\left(d-1\right)\leftarrow\left(D^{2}\right)^{d-1},
\]
where both maps are equivariant, holomorphic, non-constant maps between spaces of the same dimension: the left map is induced from the branch
divisor, the right map is the quotient. This implies that if we orient $T_{q_{\pm}}\mm_{\pm}$ using the complex orientation, then
\begin{equation*}
e^{S^{1}}\left(T_{q_{\pm}}\mm_{\pm}\right)=c\cdot e^{S^{1}}\left(T_{\left(0,...,0\right)}\left(D^{2}\right)^{d-1}\right),
\end{equation*}
where $c\in\qq$ is \emph{positive}.
The sign of $e^{S^{1}}\left(T_{\left(0,...,0\right)}\left(D^{2}\right)^{d-1}\right)$
is $\left(\mu\right)^{d-1}$; since, by \eqref{eq:or}, the orientation of $T_{q_{\pm}}\mm_{\pm}$ is twisted by $\operatorname{sgn}\left(d_{-}\leq d_{+}\right)$ relative to the complex orientation, we conclude that $s=\mu^{d}$. Equation \eqref{eq:half edge euler} is proved.
\end{proof}

\begin{proof}[Proof of Theorem \ref{thm:int_nums_equal_tree_sum}]
The 'Moreover' part follows immediately from applying Lemma \ref{lem:lem} to the calculation of the amplitudes $\AAH\left(T\right)$ in \eqref{eq:fp formula}. The passage from the trees with labeled oriented edges of $\bigsqcup_r\hts(r,\lf,\vd)$ to trees of $\TTT(\lf,\vd)$ without this additional data is responsible to the difference in combinatorial factors, by standard Orbit-Stabilizer argument. The independence of choices is a consequence of the 'Moreover' part.

The vanishing statement follows immediately from the fact the equivariant integral vanishes on forms whose de Rham degree is less than the dimension of the domain, and this is the case for $\<\prod_{i\in\lf}\tau^{\epsilon_i}_{a_i}\>_{0,\vd}$ whenever $1+\sum_{i\in\lf}  a_i   < d^++d^-.$

Note that, because fixed points components can be of any dimension, the individual contributions to $\OGW(\lf,\vd,\vec{a},\vec{\epsilon})$ don't have to vanish, and so the vanishing provides a non-trivial relation.
\end{proof}

We end this section with a simplification of the formula \eqref{eq:nice_tree_sum} for $\OGW.$
Recall the famous theorem of Cayley \cite{Cayley}, in its weighted version.
\begin{thm}
Associate a number $x_v$ to every vertex $v \in [n]$, and associate the monomial $\prod_{v \in [n]}x_v^{\val_T(v)}$ to every vertex-labeled tree $T$ on the set of vertices $[n].$ Then we have
\[
\sum_T \prod_{i \in [n]}x_i^{\val_T(i)} = \prod x_i\left(\sum_{i \in [n]}{x_i}\right)^{n-2},
\]
where the summation is taken over all trees with vertices labeled by $[n].$
\end{thm}

Let $\PPP(\lf,\vd)$ be the set of \emph{ordered partitions} of $(\lf,\vd)$ meaning sets of disk specifications \[((\lf_1,\vd_1),\ldots,(\lf_r,\vd_r))~~\text{with}~~
\bigsqcup \lf_i=\lf,~~\sum\vd_i=\vd.\]
Given $\vec{a}$ and $\vec{\epsilon}$ we define the \emph{amplitude} of $P=((\lf_1,\vd_1),\ldots,(\lf_r,\vd_r))$
to be
\[\AAA^\PPP(P,\vec{a},\vec{\epsilon}) =
\frac{1}{(-2u)^{r-1}r!}(d^{+}-d^{-})^{r-2}\prod_{i\in r}(d^{+}_i-d^{-}_i).
\prod_{i=1}^r{I(\lf_i,\vd_i,\vec{a}|_{\lf_i},\vec{\epsilon}|_{\lf_i})}.
\]
\begin{cor}\label{cor:cayley}
We have
\[
\OGW(\lf,\vd,\vec{a},\vec{\epsilon})=
\sum_{P\in\PPP(\lf,\vd)}\AAA^\PPP(P,\vec{a},\vec{\epsilon})+\delta_{d^+-d^-}\frac{d^+}{2u}I(\lf,|\vd|,\vec{a},\vec{\epsilon}).\]
\end{cor}
The proof is immediate from applying Cayley's theorem to formula \eqref{eq:nice_tree_sum}. The difference in automorphism factors follows easily from the fact that trees in $\TTT(\lf,\vd)$ are not vertex-labeled.

\subsection{Divisor and TRR for genus $0$ disk covers}\label{subsec:recursions}
In this section we prove the divisor relation, Theorem \ref{prop:divisor}, and the topological recursion for disk covers, Theorem \ref{thm:open_stat_trr_disk_cover}.
\begin{proof}[Proof of Theorem \ref{prop:divisor}]
Let $S$ be either a sphere moduli specification $(\lf,d)$ or a disk moduli specification $(\lf,\vd)$, $S'$ be the moduli specification obtained by extending $\lf$ to $\lf'=\lf\sqcup\{1\},$ without changing the degree. Let $\vec{a},\vec{\epsilon}$ be vectors of descendents and constraints, respectively, for $\lf,$ and $\vec{a'},\vec{\epsilon'}$ be their extensions to $\lf'$ obtained by defining $a'_i=a_i,\epsilon'_i=\epsilon_i$ for $i\in\lf$ and $a'_1=0,\epsilon'_1=\pm.$ We assume $\epsilon_1=+$ and denote in the disk case $d^+$ by $d.$ The case $\epsilon_1=-$ is treated in an analogous way. We will first show, and then use to derive the theorem, that
\begin{equation}\label{eq:divisor_fp}
\frac{I(S',\vec{a'},\vec{\epsilon'})}{|\Aut(S')|}=d\frac{I(S,\vec{a},\vec{\epsilon})}{|\Aut(S)|}+\frac{2u}{|\Aut(S)|}\sum_{i>1|\epsilon_i=+} I(S,\vec{a_{(i)}},\vec{\epsilon}),
\end{equation}
where the fixed point contribution $I$ is given by \eqref{eq:I_S} and $\vec{a_{(i)}}$ is the vector which is obtained from $\vec{a}$ by decreasing $a_i$ by $1.$
There is an obvious forgetful map $\for_1$ from fixed point graphs for $S'$ to those of $S.$ Equation \eqref{eq:divisor_fp} will follow if we could show that for any fixed point graph $\Gamma$ for $S$ we have
\begin{equation}\label{eq:divisor_graph}
\sum_{\Gamma'\in \for^{-1}(\Gamma)}\frac{I(\Gamma',\vec{a'},\vec{\epsilon'})}{|\Aut(\Gamma')|}=d\frac{I(\Gamma,\vec{a},\vec{\epsilon})}{|\Aut(\Gamma)|} +2u\sum_{i>1|\epsilon_i=+} \frac{I(\Gamma,\vec{a_{(i)}},\vec{\epsilon})}{|\Aut(\Gamma)|}.
\end{equation}
Observe that $\for_1^{-1}(\Gamma)$ is in bijection with the orbits of the action of $\Aut(\Gamma)$ on $V(\Gamma),$ where $\Gamma'$ corresponds to $[v]\in V(\Gamma)/\Aut(\Gamma)$ if it is obtained from $\Gamma$ by adding the tail marked $1$ to some $v\in[v]$ with $\mu(v)=+.$
Moreover, $\Aut(\Gamma')$ in this case is isomorphic to the subgroup of $\Aut(\Gamma)$ which fixes some $v\in[v],$ hence
\begin{equation}\label{eq:divisor_aut}\frac{|\Aut(\Gamma)|}{|\Aut(\Gamma')|}=|\Orb(v)|,\end{equation}
the size of the orbit of $v$ under $\Aut(\Gamma).$

Recall \eqref{eq:euler^-1} and \eqref{eq:alpha_term}. The analysis for $\Gamma'$ which corresponds to $[v],$ depends on $[v]$ as follows.
\begin{enumerate}[(a)]
\item If $\lambda(v)=\emptyset,~\val{v}=\{f\}$ then after adding the marking $1$ to $v,$ we see that
\[e_{\Gamma'}^{-1}=\frac{|\delta(f)|}{2u}e_{\Gamma}^{-1},~~
\alpha_{\Gamma'}^{\vec{a'},\vec{\epsilon'}}=2u\alpha_{\Gamma}^{\vec{a},\vec{\epsilon}}\Rightarrow I(\Gamma',\vec{a'},\vec{\epsilon'})=|\delta(f)|I(\Gamma,\vec{a},\vec{\epsilon}).\]
The first equality follows from the change in the fourth term of \eqref{eq:euler^-1} which corresponds to $v$ and the second equality follows from the change in the second term of \eqref{eq:alpha_term} which corresponds to $v.$
\item If $\lambda(v)=\emptyset,~\val{v}=\{f_1,f_2\}$ then after adding the marking $1$ to $v,~v$ becomes a vertex in $V_+$ which represents the moduli $\overline{\mm}_{0,3},$ with the three markings $1$ and the two nodes which correspond to $f_1,f_2$.
    We have
\[e_{\Gamma'}^{-1}=\frac{(2u)(\omega_{f_1}+\omega_{f_2})}{\omega_{f_1}\omega_{f_2}}e_{\Gamma}^{-1},~~
\alpha_{\Gamma'}^{\vec{a'},\vec{\epsilon'}}=(2u)\alpha_{\Gamma}^{\vec{a},\vec{\epsilon}}\Rightarrow I(\Gamma',\vec{a'},\vec{\epsilon'})=(|\delta(f_1)|+|\delta(f_2)|)I(\Gamma,\vec{a},\vec{\epsilon}),\]
where the change in the inverse Euler term comes from the first three terms in \eqref{eq:euler^-1}.
\item If $\lambda(v)=\{i\},$ with $a_i=0$ then after adding the marking $1$ to $v,~v$ becomes a vertex in $V_+$ which represents the moduli $\overline{\mm}_{0,3},$ with the three markings $1,i$ and the node.
The $\alpha$ term \eqref{eq:alpha_term} changes by $(2u)$ again, and if we write $\val{v}=\{f\}$ we see that
\[e_{\Gamma'}^{-1}=\frac{|\delta(f)|}{2u}e_{\Gamma}^{-1}\Rightarrow I(\Gamma',\vec{a'},\vec{\epsilon'})=|\delta(f)|I(\Gamma,\vec{a},\vec{\epsilon}),\]
where the change comes from the first term of \eqref{eq:euler^-1} which corresponds to $f,$ expanded to zeroth order in $\psi_f,$ and the $(2u)^{-1}$ is contributed by the second term of \eqref{eq:euler^-1}.
Note that in this case, before adding $1,$ $|\Orb(v)|=1$ since it carries the $i^{th}$ label.
\item Suppose now $\lambda(v)=\{i\},$ with $a_i>0.$ Write $\val{v}=\{f\}$. Again $|\Orb(v)|=1.$ After adding the marking $1$ to $v,~v$ becomes a vertex in $V_+$ which represents the moduli $\overline{\mm}_{0,3},$ with the three markings $1,i$ and the node. Since $a_i>0$ but $\overline{\mm}_{0,3}$ is zero dimensional, the $\alpha$ term, and hence also $I(\Gamma',\vec{a'},\vec{\epsilon'})$ vanish.
However, comparing the second expressions in $\alpha_\Gamma^{\vec{a},\vec{\epsilon}}$ and in $\alpha_\Gamma^{\vec{a_{(i)}},\vec{\epsilon}}$
\[|\delta(f)|I(\Gamma,\vec{a},\vec{\epsilon})=-(2u)I(\Gamma,\vec{a_{(i)}},\vec{\epsilon})\Rightarrow
I(\Gamma',\vec{a'},\vec{\epsilon'})=|\delta(f)|I(\Gamma,\vec{a},\vec{\epsilon})+(2u)I(\Gamma,\vec{a_{(i)}},\vec{\epsilon}).\]
\item The last case is that $v$ represents a contracted component with $\lambda(v)=\lf_1$ and $\val{v}=\{f_1,\ldots,f_s\}.$
From combining the first terms in \eqref{eq:euler^-1},\eqref{eq:alpha_term} we see that the multiplicative contribution of $v$ before adding $1$ is
\[\sum_{\substack{b_1,\ldots, b_s\geq 0,\\
\sum b_j=|\lf_1|+s-3-\sum_{i\in\lf_1}a_i  }}\prod_{j=1}^s\frac{|\delta(f_j)|^{b_j+1}}{(2u)^{b_j}} \<\prod_{i\in \lf_1}\tau_{a_i}\prod_{j=1}^s\tau_{b_j}\>^c,\]
where $\<\cdots\>^c$ stands for closed intersection numbers over the moduli of stable marked spheres.
Adding the marking $1$ changes the contribution of $v$ to
\begin{align}\label{eq:string_simplification_for_div_5_item}
&(2u)\sum_{\substack{b_1,\ldots, b_s\geq 0,\\
\sum b_j=|\lf_1|+s-2-\sum_{i\in\lf_1}a_i  }}\prod_{j=1}^s\frac{|\delta(f_j)|^{b_j+1}}{(2u)^{b_j}} \<\tau_0\prod_{i\in \lf_1}\tau_{a_i}\prod_{j=1}^s\tau_{b_j}\>^c=\\
\notag=&
(2u)\sum_{\substack{b_1,\ldots, b_s\geq 0,\\
\sum b_j=|\lf_1|+s-2-\sum_{i\in\lf_1}a_i  }}\prod_{j=1}^s\frac{|\delta(f_j)|^{b_j+1}}{(2u)^{b_j}}
\left(\sum_{k\in\lf_1}\<\tau_{a_k-1}\prod_{i\in \lf_1,i\neq k}\tau_{a_i}\prod_{j=1}^s\tau_{b_j}\>^c
+\sum_{k\in[s]}\<\tau_{b_k-1}\prod_{i\in \lf_1,i\neq k}\tau_{a_i}\prod_{j\in[s]\setminus\{k\}}\tau_{b_j}\>^c\right),
\end{align}
where we used the string equation
\begin{equation}\label{eq:string}\<\tau_0\prod_{i=1}^n\tau_{a_i}\>^c=\sum_{j=1}^n\<\tau_{a_j-1}\prod_{i\in[n]\backslash\{j\}}\>^c.
\end{equation}
Substituting \eqref{eq:string_simplification_for_div_5_item} in the expression for $I(\Gamma',\vec{a'},\vec{\epsilon'})$ we obtain
\begin{equation}
I(\Gamma',\vec{a'},\vec{\epsilon'})=(2u)\sum_{i\in\lf_1}I(\Gamma,\vec{a_{(i)}},\vec{\epsilon})+
\left(\sum_{i=1}^s|\delta(f_i)|\right)I(\Gamma,\vec{a},\vec{\epsilon}).
\end{equation}
\end{enumerate}
Summing the above items over all $[v]\in V(\Gamma)/\Aut(\Gamma),$ where the term for $[v]$ is taken with multiplicity $|\Orb(v)|,$ using \eqref{eq:divisor_aut} and the definition of the fixed-point contributions \eqref{eq:I_Gamma} for \\$(\Gamma',\vec{a'},\vec{\epsilon'}),(\Gamma,\vec{a},\vec{\epsilon}),(\Gamma,\vec{a_{(i)}},\vec{\epsilon})$ for the different $i,$ we obtain \eqref{eq:divisor_graph} and hence \eqref{eq:divisor_fp}.

We will now use \eqref{eq:divisor_fp} to derive the divisor relation. Write $\lf' = [l],\lf=\{2,\ldots,l\}.$ We want to prove that \[\OGW(\lf',\vd,\vec{a'},\vec{\epsilon'})=d^\pm \OGW(\lf,\vd,\vec{a},\vec{\epsilon})+(2u)\sum_{i\in\lf|\epsilon_i=\epsilon_1}\OGW(\lf,\vd,\vec{a_{(i)}},\vec{\epsilon}),\]
where again $\vec{a},\vec{\epsilon}$ are vectors of descendents and constraints respectively for $\lf,$ and $\vec{a'},\vec{\epsilon'},\vec{a_{(i)}}$ are defined as in the previous part of the proof.

We begin with considering the first term of \eqref{eq:nice_tree_sum} for $\OGW.$ We would like to show that
\begin{equation}\label{eq:divisor_trees}\sum_{T\in\TTT(\lf',\vd')}\AAA(T,\vec{a'},\vec{\epsilon'})=\sum_{T\in\TTT(\lf,\vd)}
\left(d^\pm\AAA(T,\vec{a},\vec{\epsilon})+
2u\sum_{i\in\lf|\epsilon_i=\epsilon_1}\AAA(T,\vec{a_{(i)}},\vec{\epsilon})\right).
\end{equation}
Denote by $\widetilde{\for}_1:\ts(\lf',\vd)\to\ts(\lf',\vd)$ the obvious forgetful map. Equation \eqref{eq:divisor_trees} will follow if we could show that for any $T\in\ts(\lf,\vd)$ we have
\[
\sum_{T'\in \widetilde{\for}_1^{-1}(T)}\AAA(T',\vec{a'},\vec{\epsilon'})=d^\pm \AAA(T,\vec{a},\vec{\epsilon})+2u\sum_{i\in\lf|\epsilon_i=\epsilon_1}\AAA(T,\vec{a_{(i)}},\vec{\epsilon}).
\]
However, elements $T'\in \widetilde{\for}_1^{-1}(T)$ are in bijection with the equivalence classes of $[v]\in V(T)/\Aut(T),$ which specify to which vertex the marking $1$ will be sent. For $T'$ which corresponds to $[v],$ by \eqref{eq:divisor_fp} and the definition of the amplitude, equation \eqref{eq:ampli_intro_g=0},
\[\AAA(T',\vec{a'},\vec{\epsilon'}) = \frac{|\Aut(T)|}{|\Aut(T')|}\left(d^\pm_v\AAA(T,\vec{a'},\vec{\epsilon'})+2u\sum_{i\in\lf_v|\epsilon_i=\epsilon_1}
\AAA(T,\vec{a_{(i)}},\vec{\epsilon})\right).\]
The group $\Aut(T')$ is canonically identified with the subgroup $\Aut(T,v) $ of $\Aut(T)$ preserving some $v\in[v].$ Therefore $\frac{|\Aut(T)|}{|\Aut(T')|}=|\Orb(v)|,$ the size of the orbit of $v$ under $\Aut(T),$ and in particular $\Aut(T')\simeq\Aut(T)$ when $\lf_v\neq\emptyset.$
Thus,
\[\sum_{[v]\in V(T)/\Aut(T)} |\Orb(v)|d_v^\pm=\sum_{v\in V(T)}d_v^\pm=d^\pm,\] and equation~\eqref{eq:divisor_trees} follows.

If $d^+=d^-$, then the expressions of \eqref{eq:nice_tree_sum} for $\OGW(\lf',\vd,\vec{a'},\vec{\epsilon'}), \OGW(\lf,\vd,\vec{a},\vec{\epsilon})$ include the second (closed term). Applying \eqref{eq:divisor_fp} to $S=(\lf',d),\vec{a},\vec{\epsilon}$ we see immediately that the second term satisfies the expected relation, and the divisor equation follows.

The 'In particular' part follows from taking the non-equivariant limit (order $u^0$ invariants) and using the fact that terms with negative powers of $u$ vanish, by the 'In particular' part of Theorem~\ref{thm:int_nums_equal_tree_sum}.\end{proof}

%%%An alternative proof for the stationary case
%The proof of Theorem \ref{thm:loc for CP1,RP1} shows that in order to calculate $\<\prod_{i=2}^l\tau_{a_i}^{\epsilon_i}\>_{0,\vd}%^{\circ}
%$ one needs only to have a coherent integrand bounded by $\lf$ (defined in the end of Definition \ref{def:coherent integrand}). One may extend a coherent integrand bounded by $\lf$ to a coherent integrand bounded by $\lf'$ by defining $\alpha'_{\lf'',\vd}=\alpha_{\lf'',\vd}$ if $1\notin\lf''$ and otherwise by $f_1^*\alpha_{\lf''\setminus\{1\},\vd}\ev_1^*\rho_{\epsilon_1},$ and act similarly for $\beta'_{*,*},$ where $f_1$ is the map which forgets the first marking. The fact that the resulting family of forms is a coherent integrand bounded by $\lf$ is straightforward from the definition of coherent integrands and Lemma \ref{lem:d-forget cartesian}.
%\[\int_{\overline{\mm}_{0,\emptyset,\lf'}\left(\vd\right)}\alpha'_{\lf',(d,0)}=
%\int_{\overline{\mm}_{0,\emptyset,\lf'}\left(\vd\right)}f_1^*\alpha_{\lf,\vd}\ev_1^*\rho_{\epsilon_1} = d^{\epsilon_1}\int_{\overline{\mm}_{0,\emptyset,\lf}\left(\vd\right)}\alpha_{\lf,\vd},\]
%where the last equality follows from integration along the fiber of $f_1.$

\begin{proof}[Proof of Theorem \ref{thm:open_stat_trr_disk_cover}]
We work with $\lf=[l].$ Write $l_0=|\{i>2|a_i=0\}|.$
We first consider the case $l_0=0.$ In this case the first term on the RHS of \eqref{eq:trr_disk_disk} doesn't contribute, for degree reasons. By Theorem \ref{thm:int_nums_equal_tree_sum}, we have
\[
\<\tau^+_{a_1+1}\prod_{i=2}^l\tau^{+}_{a_i}\>_{0,(d,0)}=
\sum_{n\geq 1}\sum_{{T}\in \TTT\left([l],(d,0)\right)}\AA\AAH\left({T},\vec{a'},\vec{\epsilon}\right),
\]
where $\vec{a'}=\vec{a}+(1,0,\ldots,0).$

For any tree $T$ which appears in the expression above and has a non-zero contribution, and any vertex $v$ of $T,$ it must hold that $\lf(v)$ is either empty or a singleton ($\lf(v)$ and the other tree notations are defined in Definition \ref{def:trees}). Write $\TTT'\left([l],(d,0)\right)$ for the subset of $\TTT\left([l],(d,0)\right)$ whose elements are trees satisfying this constraint. Consider an arbitrary ${T}\in \TTT'\left([l],(d,0)\right).$ Let $v$ be the vertex with $\lf(v)=\{1\}.$ There is a unique path in the tree between $v$ and the vertex $u$ with $\lf(u)=\{2\}.$ Let $e$ be the first edge of this path, starting from $v.$ Suppose it connects $v$ to some other vertex $w.$ Erasing $e$ divides ${T}$ into two trees, ${T}_1$ which contains $v$ and ${T}_2$ which contains $u.$ Denote by $R\cup\{1\}$ the label set of ${T}_1$ and by $S\cup\{2\}$ the label set of ${T}_2.$ From the definition of the amplitude, equation~\eqref{eq:ampli_intro_g=0}, it follows that
\begin{equation}\label{eq:tree_dec}\AAA\left({T},\vec{a'},\vec{\epsilon}\right)=
\frac{d^{+}(v)}{-2u}d^{+}(w)\AAA({T}_1,\vec{a'}|_{R\cup\{1\}},
\vec{\epsilon}|_{R\cup\{1\}})\AAA({T}_2,\vec{a}|_{S\cup\{2\}},\vec{\epsilon}|_{S\cup\{2\}}).
\end{equation}
Note that this simple procedure gives a bijection
\[
{T}\rightarrow J({T})=({T}_1,{T}_2,w)
\]
between $\TTT'\left([l],(d,0)\right)$ and
\[
\bigsqcup_{\substack{d_1,d_2>0,~d_1+d_2=d,\\R,S\subseteq\{3,\ldots,l\},~R\sqcup S=\{3,\ldots, l\}}}
\TTT'\left(R\cup\{1\},(d_1,0)\right)\times \bigsqcup_{{T}_2\in\TTT'\left({S\cup\{2\}},(d_2,0)\right)}
 V({T}_2).
\]

For any ${T}_2\in\TTT'\left(\lf_{S\cup\{2\}},(d_2,0)\right)$ we have
\begin{equation}\label{eq:sum_w}
\sum_{w\in V({T}_2)} d^+(w) = d_2.
\end{equation}
In addition, by Theorem \ref{thm:int_nums_equal_tree_sum}, we have
\begin{equation}\label{eq:sum_T2}
\sum_{{T}_2\in\TTT'\left(\lf|_{S\cup\{2\}},(d_2,0)\right)}\AAA({{T}_2},\vec{a}|_{S\cup\{2\}},\vec{\epsilon}|_{S\cup\{2\}})
=\left\langle \tau^+_{a_2}\prod_{i \in R} \tau^+_{a_i}\right\rangle_{0,(d_2,0)}.
\end{equation}
Putting together, using \eqref{eq:tree_dec}, we see that for any $d_1,R$ and ${T}_1,$ it holds that
\begin{align}
\sum_{{T}_2\in\TTT'\left(\lf_{S\cup\{2\}},(d_2,0)\right)}&\sum_{w\in V({T}_2)} \AAA(J^{-1}({T}_1,{T}_2,w),\vec{a'},\vec{\epsilon})=\label{eq:sum_wT2}\\
=&\sum_{{T}_2\in\TTT'\left(\lf|_{S\cup\{2\}},(d_2,0)\right)}\sum_{w\in  V({T}_2)}
\left(\frac{d^{+}(v)}{-2u}\AAA({T}_1,\vec{a'}|_{R\cup\{1\}},\vec{\epsilon}|_{R\cup\{1\}})\right)
\left(d^{+}(w)\AAA({T}_2,\vec{a}|_{S\cup\{2\}},\vec{\epsilon}|_{S\cup\{2\}})\right)=\notag\\
\notag=&d_2\left(\frac{d^{+}(v)}{-2u}\AAA({T}_1,\vec{a'}|_{R\cup\{1\}},\vec{\epsilon}|_{R\cup\{1\}},\vec{a'}|_{R\cup\{1\}},\vec{\epsilon}|_{R\cup\{1\}})\right)\left\langle \tau^+_{a_2}\prod_{i \in S} \tau^+_{a_i}\right\rangle_{0,(d_2,0)}.\notag
\end{align}

Observe that, by the definition of the amplitude, we have
\[
\frac{d^{+}(v)}{-2u}\AAA({T}_1,\vec{a'}|_{R\cup\{1\}},\vec{\epsilon}|_{R\cup\{1\}})=
\AAA({T}_1,\vec{a}|_{R\cup\{1\}},\vec{\epsilon}|_{R\cup\{1\}}).
\] %has the effect of reducing the $\psi$ power of the first marking by $1.$
Thus, using Theorem \ref{thm:int_nums_equal_tree_sum} again, we get, for fixed $d_1,R,$
\begin{equation}
\label{eq:2_for_trr_disk_disk}
\sum_{{T}_1\in \TTT'\left(\lf|_{R\cup\{1\}},(d_1,0)\right)}\frac{d^{+}(v)}{-2u}{\AAA({T}_1,\vec{a'}|_{R\cup\{1\}},\vec{\epsilon}|_{R\cup\{1\}})}=
\sum_{T_1\in\TTT'\left(\lf|_{R\cup\{1\}},(d_1,0)\right)}{\AAA({T}_1,\vec{a}|_{R\cup\{1\}},\vec{\epsilon}|_{R\cup\{1\}})}=
\left\langle \tau^+_{a_1}\prod_{i \in R} \tau^+_{a_i}\right\rangle_{0,(d_1,0)}.
\end{equation}

Combining equations \eqref{eq:sum_wT2},\eqref{eq:2_for_trr_disk_disk}, summing over all $d_1,R$ and using the bijection $J$ we obtain \eqref{eq:trr_disk_disk}.

For the general TRR, we order the pairs $(l_0,l)$ lexicographically, so that $(l'_0,l')<(l_0,l)$ if either $l'_0<l_0$ or $l'_0=l_0$ but $l'<l.$
We use induction on $(l_0,l).$ We have proved \eqref{eq:trr_disk_disk} for $l_0=0.$ Suppose we have proved it for all $(l'_0,l')<(l_0,l).$ We now prove for $(l_0,l).$
Since $l_0\geq 1$, there is some $i>2$ with $a_i=0.$ Without loss of generality take it to be $l.$
Theorem \ref{prop:divisor} and the string equation \eqref{eq:string} allow us to write the closed-open term of the RHS of \eqref{eq:trr_disk_disk},
\begin{equation}\label{eq:co_trr}
\sum_{\substack{R\sqcup S = \{3,\ldots,l\}}}(2u)^{|R|}\<\tau_{a_1}\left(\prod_{i \in R} \tau_{a_i}\right)\tau_0\>_0^c\<\tau^+_0\tau^+_{a_2}\prod_{i \in S} \tau^+_{a_i}\>_{0,(d,0)},
\end{equation}
as the sum of four terms $Q_1+\ldots+Q_4,$ where
\begin{align*}
Q_1=&(2u)\delta_{a_1=0}\<\tau^+_0\prod_{i \neq 1,l} \tau^+_{a_i}\>_{0,(d,0)},\\
Q_2=&\sum_{\substack{R\sqcup S = \{3,\ldots,l-1\}}}(2u)^{|R|+1}\sum_{j\in R\cup\{1\}}\<\left(\prod_{i \in R\cup\{1\}}\tau_{a_i-\delta_{j=i}}\right) \tau_0\>_0^c\<\tau^+_0\tau^+_{a_2}\prod_{i \in S} \tau^+_{a_i}\>_{0,(d,0)},\\
Q_3=&d\sum_{\substack{R\sqcup S = \{3,\ldots,l-1\}}}(2u)^{|R|}\<\tau_{a_1}\left(\prod_{i \in R} \tau_{a_i}\right)\tau_0\>_0^c\<\tau^+_0\tau^+_{a_2}\prod_{i \in S} \tau^+_{a_i}\>_{0,(d,0)},\\
Q_4=&(2u)\sum_{\substack{R\sqcup S = \{3,\ldots,l-1\}}}\sum_{j\in S\cup\{2\}}(2u)^{|R|}\<\tau_{a_1}\left(\prod_{i \in R} \tau_{a_i}\right)\tau_0\>_0^c\<\tau^+_0\prod_{i \in S\cup\{2\}} \tau^+_{a_i-\delta_{j=i}}\>_{0,(d,0)}.
\end{align*}
Indeed, $Q_1$ is obtained from the case $R=\{l\}$ in \eqref{eq:co_trr}, $Q_2$ from applying the string equation to sets $R\neq\{l\}$ which contain $l,$ $Q_3,Q_4$ are obtained from applying the divisor equation to the case $l\in S.$ Similarly, the open-open term of the RHS of \eqref{eq:trr_disk_disk},
\begin{equation}\label{eq:oo_trr}
\sum_{\substack{R\sqcup S = \{3,\ldots,l\} \\ d_1 + d_2 = d}}d_2 \<\tau^+_{a_1} \prod_{i \in R} \tau^+_{a_i}\>_{0,(d_1,0)}%^{\circ}
\<\tau^+_{a_2}\prod_{i \in S} \tau^+_{a_i}\>_{0,(d_2,0)},
\end{equation}
can be written as the sum $Q_5+Q_6$, where
\begin{align*}
Q_5=&(d_1+d_2)\sum_{\substack{R\sqcup S = \{3,\ldots,l-1\} \\ d_1 + d_2 = d}}d_2 \<\tau^+_{a_1} \prod_{i \in R} \tau^+_{a_i}\>_{0,(d_1,0)}\<\tau^+_{a_2}\prod_{i \in S} \tau^+_{a_i}\>_{0,(d_2,0)},\\
Q_6=&(2u)\sum_{\substack{R\sqcup S = \{3,\ldots,l-1\} \\ d_1 + d_2 = d}}\sum_{j\in[l-1]}d_2 \<\prod_{i \in R\cup\{1\}} \tau^+_{a_i-\delta_{j=i}}\>_{0,(d_1,0)}\<\prod_{i \in S\cup\{2\}} \tau^+_{a_i-\delta_{j=i}}\>_{0,(d_2,0)},
\end{align*}
and both are obtained by applying the divisor equation to \eqref{eq:oo_trr} in the two cases $l\in R$ and $l\in S.$

By the divisor equation, the LHS of \eqref{eq:trr_disk_disk} is equal to
\begin{equation}\label{eq:div_for_trr}
d\<\tau^+_{a_1+1}\prod_{i=2}^{l-1}\tau^{+}_{a_i}\>_{0,(d,0)}+
(2u)\left(\<\tau^+_{a_1}\prod_{i=2}^{l-1}\tau^{+}_{a_i}\>_{0,(d,0)}+
\sum_{j=2}^{l-1}\<\tau^+_{a_1+1}\prod_{i=2}^{l-1}\tau^{+}_{a_i-\delta_{j=i}}\>_{0,(d,0)}
\right).\end{equation}
%\begin{equation}\label{eq:div_for_trr}
%d\<\tau^+_{a_i+1}\prod_{i=2}^{l-1}\tau^{+}_{a_i}\>_{0,(d,0)}+
%(2u)\left(\<\tau^+_{a_1}\prod_{i=2}^{l-1}\tau^{+}_{a_i}\>_{0,(d,0)}+
%\sum_{j=2}^{l-1}\<\tau^+_{a_1+1}\prod_{i=2}^{l-1}\tau^{+}_{a_i-\delta_{j=i}\<\tau^+_{a_i+1}\prod_{i=2}^l\tau^{+}_{a_i}\>_{0,(d,0)}}\>_{0,(d,0)}
%\right).\end{equation}
All summands in \eqref{eq:div_for_trr} are smaller in lexicographic order than $(l_0,l),$ hence we can rely on the induction and apply TRR to \eqref{eq:div_for_trr}.
We will obtain the sum of seven terms $P_1+\ldots+P_7$, where
\begin{align*}
P_1=&d\sum_{\substack{R\sqcup S = \{3,\ldots,l-1\}}}(2u)^{|R|}\<\tau_{a_1}\left(\prod_{i \in R} \tau_{a_i}\right)\tau_0\>_0^c\<\tau^+_0\tau^+_{a_2}\prod_{i \in S} \tau^+_{a_i}\>_{0,(d,0)},\\
P_2=&d\sum_{\substack{R\sqcup S = \{3,\ldots,l-1\} \\ d_1 + d_2 = d}}d_2 \<\tau^+_{a_1} \prod_{i \in R} \tau^+_{a_i}\>_{0,(d_1,0)}%^{\circ}
\<\tau^+_{a_2}\prod_{i \in S} \tau^+_{a_i}\>_{0,(d_2,0)},\\
P_3=&\delta_{a_1=0}(2u)\<\tau^+_{0}\prod_{i=2}^{l-1}\tau^{+}_{a_i}\>_{0,(d,0)},\\
P_4=&(2u)\sum_{\substack{R\sqcup S = \{3,\ldots,l-1\}}}(2u)^{|R|}\<\tau_{a_1-1}\left(\prod_{i \in R} \tau_{a_i}\right)\tau_0\>_0^c\<\tau^+_0\tau^+_{a_2}\prod_{i \in S} \tau^+_{a_i}\>_{0,(d,0)},\\
P_5=&(2u)\sum_{\substack{R\sqcup S = \{3,\ldots,l-1\} \\ d_1 + d_2 = d}}d_2 \<\tau^+_{a_1-1} \prod_{i \in R} \tau^+_{a_i}\>_{0,(d_1,0)}%^{\circ}
\<\tau^+_{a_2}\prod_{i \in S} \tau^+_{a_i}\>_{0,(d_2,0)},\\
P_6=&(2u)\sum_{j=2}^{l-1}\sum_{\substack{R\sqcup S = \{3,\ldots,l-1\}}}(2u)^{|R|}\<\tau_{a_1}\left(\prod_{i \in R} \tau_{a_i-\delta_{i=j}}\right)\tau_0\>_0^c\<\tau^+_0\prod_{i \in S\cup\{2\}} \tau^+_{a_i-\delta_{i=j}}\>_{0,(d,0)},\\
P_7=&(2u)\sum_{j=2}^{l-1}\sum_{\substack{R\sqcup S = \{3,\ldots,l-1\} \\ d_1 + d_2 = d}}d_2 \<\tau^+_{a_1} \prod_{i \in R} \tau^+_{a_i-\delta_{i=j}}\>_{0,(d_1,0)}\<\prod_{i \in S\cup\{2\}} \tau^+_{a_i-\delta_{i=j}}\>_{0,(d_2,0)}.
\end{align*}
Here $P_1,P_2$ are obtained from applying TRR to the first term in \eqref{eq:div_for_trr}, $P_4,P_5$ are obtained from applying TRR to the second term, assuming $a_1\neq 0,$ $P_6,P_7$ are obtained from applying TRR to the last term. When $a_1=0$ we cannot apply TRR, and $P_3$ is the contribution in this case. The induction will follow if we could show that $Q_1+\ldots+Q_6=P_1+\ldots+P_7$. Indeed,
\[
Q_1=P_3,\quad Q_3=P_1,\quad Q_5=P_2,\quad Q_6=P_5+P_7,\quad Q_2+Q_4=P_4+P_6,
\]
and the theorem follows.

The 'In particular' part follows from taking the non-equivariant limit (order $u^0$ invariants) and using the fact that terms with negative powers of $u$ vanish, by the 'In particular' part of Theorem~\ref{thm:int_nums_equal_tree_sum}. In this case also the closed term $\<\>_c$ vanishes.
\end{proof}

\subsection{Genus $0$ disk cover formula and genus $0~~(d,d)-$vanishing}\label{subsec:formulas}
\begin{proof}[Proof of Theorem \ref{thm:disk_disk_covers_formula}]
%\eqref{eq:disk_disk_covers_formula_minus} follows from \eqref{eq:disk_disk_covers_formula} by Lemma \ref{lem:(d,d)_vanishing} below, so we will concentrate on proving \eqref{eq:disk_disk_covers_formula}.
For shortness, throughout the proof we will omit the superscripts $+$ in the symbols $\tau_{a_i}^+,$ as well as the genus and other superscripts and subscripts which are not the degree, since we consider disks which are mapped to upper hemisphere. Stationary invariants satisfy the following three properties:
\begin{align}
&\<\>_{(1,0)}=1,\label{eq:initial condition}\\
&\<\tau_0\prod_{i=1}^n\tau_{a_i}\>_{(1+a_{[n]},0)}=(1+a_{[n]})\<\prod_{i=1}^n\tau_{a_i}\>_{(1+a_{[n]},0)},&& n\ge 0,\label{eq:open divisor}\\
&\<\tau_{a_1+1}\prod_{i=2}^n\tau_{a_i}\>_{(2+a_{[n]},0)}=\sum_{\substack{I\sqcup J=[n]\\1\in I,\, n\in J}}(1+a_J)\<\prod_{i\in I}\tau_{a_i}\>_{(1+a_I,0)}\<\prod_{j\in J}\tau_{a_j}\>_{(1+a_J,0)},&& n\ge 2.\label{eq:open TRR0}
\end{align}
Here for an $n$-tuple $a_1,\ldots,a_n$ and a subset $I\subset [n]$ we use the notation $a_I=\sum_{i\in I}a_i$. Note that \eqref{eq:open TRR0} is a special case of Theorem \ref{thm:open_stat_trr_disk_cover}, and \eqref{eq:open divisor} is a special case of Theorem \ref{prop:divisor}.

We prove~\eqref{eq:disk_disk_covers_formula} by induction on $n$ and on $\sum a_i$. The case $n=0$ follows from~\eqref{eq:initial condition}. Suppose $n=1$. By~\eqref{eq:open divisor} and~\eqref{eq:initial condition}, $\<\tau_0\>_{(1,0)}=1$. Suppose that $a\ge 1$. Then we have
$$
\<\tau_a\>_{(a+1,0)}\stackrel{\text{by \eqref{eq:open divisor}}}{=}\frac{1}{a+1}\<\tau_a\tau_0\>_{(a+1,0)}\stackrel{\text{by~\eqref{eq:open TRR0}}}{=}\frac{1}{a+1}\<\tau_{a-1}\>_{(a,0)}\<\tau_0\>_{(1,0)}=\frac{1}{a+1}\<\tau_{a-1}\>_{(a,0)}.
$$
Therefore, $\<\tau_a\>_{(a+1,0)}=\frac{1}{(a+1)!}$.

Suppose $n\ge 2$. If all $a_i$'s are zero, then $\<\tau_0^n\>_{(1,0)}=1$. Suppose some of $a_i$'s are not zero. Without loss of generality we can assume that $a_1\ge 1$. Then, by~\eqref{eq:open TRR0}, we have
\begin{align*}
\<\prod_{i=1}^n\tau_{a_i}\>_{(1+a_{[n]},0)}=&\sum_{\substack{I\sqcup J=\{2,\ldots,n\}\\n\in J}}(1+a_J)\<\tau_{a_1-1}\prod_{i\in I}\tau_{a_i}\>_{(a_1+a_I,0)}\<\prod_{j\in J}\tau_{a_j}\>_{(1+a_J,0)}\stackrel{\substack{\text{by the induction}\\\text{assumption}}}{=}\\
=&\frac{1}{\prod a_i!}\sum_{\substack{I\sqcup J=[n]\\1\in I,\,n\in J}}a_1a_I^{|I|-2}(1+a_J)^{|J|-1}.
\end{align*}
It remains to prove that for any $n\ge 2$ the following identity is true:
\begin{gather}\label{eq:identity for disk to disk}
\sum_{\substack{I\sqcup J=[n]\\1\in I,\,n\in J}}a_1a_I^{|I|-2}(1+a_J)^{|J|-1}=(1+a_{[n]})^{n-2},
\end{gather}
where we consider $a_1,\ldots,a_n$ as formal variables.

We prove identity~\eqref{eq:identity for disk to disk} by induction on $n$. Both sides of it are polynomials in $a_1,\ldots,a_n$. Denote the left-hand side by $L_n(a_1,\ldots,a_n)$. For $n=2$ identity~\eqref{eq:identity for disk to disk} is trivial. Suppose $n\ge 3$. Since both sides of~\eqref{eq:identity for disk to disk} have degree $n-2$, it is sufficient to check the following two properties:
\begin{enumerate}[\textbullet]
\item $L_n|_{a_i=0}=(1+\sum_{j\ne i}a_j)^{n-2}$, for any $2\le i\le n-1$.
\item Coefficient of $a_2a_3\cdots a_{n-1}$ in $L_n$ is equal to $(n-2)!$.
\end{enumerate}
For this we compute
\begin{align*}
L_n|_{a_i=0}=&\sum_{\substack{I\sqcup J=[n]\backslash\{i\}\\1\in I,\,n\in J}}\left(a_1a_I^{|I|-2}(1+a_J)^{|J|-1}\right)(a_I+1+a_J)=\\
=&L_{n-1}(a_1,\ldots,\widehat{a_i},\ldots,a_n)\left(1+\sum_{j\ne i}a_j\right)\stackrel{\substack{\text{by the induction}\\\text{assumption}}}{=}\left(1+\sum_{j\ne i}a_j\right)^{n-2}.
\end{align*}
Clearly, the coefficient of $a_2\cdots a_{n-1}$ in $a_1a_I^{|I|-2}(1+a_J)^{|J|-1}$ is equal to zero, unless $|I|=1$. If $|I|=1$, then this coefficient is equal to $(n-2)!$. This completes the proof of identity~\eqref{eq:identity for disk to disk} and, hence, the theorem is proved.
\end{proof}

\begin{lemma}\label{lem:(d,d)_vanishing}
%\[\<\prod_{i=1}^l\tau_{a_i}^{\epsilon_i}\>_{0,(d_1,d_2)}%^{\circ}
%=-\<\prod_{i=1}^l\tau_{a_i}^{-\epsilon_i}\>_{0,(d_2,d_1)}%^{\circ}
%.\]
%In particular, a
The genus $0$ stationary intersection numbers for degree $(d,d)$ vanish.
\end{lemma}

\begin{proof}
%The lemma is an immediate consequence of the simple observation that inversion on the target is orientation reversing for the moduli, by Observation~\ref{obs:behavior of orientation under inversion}, but orientation preserving for the tautological lines and the point constraints.\footnote{A sanity check: can this be seen from the formula as well? make sure!.}
%
%For $\vd=(d,d)$ we provide an alternative proof, using Theorem \ref{thm:int_nums_equal_tree_sum} and Corollary \ref{cor:cayley}.
Consider any intersection number $\langle\prod_{i\in\lf}\tau^{\epsilon_i}_{a_i}\rangle_{0,(d,d)}.$
By Corollary \ref{cor:cayley}, partitions $P$ with $r>2$ pieces which appear in the expression for $\OGW(\lf,(d,d),\vec{a},\vec{\epsilon})$ vanish, since their amplitude contains the term $(d-d)^{r-2}.$

When $r=2$ this argument does not guarantee vanishing. We would like to show that the $r=2$ terms of \eqref{eq:nice_tree_sum} precisely cancel the second summand of that equation. In other words, we would like to show that
\begin{equation}\label{eq:2_vs_1_terms_cancellation_in_dd}\sum_{\substack{((\lf_+,\vd_+),(\lf_-,\vd_-))}}\frac{-(\partial^H(\vd_+))^2}{-2u}I(\lf_+,\vd_+)I(\lf_-,\vd_-)=\frac{d}{-2u}I(\lf,d), \end{equation}
where the sum is over pairs of moduli specifications $((\lf_+,\vd_+),(\lf_-,\vd_-)=(\lf\setminus\lf_+,\vd-\vd_+)),$ where $\vd_+=(b+a,b)$ with $a>0$ and $\partial^H$ is the connecting map in homology, given by \eqref{eq:connecting}, so that $\partial^H(\vd_+)=a.$

By the definition of fixed-points contribution, Definition~\ref{def:inhomterms}, the left-hand side equals
\[
\sum_{(\Gamma_+,\Gamma_-)\in B{(\lf,\vd)}}\frac{a(\Gamma_+)^2}{2h}\prod_{s\in\{+,-\}}\frac{I(\Gamma_s,\vec{a}|_{\lf_s},\vec{\epsilon}|_{\lf_s})}{|\Aut(\Gamma_s)|},
\]
where $B(\lf,\vd)$ is the collection of pairs of fixed-point graphs $\Gamma_\pm$ for moduli specifications $(\lf_\pm,\vd_\pm)$ as above with $a(\Gamma_+)=\partial^H(\vd_+)>0$, $e_+$ is the unique disk edge of $\Gamma_+$ and $\delta(e^+)=(a(\Gamma_+),0)$. We have
\[
I(\Gamma_s,\vec{a}|_{\lf_s},\vec{\epsilon}|_{\lf_s})=\frac{1}{|A^0_{\Gamma_s}|}\int_{\overline{\mm}_{\Gamma_s}}e_{\Gamma_s}^{-1}
\cdot\alpha_{\Gamma_s}^{\vec{a}|_{\lf_s},\vec{\epsilon}|_{\lf_s}},
\]
as in Definition \ref{def:inhomterms}.

Recall item \eqref{it:graphical} of Remark \ref{rmk:why_not_naive}, concerning the graphical representation of \eqref{eq:nice_tree_sum}. It allows us to write the right-hand side of \eqref{eq:2_vs_1_terms_cancellation_in_dd} as the sum
\[
\sum\frac{-|\delta(e)|}{-2u}\frac{I(\Gamma,\vec{a},\vec{\epsilon})}{|\Aut(\Gamma,e)|}
\]
over pairs $(\Gamma,e),$ where $\Gamma$ is a fixed-point graph for the specification $(\lf,d),$ $e$ a sphere edge of $\Gamma$ and $\Aut(\Gamma,e)$ is the group of automorphisms of $\Gamma$ which fix $e.$ Let $A{(\lf,\vd)}$ be the collection of these pairs.

There is an obvious bijection $q:B(\lf,\vd)\to A(\lf,\vd)$ obtained by sending $(\Gamma_+,\Gamma_-)$ to $(\Gamma,e)$, where $\Gamma$ is the result of gluing $\Gamma_\pm$ along their boundaries and $e$ is the new sphere edge. The inverse $q^{-1}(\Gamma,e)$ is obtained cutting $\Gamma$ along the equator of $e.$

Under this bijection, if $(\Gamma,e)=q(\Gamma_+,\Gamma_-),$ then
\[
|{\delta}(e)|=\partial^H(\delta(e_+)),\qquad |\Aut(\Gamma_+)||\Aut(\Gamma_-)|=|\Aut(\Gamma,e)|,
\]
and, by Observation \ref{obs:sphere_vs_2_disks_cont} below, we have
\[
I(\Gamma_+,\vec{a},\vec{\epsilon})I(\Gamma_-,\vec{a},\vec{\epsilon})(-(\partial^H\delta(e_+))^2)+(-|{\delta}(e)|)I(\Gamma,\vec{a},\vec{\epsilon})=0.
\]
Thus,
\[\frac{I(\Gamma_+,\vec{a},\vec{\epsilon})}{|\Aut(\Gamma_+)|}
\frac{I(\Gamma_-,\vec{a},\vec{\epsilon})}{{|\Aut(\Gamma_-)|}}(-(\partial^H\delta(e_+))^2)+
(-|{\delta}(e)|)\frac{I(\Gamma,\vec{a},\vec{\epsilon})}{{|\Aut(\Gamma,e)|}}=0.\]
Summing over $B(\lf,\vd)$ we obtain \eqref{eq:2_vs_1_terms_cancellation_in_dd}.
\end{proof}
\begin{obs}\label{obs:sphere_vs_2_disks_cont}
If
\[
\operatorname{he}\left(\mu,d\right)=\frac{1}{d}\frac{\mu}{d!}\left(\mu\frac{2u}{d}\right)^{-d}
\]
denotes the factor in $I\left(\Gamma,\vec{a},\vec{\epsilon}\right),$ coming from \eqref{eq:euler^-1}, corresponding to a disk edge (taking into account also the size $d$ group of geometric automorphisms)
and
\[
\operatorname{ed}\left(d\right)=\frac{1}{d}\left(-1\right)^{d}\frac{d^{2d}}{\left(2u\right)^{2d}\left(d!\right)^{2}}
\]
represents the contribution of a sphere edge, then
\begin{equation}
\operatorname{ed}\left(d\right)=(-d)\cdot \operatorname{he}\left(+1,d\right)\cdot \operatorname{he}\left(-1,d\right).\label{eq:ran's observation}
\end{equation}
\end{obs}
\begin{figure}[t]
\centering
\includegraphics[scale=.5]{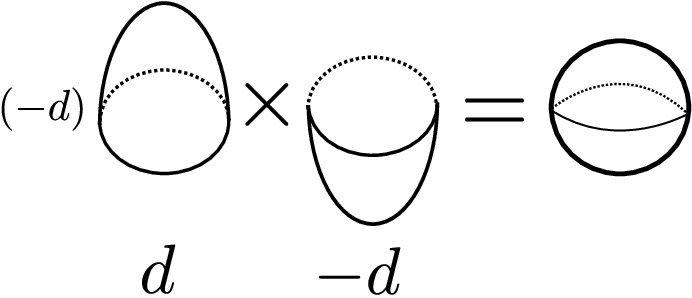}
\caption{A graphical representation of Observation \ref{obs:sphere_vs_2_disks_cont}.}
\label{fig:half_edges_vs_edges}
\end{figure}

\section{All genera definition by localization}\label{sec:high_genus}
In this section we \emph{define} via localization the stationary descendent integrals for all moduli specifications, using what should be the fixed-point formula expressing them. We will see that this definition agrees with the geometric definition for closed moduli specifications and for disk moduli specifications. We use the terminology of Sections \ref{subsec:mod_spec}, \ref{subsubsec:fpg} and \ref{subsec:fp_graph_cont}.

We begin with preliminary definitions of auxiliary objects.
\subsection{BC (Boundary contribution) graphs}
\begin{definition}\label{BC graphs}
A \emph{bare boundary contribution graph} $G=(V,E,L)$ is the following data:
\begin{enumerate}[(a)]
\item A set of vertices $V.$
\item A collection of edges $E=\vec{E}\cup \widetilde{E}$ between them.
\item A collection $L$ called the \emph{loops} which should be thought of as oriented loops. \end{enumerate}
This data is required to satisfy
\begin{enumerate}[(a)]
\item All edges in $\vec{E}$ are oriented, and any vertex has exactly one such incoming edge, and one such outgoing edge.
\item Every vertex touches a single edge of $\widetilde{E}.$ These edges are called the \emph{wavy edges}, they are not oriented and each touches two different vertices.
\end{enumerate}
From this data one may construct the set $F^{new}$ of oriented loops defined as $L$ together with the oriented closed paths of the graph $(V,\vec{E}).$
Another set of oriented loops which we define is $F^{old},$ made of $L$ and of all oriented loops which can be written as a sequence $(v_0,e_1,v_1,e_2,\ldots, e_{2m})$ (up to a cyclic change of order), under the following constraints. First, all edges $e_j$ which are directed are different. For an even $i,~e_i\in \widetilde{E},$ and it connects $v_{i-1}$ to $v_{(i(\text{mod}~2m))}.$ For an odd $i,~e_i\in \vec{E},$ and it connects $v_{i-1}$ to $v_i,$ in agreement with its direction. Elements of $F^{new}$ ($F^{old}$) are called new (old) faces.
We write $e\in f$ whenever the edge $e$ is a part of the closed path $f.$
\end{definition}

\begin{definition}
A \emph{boundary contribution graph} $G=(V,E,L,d),$ or a \emph{BC graph} for shortness, is a bare boundary contribution graph $G=(V,E,L)$ together with a \emph{perimeter function}
\[
d:F^{new}\cup F^{old}\to \Z,
\]
which associates to a face its perimeter. It is required that the perimeters of new faces will be non-zero, and that the sum of perimeters of old faces equals the sum of perimeter of new faces, in any connected component $H$ of $G.$ We write $\vec{d}$ for the collection of perimeters of faces, and we sometimes use this notation instead of writing the function $d.$

An isomorphism of BC graphs $G=(V,E,L,d),G'=(V',E',L',d')$ is a collection of bijections
\[
f^V:V\to V',\quad f^E:E\to E',\quad f^L:L\to L',
\]
which preserve incidence relations between edges and vertices, directions of directed edges and perimeters of faces. In particular, there are induced bijections between new and old faces of $G,G'.$

An \emph{enumeration} of a BC graph is an enumeration of its new faces by $1,\ldots,|F^{new}|,$ of its old faces by $1,\ldots, |F^{old}|,$ of its wavy edges by $1,\ldots,|\widetilde{E}|$, and a choice of orientation to each wavy edge. If we define isomorphisms and automorphisms in the expected way, then the automorphism group of an \emph{enumerated BC graph} is easily seen to be trivial.
\end{definition}
Observe that any element of $\vec{E}$ belongs to one new and one old face, while every element of $\widetilde{E}$ belongs to two old faces (or one old, but appears twice in that face).

\begin{definition}\label{def:conf_space}
A \emph{metric} on a BC graph $G$ is an association of a (signed) length $x_e\in\R$ for any $e\in\vec{E}$ such that
\begin{enumerate}[(a)]
\item The sum of lengths of the oriented edges of any circuit of the graph is an integer.
\item The sum of lengths of the oriented edges which belong to a new or old face is the perimeter of that face.
\item If $e\in\vec{E}$ is an edge of a new face $f$, then $x_ed(f)>0.$
\end{enumerate}
We denote by $W_G\subseteq \R^{\vec{E}}$ the space of metrics on $G$ and by $\overline{W}_G$ its closure in $\R^{\vec{E}}.$
\end{definition}
We consider the lengths as functions on $W_G,~x_e:W_G\to\R.$
\begin{defn}\label{def:form}
Let $\mathcal{E}\subseteq \vec{E}$ be a set of edges that together with $\widetilde{E}$ forms a spanning forest of $G.$ This means that $\mathcal{E}\sqcup\vec{E}$ do not contain any circuit, and any two vertices which belong to the same connected component of $G$ are connected by a path in $\mathcal{E}\sqcup\vec{E}.$
Write \[{\Omega}_G = \bigwedge_{e\in\mathcal{E}} dx_e.\]
\end{defn}
\begin{obs}
$\Omega_G$ is well-defined up to a sign.
\end{obs}
\begin{proof}
Changing the order in which the wedge is taken results in a change of sign to $\Omega_G.$ Replacing~$\mathcal{E}$ by another set of edges with the same properties $\mathcal{E}'$ can be achieved in a sequence of steps
\[
\mathcal{E}=\mathcal{E}_0\to\mathcal{E}_1\to\cdots\to\mathcal{E}_r=\mathcal{E}',
\]
where for each $i$ we have $\mathcal{E}_{i+1}=\mathcal{E}_i\backslash\{e_i\}\cup\{e_{i+1}\}.$
Adding $e_{i+1}$ to $\mathcal{E}\sqcup\widetilde{E}$ creates a circuit in the graph which contains $e_i,e_{i+1},$ additional oriented edges $\{e_\alpha\}_{\alpha\in A}\subseteq\mathcal{E}\backslash\{e_i\},$ and some elements of $\widetilde{E}.$ Since any circuit has an integer perimeter, this perimeter is locally constant on $W_G.$ Thus, on $W_G$ we have
\[dx_{e_{i+1}}=-dx_{e_i}-\sum_{\alpha\in A} dx_\alpha,\]
hence,
\[
\bigwedge_{e\in\mathcal{E}_i} dx_e=\pm\left(\bigwedge_{e\in\mathcal{E}_i\backslash{e_i}} dx_e\right)\wedge\left(-dx_{e_i}-\sum_{\alpha\in A} dx_\alpha\right)=\pm\left(\bigwedge_{e\in\mathcal{E}_i\backslash{e_i}} dx_e\right)\wedge dx_{e_{i+1}}=\pm\bigwedge_{e\in\mathcal{E}_{i+1}} dx_e,
\]
and the proof follows.
\end{proof}

\begin{definition}\label{def:volume}
Define the \emph{volume} of a BC graph $G=(V,E,L,d)$ to be
\[\Vol_{G}=(-1)^{alt}\left|\left(\prod_{f\in F^{new}\backslash L}d(f)\right)\int_{\overline{W}_G}\Omega_G\right|,\]
where $alt$ is the number of edges of $\widetilde{E}$ which touch two new faces $f_1,f_2$ with $d(f_1)d(f_2)<0.$
\end{definition}
In the sequel we will study those volumes and we will see, in particular, that they are piecewise polynomial functions in the perimeters of the new and old faces.

\subsection{Morphisms between specifications and the compact form of the localization definition}
A moduli specification is \emph{pure} if there is no boundary of degree $0$ ($0$ is not in the image of $d$).

\begin{defn}\label{def:hom_between_specs}
A \emph{morphism} from a moduli specification $S^{new}$ to a moduli specification $S^{old}$
is a decorated graph $M=(V_b\sqcup V_w,E,H^{CB},\widetilde{E},g_s,\lf,\vd,d,\sigma,d^{old})$ such that
\begin{enumerate}[(a)]
\item $(V_b\sqcup V_w,E,\vd,g_s,\lf,\vd,d)$ is the moduli specification $S'.$
\item The \emph{wavy edges} $\widetilde{E}$ connect white vertices.
\item The \emph{contracted boundary half-edges} $H$ are half-edges which emanate from black vertices.
\item A \emph{wavy flag} is a pair of a wavy edge and an orientation of it.
For any white vertex $v,~\sigma_v$ is a cyclic order of the wavy flags which are incident to $v$ and oriented outward of $v.$
\item There is an involution $\tau$ on wavy flags defined by associating a given wavy flag the other wavy flag which corresponds to the same wavy edge with opposite orientation. The function $d^{old}$ is a mapping from the set of cycles of $\sigma^{-1}\circ\tau$ to $\zz.$ Here and afterwards we consider a white vertex which is not incident to a wavy edge as a cycle of $\sigma^{-1}\circ\tau$ of its own.
\item $S^{old}$ is the \emph{desingularization} of $M,$ where the desingularization is the moduli specification defined as follows.
\begin{enumerate}[\textbullet]
\item Associate a black vertex to any connected component of the graph $G=(V_b\sqcup V_w,E\sqcup\widetilde{E})$ (we consider $G$ as a graph in the usual sense, with vertices $V_b\sqcup V_w$ and edges $E\sqcup\widetilde{E}$).
\item For such a vertex $v'$ define
\[
\vd(v')=\sum_v \vd(v),\quad \lf(v')=\bigsqcup_v\lf(v),\quad g_s(v')=h_1(G)+\sum_v g_s(v),
\]
where the summation or union are over all black vertices of $G$ which belong to the connected component represented by $v'.$
\item We associate a white vertex for any cycle of $\sigma^{-1}\circ\tau,$ connect it to the black vertex $v$ which corresponds to the connected component which contains this cycle. Define, for a white vertex $w,~d(w)=d^{old}(c),$ where $c$ is the cycle which corresponds to $w.$
\item We associate a white vertex for any contracted boundary half-edge, define $d$ for such a vertex to be $0,$ and connect this vertex to the black vertex which represents the connected component containing this contracted boundary half-edge.
\end{enumerate}
\end{enumerate}

As usual, isomorphisms or automorphisms are graph isomorphisms or automorphisms which respect the additional data, and $\Aut(M)$ is the automorphism group of $M.$

Write $\Hom(S',S)$ for the set of isomorphism types of morphisms from $S'$ to $S.$
\end{defn}
\begin{definition}\label{def:BCndCBforMor}
One can associate a BC graph $G$ for any morphism \[M=(V_b\sqcup V_w,E,H^{CB},\widetilde{E},\vd,g_s,\lf,\vd,d,\sigma,d^{old})\in \Hom(S^{new},S^{old})\] as follows:
\begin{enumerate}[\textbullet]
\item Associate a vertex $v$ for any wavy flag $e_v$.
\item Connect two vertices which correspond to flags which are paired via $\tau$ by a wavy edge of~$G.$
\item Draw a directed edge $(v,u)$ if $\sigma(e_v)=e_u.$
\item Associate a loop to any white vertex without wavy edges.
\item So far we obtain a bare boundary contribution graph, whose new faces correspond white vertices of $M$ or equivalently of $S^{new},$ and the old faces correspond to the $\sigma^{-1}\circ\tau-$cycles of $M,$ or equivalently the white vertices of $S.$ Define $d$ for an old face to be $d^{old}$ of the corresponding cycle, and $d$ of a new face to be $d$ of the corresponding vertex.
\end{enumerate}
We denote $G$ by $BC(M).$

In addition, one can associate a degree, which we also denote by $d$ to a contracted boundary half-edge $h$ of $M$. Suppose $h$ is connected to a black vertex $v.$ Define
\[\delta(h)=\delta^{morph}(h)=d^+(v)-\sum_{w\in V_w(v)|d(w)>0}d(w),\]
where $V_w(v)$ is the set of white neighbors of $v.$
\end{definition}

\begin{defn}\label{def:OGW_loc}
Recall Definition \ref{def:inhomterms}. Let $S$ be a moduli specification, $\vec{a},\vec{\epsilon}$ be vectors of descendents and constraints. Define the \emph{open Gromov-Witten invariant} of $(S,\vec{a},\vec{\epsilon})$ as follows:
\begin{equation}\label{eq:ogw_compact}
\OGW(S) = \sum_{\text{$S'$ is pure}}\sum_{M\in \Hom(S',S)}\frac{1}{|\Aut(M)|}\frac{\Vol_{BC(M)}}{(-2u)^{|\widetilde{E}|}}\left(\prod_{h\in H^{CB}(M)}\frac{\delta^{morph}(h)}{2u}\right)I(S',\vec{a},\vec{\epsilon}).
\end{equation}
\end{defn}
Note that, although the summation is over infinitely many specifications, only for finitely many of them $\Hom(S',S)$ is non-empty. Indeed, for a given $S$ the only specifications $S'$ for which $\Hom(S',S)$ is non-empty must have the same set of underlying labels, same total degree and total genus bounded by that of $S.$ There are only finitely many such specifications.
Even when $\Hom(S',S)$ is non-empty, it could still happen that the summand will vanish because of the volume term.

\subsection{Equivalent definition in terms of graphs}
\begin{defn}\label{def:GraphsSnewSold}
A \emph{localization graph} is a tuple
\[\mathcal{G}=\left(V\sqcup V_\circ,\FFF,\delta,\gamma,\mu,\lambda,H^{CB},\widetilde{E},\sigma,d^{old}\right)\]
such that
\begin{enumerate}[(a)]
\item $\Gamma=\left(V\sqcup V_\circ,\FFF,\delta,\gamma,\mu,\lambda\right)$ is a fixed point graph. We denote $\Gamma$ by $FP(\mathcal{G}).$
\item $H^{CB}$ are half-edges which emanate from equators of $\Gamma$ and are called \emph{contracted boundary~edges}.
\item $\widetilde{E}$ are the \emph{wavy edges} which connect boundaries of $\Gamma.$ Wavy flags are wavy edges together with orientations.
\item $\sigma_v,$ for a boundary vertex $v\in V_\circ,$ is a cyclic order of wavy flags which emanate from $v.$ We define $\tau$ as in Definition \ref{def:hom_between_specs}.
\item $d^{old}$ is a function from cycles of $\sigma^{-1}\circ\tau$ to $\zz.$
\end{enumerate}

Given moduli specifications $S^{new},S^{old},$ we say that $\mathcal{G}$ is a fixed point graph of type $(S^{new},S^{old}),$ if $FP(\mathcal{G})$ is of type $S^{new}$ and the \emph{contraction} of $\mathcal{G}$ belongs to $\Hom(S^{new},S^{old}),$ where the contraction $\Contract(\mathcal{G})$ is the tuple $(V_b\sqcup V_w,E,H^{CB},\widetilde{E},g_s,\lf,\vd,d,\sigma,d^{old})$ defined as follows.

$(V_b\sqcup V_w,E,g_s,\lf,\vd,d)$ is the contraction of $FP(\mathcal{G})$ as defined in the end of Definition \ref{def:fpg}. In particular there is an identification between boundaries in $V_\circ$ and $V_w.$ With this identification we draw wavy edges, also denoted $\widetilde{E},$ between white vertices, and denote the induced cyclic orders also by $\sigma=\sigma^{\Contract(\mathcal{G})}.$ This also allows us to define $\tau=\tau^{\Contract(\mathcal{G})},$ and hence induce the function, which is still denoted $d^{old},$ on cycles of $\sigma^{-1}\circ\tau$ of $\Contract(\mathcal{G}).$
Finally, by the identification between $V_b$ and connected components of $FP(\mathcal{G}),$ we associate, for any contracted boundary half-edge $h\in H^{CB}(\mathcal{G})$ a contracted boundary half-edge, which emanates from the black vertex which corresponds to the connected component of $FP(\mathcal{G})$ which contains $h.$

Isomorphisms, automorphisms and the group $\Aut(\mathcal{G})$ are defined in the expected way.

We denote by $\Graphs(S',S)$ the set of isomorphism types of localization graphs of type $(S',S).$

We denote $BC(\Contract(\mathcal{G}))$ by $BC(\mathcal{G}).$
Finally, for a contracted boundary half-edge $h$ which is connected to an equator $v$ we set $\delta(h)=\delta^{graph}(h)=|\delta(v)|.$
\end{defn}
The following lemma provides an equivalent interpretation of $\OGW$ in terms of graph summation. We will consider it as an alternative definition.
\begin{lemma}\label{lem:OGW_loc_graphs}
Let $S$ be a moduli specification, $\vec{a},\vec{\epsilon}$ be vectors of descendents and constraints. We have
\begin{equation}\label{eq:ogw_graph_sum}
\OGW(S) = \sum_{\text{$S'$ is pure}}\sum_{\mathcal{G}\in \Graphs(S',S)}\frac{1}{|\Aut(\mathcal{G})|}\frac{\Vol_{BC(\mathcal{G})}}{(-2u)^{|\widetilde{E}|}}\left(\prod_{h\in H^{CB}(\mathcal{G})}\frac{\delta^{graph}(h)}{2u}\right)I(FP(\mathcal{G}),\vec{a},\vec{\epsilon}).
\end{equation}
\end{lemma}
Definition \ref{def:OGW_loc} is a compact, renormalized packing of \eqref{eq:ogw_graph_sum}. The advantage of \eqref{eq:ogw_graph_sum} is that it can be described graphically in a way which generalizes nicely the localization graphs of closed $GW$ theory, see Figure \ref{fig:high_g_cont} for an illustration.
\begin{proof}
Write $\Pi(\mathcal{G})=(\Contract(\mathcal{G}),FP(\mathcal{G})).$
Then the RHS of \eqref{eq:ogw_graph_sum} can be written as
\[
\sum_{\text{$S'$ is pure}}\sum_{\substack{M\in \Hom(S',S)\\
\text{$\Gamma$ is a fixed-point graph for $S'$}}}
\sum_{\mathcal{G}\in\Pi^{-1}(M,\Gamma)}\frac{1}{|\Aut(\mathcal{G})|}\frac{\Vol_{BC(M)}}{(-2u)^{|\widetilde{E}|}}\left(\prod_{h\in H^{CB}(\mathcal{G})}\frac{\delta^{graph}(h)}{2u}\right)I(\Gamma,\vec{a},\vec{\epsilon}).
\]
On the other hand, using Definition \ref{def:inhomterms}, the RHS of \eqref{eq:ogw_compact} can be written as
\[
\OGW(S) = \sum_{\text{$S'$ is pure}}\sum_{M\in \Hom(S',S)}\frac{|\Aut(S')|}{|\Aut(M)|}\frac{\Vol_{BC(M)}}{(-2u)^{|\widetilde{E}|}}\left(\prod_{h\in H^{CB}(M)}\frac{\delta^{morph}(h)}{2u}\right)\sum_{\substack{\text{$\Gamma$ is a fixed-point}\\\text{graph for $S'$}}}\frac{I(\Gamma,\vec{a},\vec{\epsilon})}{|\Aut(\Gamma)|}.
\]
The lemma will follow if we could prove that for any $M\in \Hom(S',S)$ and a fixed-point graph $\Gamma$ for $S'$ we have
\begin{equation}\label{eq:equiv1}
\frac{|\Aut(S')|}{|\Aut(M)||\Aut(\Gamma)|}\left(\prod_{h\in H^{CB}(M)}\frac{\delta^{morph}(h)}{2u}\right)=
\sum_{\mathcal{G}\in\Pi^{-1}(M,\Gamma)}\frac{1}{|\Aut(\mathcal{G})|}\left(\prod_{h\in H^{CB}(\mathcal{G})}\frac{\delta^{graph}(h)}{2u}\right).
\end{equation}
Define an \emph{enumeration} of a finite set $A$ by a set $B$ to be a bijection from $A\to B.$
A moduli specification $S$ is \emph{enumerated} if for each connected component $S'$ and $a\in \zz$ the white vertices of degree $a$ are enumerated by $[m_a(S')],$ where $m_a(S')$ is the number of these white vertices, and, in addition, for each isomorphism type of connected moduli specification $S',$ the connected components of $S$ of this type are enumerated by $[m_{S'}(S)]$, where $ m_{S'}(S)$ is the number of these components.

An \emph{enumeration of a morphism} $M\in \Hom(S',S)$ is an enumeration of $S'$ and $S$ (where we identify white vertices of $S$ with cycles of $\sigma^{-1}\circ\tau$) together with an enumeration of $\widetilde{E}$ by $[|\widetilde{E}|],$ an enumeration of $H^{CB}(M)$ by $[|H^{CB}(M)|],$ and a choice of orientation for each element of $\widetilde{E}$.

An \emph{enumeration of a fixed point graph} $\Gamma$ is an enumeration of its contraction $S$ (where we identify boundaries of $\Gamma$ with white vertices of $S$) together with enumerations $V(\Gamma),~E(\Gamma)$ by $[|V(\Gamma)|],[|E(\Gamma)|]$ respectively. An enumeration of a localization graph $\mathcal{G}\in \Hom(S',S)$ is a pair of enumerations, one for $FP(\mathcal{G})$ and the other for $\Contract(\mathcal{G})$ which are compatible in the sense that they induce the same enumeration on $S'.$

We define isomorphisms and automorphisms and the group $\Aut$ of each of the above enumerated objects in the expected way, and write $\Enum(X)$ for the set of isomorphism types of enumerations of the object $X.$ We see that for each of the above enumerated objects the automorphism group is trivial, therefore, by the Orbit-Stabilizer theorem we have
\begin{obs}
\label{obs:orb_stab}
\begin{enumerate}[(a)]
\item\label{it:orb1} For a fixed point graph $\Gamma$ of type $S'$ we have
\[|\Aut(\Gamma)||\Enum(\Gamma)|=|\Aut(S')||V(\Gamma)|!|E(\Gamma)|!.\]
\item\label{it:orb2} For a morphism $M\in \Hom(S',S)$ we have
\[|\Aut(M)||\Enum(M)|=|\Aut(S')||\Aut(S)|2^{|\widetilde{E}(M)|}{|\widetilde{E}(M)|!}{|H^{CB}(M)|!}.\]
\item\label{it:orb3} We have
\[|\Aut(\mathcal{G})||\Enum(\mathcal{G})|={|\Aut(S')||\Aut(S)|2^{|\widetilde{E}(M)|}{|\widetilde{E}(M)|!}{|H^{CB}(M)|!}{|V(\Gamma)|!|E(\Gamma)|!}}.\]
\end{enumerate}
\end{obs}
We also observe that for a given moduli specification $S$ there are precisely $|\Aut(S)|$ enumerations and they are all isomorphic.

Item \eqref{it:orb3} of the previous observation allows us to rewrite the RHS of \eqref{eq:equiv1} as
\begin{equation}\label{eq:equiv2}
\sum_{\mathcal{G}\in\Pi^{-1}(M,\Gamma)}
\frac{|\Enum(\mathcal{G})|}{|\Aut(S')||\Aut(S)|2^{|\widetilde{E}(M)|}{|\widetilde{E}(M)|!}{|H^{CB}(M)|!}{|V(\Gamma)|!|E(\Gamma)|!}}
\left(\prod_{h\in H^{CB}(\mathcal{G})}\frac{\delta^{graph}(h)}{2u}\right).
\end{equation}
Note that any enumeration of $\mathcal{G}\in\Pi^{-1}(M,\Gamma)$ induces an enumeration of $\Gamma$ and of $M.$ When $c=0,$ the opposite is also true: enumerations of $\Gamma$ and of $M$ allow us to reconstruct a unique $\mathcal{G}\in \Pi^{-1}(M,\Gamma)$ which induces the original enumerations on $\Gamma,M$ since the only difference between the data of $\mathcal{G}$ and of $(M,\Gamma)$ is how to identify boundaries of $\Gamma$ with white vertices of $M,$ and the enumeration gives us such an identification. Thus, when $c=0$ we have
\begin{equation}\label{eq:c=0}
\sum_{\mathcal{G}\in\Pi^{-1}(M,\Gamma)}|\Enum(\mathcal{G})|=|\Enum(\Gamma)||\Enum(M)|.
\end{equation}
For a general $c$ we have
\begin{equation}\label{eq:cnn0}
\sum_{\mathcal{G}\in\Pi^{-1}(M,\Gamma)}|\Enum(\mathcal{G})|\left(\prod_{h\in H^{CB}(\mathcal{G})}\frac{\delta^{graph}(h)}{2u}\right)=
|\Enum(\Gamma)||\Enum(M)|\left(\prod_{h\in H^{CB}(M)}\frac{\delta^{morph}(h)}{2u}\right).
\end{equation}
The proof is by simple induction. The case $c=0$ is just \eqref{eq:c=0}. If the claim holds for $c$ contracted boundary half-edges, then adding a new contracted boundary half-edge $h$ which belongs to a connected component $S'_0$ of $S'$ changes the RHS by $\frac{\delta^{morph}(h)}{2u}.$ The right-hand side changes, by definition, by the sum of $\frac{|\delta(v_\circ)|}{2u},$ where $v_\circ$ runs over all equators in the connected component $\Gamma_0$ of $\Gamma$ which corresponds to $S'_0.$ This sum is again, by definition, $\frac{\delta^{morph}(h)}{2u}.$
Plugging \eqref{eq:cnn0} into \eqref{eq:equiv2} we obtain
\begin{multline}\label{eq:equiv3}
\frac{|\Enum(\Gamma)||\Enum(M)|\left(\prod_{h\in H^{CB}(M)}\frac{\delta^{morph}(h)}{2u}\right)}{|\Aut(S')||\Aut(S)|2^{|\widetilde{E}(M)|}{|\widetilde{E}(M)|!}{|H^{CB}(M)|!}{|V(\Gamma)|!|E(\Gamma)|!}}=\\
=|\Aut(S')|\frac{|\Enum(\Gamma)|}{|\Aut(S')|{|V(\Gamma)|!|E(\Gamma)|!}}\frac{|\Enum(M)|\left(\prod_{h\in H^{CB}(M)}\frac{\delta^{morph}(h)}{2u}\right)}{|\Aut(S')||\Aut(S)|2^{|\widetilde{E}(M)|}{|\widetilde{E}(M)|!}{|H^{CB}(M)|!}}.
\end{multline}
Using the first two items of Observation \ref{obs:orb_stab} this expression gives exactly the LHS of \eqref{eq:equiv1}.
\end{proof}

\begin{figure}[t]
\centering
\includegraphics[scale=.6]{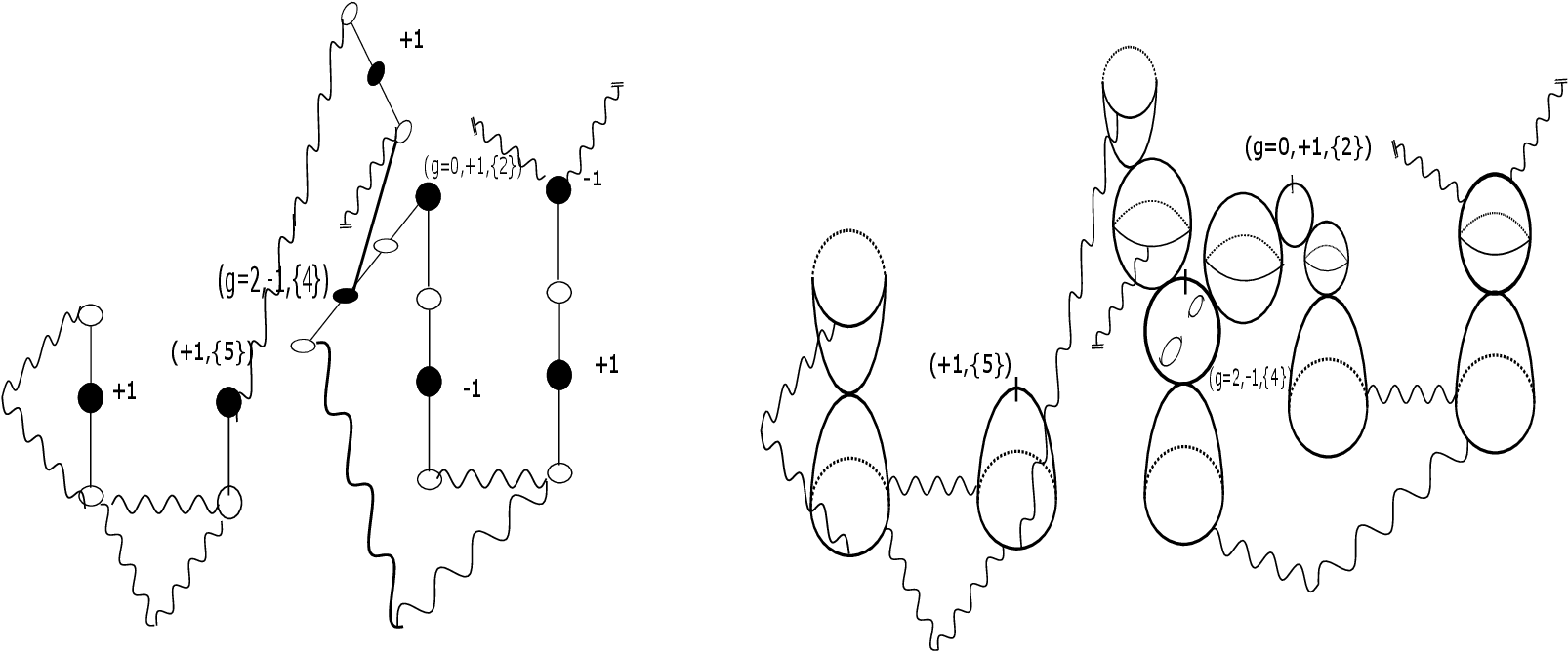}
\caption{A localization graph which contributes to a high genus invariant and the corresponding geometric stratum. Again for illustration reasons we do not contract degree $0$ components. Empty vertices stand for boundary and equator vertices. Next to the remaining vertices we write their genus, if they correspond to contracted component, $\mu$ and the markings.}
\label{fig:high_g_cont}
\end{figure}

\subsection{Agreement in $g=0$ and for closed topologies}
\begin{prop}\label{prop:agreement_of_defs}
Definition \ref{def:OGW_loc} when restricted to $g=0$ agrees with the formula of Theorem~\ref{thm:loc for CP1,RP1}.
\end{prop}
\begin{proof}
The sums in Definition \ref{def:OGW_loc}, in the $g=0$ case are of two types. Either there is no contracted boundary half-edge, or there is a single contracted boundary half-edge, but no wavy edges.

The second case corresponds to the second summand in \eqref{eq:nice_tree_sum}, and in both formulas these terms have the same value.

In the first case, in Definition \ref{def:OGW_loc} there is a sum over BC graphs, which in this case can be identified with trees whose vertices correspond to new faces, edges to wavy edges, each vertex has a non-zero integer degree and a cyclic order of the edges from each vertex is given. On the other hand, in \eqref{eq:nice_tree_sum} the summation is over the exact same trees, with non-zero integral degrees for vertices, \emph{but without a cyclic order for edges which emanate from a vertex}.

Thus, there is a natural surjection $s$ from the collection of BC graphs which appear in $\OGW(\lf,\vd)$ and $\TTT\left(\lf,\vd\right),$ obtained by forgetting the cyclic orders at vertices.
The proposition would therefore follow if we could show that for any $T\in\TTT\left(\lf,\vd\right)$ we have
\begin{equation}\label{eq:vol_sum}\sum_{G\in s^{-1}(T)} \Vol_G=\prod_{v\in V}\left(d^{+}\left(v\right)-d^{-}\left(v\right)\right)^{\val\left(v\right)}.
\end{equation}
The sign of each summand in the LHS is $(-1)^{alt},$ where $alt$ is the number of edges between vertices of the tree $G$ which connect a vertex of positive degree to a vertex of negative degree. The sign of the RHS is the sum of valencies of vertices of negative degrees. These two numbers agree modulo~$2,$ hence the signs of the two sides of \eqref{eq:vol_sum} are the same.

Recall Definitions \ref{def:form}, \ref{def:volume}. When the tree $T$ is trivial, equation \eqref{eq:vol_sum} clearly holds. Otherwise, any $G\in s^{-1}(T)$ has no loops and has a single old face whose perimeter is the sum of perimeters of the new faces. If $f_v$ is the new face of $G$ which corresponds to the vertex $v$ of $T,$ and if we write $e\in f_v$ to specify that an oriented edge $e$ belongs to the new face $f_v,$ then the space $W_G$ is defined as the space
\[
\left\{(x_e)\in\R^{\vec{E}}\left|\forall v~\sum_{e\in f_v} x_e=d(f_v)~\text{and }x_ed(f_v)>0\right.\right\}.
\]
Thus, $W_G$ is identified with a product of simplices of the form \[\{(x_1,\ldots,x_r)| \sum x_i=d \text{ and all $x_i$ have the same sign}\}.\]
If we take $\mathcal{E}$ as in Definition \ref{def:form} as a set whose intersection with $\{e|e\in f_v\}$ is of cardinality $|\{e|e\in f_v\}|-1$ for any $v,$ we see that
\[
\left|\int_{\overline{W}_G}\Omega_G\right|=\prod_v|d(f_v)|^{|\{e|e\in f_v\}|-1}/(|\{e|e\in f_v\}|-1)!
\]
as the product of volumes of simplices. Multiplying by $\prod_{f\in F^{new}}d(f)=\prod_vd(f(v)),$ and summing over all $\prod_v(|\{e|e\in f_v\}|-1)!=\prod_v(\val(v)-1)!$ elements of $s^{-1}(T)$ we obtain
\[
\sum_{G\in s^{-1}(T)}|\Vol_G|=\prod_{f\in F^{new}}\left|d(f)\right|^{|\{e|e\in f\}|}=\left|\prod_{v\in V}\left(d^{+}\left(v\right)-d^{-}\left(v\right)\right)^{\val\left(v\right)}\right|,
\]
which completes the proof of equation \eqref{eq:vol_sum}.
%Now, working in the notations of Definitions \ref{def:conf_space}-\ref{def:volume}\[
%\sum_{G\in s^{-1}(T)} |\Vol_G|=|\int_{\bigcup_{G\in s^{-1}(T)}\overline{W}_G}\bigwedge_{i=1}^r d\theta_i|=
%|\int_{W}\bigwedge_{i=1}^r d\theta_i|,
%\]
%where \[W =\prod_{i=1}^r\left(\partial^{k(i,1)}{_u\times_u}\partial^{k(i,2)}\right),\] and the spaces $\overline{W}_G$ for $G\in s^{-1}(T)$ naturally embed in $W,$ cover $W$ and have intersections of positive codimensions. Clearly,
%\[|\int_{W}\bigwedge_{i=1}^r d\theta_i|=\prod_{i=1}^r|\int_{\partial^{k(i,1)}{_u\times_u}\partial^{k(i,2)}}d\theta_i|.\]
%\eqref{eq:vol_sum} will therefore follow from showing\[|\int_{\partial^{k(i,1)}{_u\times_u}\partial^{k(i,2)}}d\theta_i|=|d^{k(i,1),new}\cdot d^{k(i,2),new}|.\]
%For convenience we write $d_j=d^{k(i,j),new}$ for $j=1,2.$
%$\partial^{k(i,1)}{_u\times_u}\partial^{k(i,2)}$ can be identified with
%\[\{(z,w)|~z^{d_1}=w^{d_2},~\text{and }|z|=|w|=1\}\] which in turn can be parameterized as \[
%\{(v^{lcm(d_1,d_2)/d_1},\xi v^{lcm(d_1,d_2)/d_2})|~\xi^{gcd(d_1,d_2)}=1,~|v|=1\}.\]
%In addition, the pull back of $d\theta_i$ to $\{(v,\xi)|~\xi^{gcd(d_1,d_2)}=1,~|v|=1\}$ is
%\[lcm(d_1,d_2)\frac{dv}{2\pi iv}.\]
%Finally,
%\[\int_{\{(v,\xi)|~\xi^{gcd(d_1,d_2)=1},~|v|=1\}}lcm(d_1,d_2)\frac{dv}{2\pi iv}=
%gcd(v_1,v_2)\int_{|v|=1}lcm(d_1,d_2)\frac{dv}{2\pi iv}=d_1d_2.
%\]
%And the claim follows.
\end{proof}
\begin{rmk}\label{rmk:non_refined_numbers}
The genus $0$ case and the positive genus cases share the feature that the contribution of wavy edges (the $\Vol_G$ term) decouples from that of the fixed-point graphs. Still, the genus $0$ case is simpler than the general case as the wavy edge contribution factorizes as a product over wavy edges of the wavy edge term
\begin{equation}\label{eq:diff_vertices}
\frac{(d^+(w_1)-d^-(w_1))(d^+(w_2)-d^-(w_2))}{-2u},
\end{equation}
where the wavy edge $e$ connects $w_1$ and $w_2.$
It turns out that if one considers less refined intersection numbers, obtained by summing $\OGW(S,\vec{a},\vec{\epsilon})$ over all moduli specifications with fixed $g,\lf,\vd$ (but varying number of boundaries and boundary degrees), one obtains a formula which is a sum over graphs which are similar to localization graphs, but without cyclic orders. In this case the wavy edge contribution again factorizes as a product of edge weights, where wavy edges which connect different vertices have the weight \eqref{eq:diff_vertices}, while wavy edges which connect the same white vertex $w$ have the weight
\[-\frac{1}{2u}\binom{|d^+w-d^-w|}{2}.\]
More details, including a proof of this claim, will appear in the sequel.
\end{rmk}

Another important observation is that in the special case of closed moduli specifications, Definition \ref{def:OGW_loc} agrees with the geometric definition of stationary descendent integrals. More precisely, we have
\begin{thm}\label{lem:closed_fp}
For a closed connected moduli specification $S$ we have
\[
\OGW(S,\vec{a},\vec{\epsilon}) = \int_{\overline \mm_{g_s(S),\lf(S)}(|\vd(S)|)} \prod_{i\in\lf(S)} \psi_i^{a_i}\ev_i^*\pt.
\]
\end{thm}
\begin{proof}
This follows from the virtual localization formula \cite{pandharipande-virtual-loc}.
\end{proof}

\subsection{Main conjecture}\label{tmconj}
Theorem \ref{lem:closed_fp} and Proposition \ref{prop:agreement_of_defs} provide evidences to our main conjecture.
\begin{conj}\label{conj:high_genus_def_and_loc}
There exists a \emph{geometric definition} for $\OGW(S,\vec{a},\vec{\epsilon})$ for all moduli specifications~$S$. In particular, we have
\[\OGW(S,\vec{a},\vec{\epsilon}) = 0 \]
for all \emph{underdetermined} $(S,\vec{a},\vec{\epsilon})$.
\end{conj}
A pair $(S,\vec{a})$ is \emph{underdetermined} if the degree of the integrand,
\[
\operatorname{deg}(S) = \sum_{i \in \lf} 2\,(a_i + 1),
\]
is less than the expected dimension of the moduli,
\[\operatorname{vdim}_{\rr}(S) = 2d + 2g-2 + 2\,|\lf|.\]
In the sequel we will sketch a proof that under ideal, though non-existent, transversality assumptions, the conjecture holds. The sketch generalizes in a non-trivial way the proof of Theorem \ref{thm:loc for CP1,RP1}, and we believe that when open virtual localization will be rigourously defined, it will be automatic to promote the sketch of proof to a full proof.

\section{Canonical orientations for the moduli spaces and their properties}\label{sec:or}
This section is devoted to defining and understanding the properties of the orientations we have used throughout the paper.
We will analyze the orientation of the moduli of maps from disks to $(\CP^1,\RP^1)$ using a natural decomposition of this moduli into chambers of full dimension. The chambers will be described in terms of certain types of maps. These maps will also play a role in the map decomposition property of $\OGW(S),$ that will be discussed in the sequel. The section is built as follows. We first describe the maps we need. We then use them to define an orientation for each chamber. We show that the orientations glue to the whole moduli, and then, again using the language of maps, analyze the behavior of the orientations with respect to degenerations.
\subsection{Decorated maps}
For a moduli specification $S$ write $\Sigma(S)$ for the topological marked surface whose label set is $\lf(S)$ and whose topological type is the one dictated by $S$.
\begin{definition}\label{def:maps}
Let $S$ be a moduli specification, and $\kf\subset\Univ$ a finite set.
A \emph{rigid decorated $(S,\kf)$-map} $\mathfrak{G}$ (or a rigid decorated $S-$map when $\kf=\emptyset$) is the following data.
\begin{enumerate}[(a)]
\item A collection of smooth simple \emph{loops and arcs} drawn on $\Sigma=\Sigma(S)$ which cut $\Sigma$ into open \emph{regions}. We call the
    boundary components of $\Sigma(S)$ the \emph{boundary loops}. The other loops are called the \emph{internal loops}. The internal loops and arcs are pairwise non-intersecting.
\item The internal loops do not intersect $\partial\Sigma(S),$ the arcs intersect $\partial\Sigma$ exactly in their two endpoints.
\item Any region $\mathcal{R}$ has a well defined \emph{sign} $s(\mathcal{R})\in\{\pm1\}.$
Neighboring regions, which are regions whose closures intersect, have opposite signs.
A region is \emph{positive} if its sign is positive and otherwise is \emph{negative}.
\item Any boundary component of the closure of a region is endowed a perimeter (circumference) $p\in \N,$ in particular any internal loop $\mathcal{C}$ is endowed a perimeter $p(\mathcal{C}).$ If $\mathcal{C}$ is a boundary component of $\Sigma(S)$ which touches no arc, then the product of its perimeter and the sign of the unique region it touches equals the degree of the corresponding boundary of $S.$
\item The \emph{degree} of a region is $(q,0)\in\Z^2$ for a positive region and $(0,q)\in\Z^2$ for a negative one, where $q$ is the sum of perimeters of the region's boundary components.
The sum of degrees of regions in the connected component of $\Sigma$ which corresponds to the connected moduli specification $S'$ is $\vd(S').$
\item The internal markings are distinct and do not lie on arcs or internal loops, and for each marking it is specified to which region it belongs.
\item  There are $|\kf|$ distinct boundary markings, labeled by $\kf$ on the boundary components. They differ from the endpoints of the arcs.
\end{enumerate}
We call $\Sigma=\Sigma(S)$ the underlying surface of the map.

An isomorphism of rigid decorated maps $\mathfrak{G}_1,\mathfrak{G}_2$ is a homeomorphism of their underlying surfaces, which induces bijections between regions of $\mathfrak{G}_1$ and regions of $\mathfrak{G}_2,$ loops of $\mathfrak{G}_1$ and loops of $\mathfrak{G}_2,$ arcs of $\mathfrak{G}_1$ and arcs of $\mathfrak{G}_2,$ boundary segments of the underlying surface of $\mathfrak{G}_1$ and boundary segments of the underlying surface of $\mathfrak{G}_2,$ and endpoints of arcs of $\mathfrak{G}_1$ and endpoints of arcs of $\mathfrak{G}_2,$ which respects the (internal and boundary) markings
and such that inclusion relations and all discrete parameters: sign, perimeter, orientation (for a boundary segment) are preserved.

A \emph{decorated ($(S,\kf)$- or $S-$) map} is an isomorphism class of rigid decorated ($(S,\kf)$- or $S-$) maps. We shall sometimes use the term 'decorated map' for a rigid representative of a decorated map, and we will use the same notations for maps as for rigid decorated maps.

Write $\Maps_{S,\kf}$ for the collection of decorated $(S,\kf)$-maps and $\Maps_S$ for the collection of decorated $S-$maps. We write $\Maps_{0,k,l,\vd}$ for $\Maps_{S,k},$ where $S$ is the disk moduli specification $([l],\vd)$ and $\kf=[k].$

\end{definition}
Throughout this section it will sometimes be convenient to work with $\kf=[k]$ or $\lf=[l].$ In such cases we write $\oCM_{0,k,l}(\vd)$ for $\oCM_{0,\kf,\lf}(\vd),~\oCM_{0,0,0}$ for $\oCM_{0,\emptyset,\emptyset}(\vd)$ etc. Moduli points will be denoted by $(\Sigma,u),$ where $\Sigma$ is a nodal marked disk and $u$ is the map to $\CP^1.$

\subsection{Chambers and orientation}
It is well known that for any relative class $\vd=(d_{+},d_{-})\neq\vec{0}$ the moduli space $\oCM_{0,0,0}(\vd)$ is orientable (see \cite{FOOO-II} for a much more general theorem; see also \cite{PZ}). A relatively simple consequence is
\begin{thm}\label{thm:moduli_is_orientable}
Suppose $k,l,\vd=(d_{+},d_{-})$ satisfy the stability constraint \eqref{eq:stability},
then the moduli $\oCM_{0,k,l}(\vd)$ is orientable.
\end{thm}
Indeed, $\oCM_{0,k,l}(\vec{0})\simeq\RP^1\times\oCM_{0,k,l},$ where $\oCM_{0,k,l}$ is the moduli of stable marked disk. $\oCM_{0,k,l}$ is orientable. See, for example, \cite[Section 2.5]{PST14}.
One possible proof for the case $\vd\neq\vec{0}$ is as follows. $\CM_{0,k,l}(\vd)$ is orientable, as can be seen by iteratively forgetting the markings and noting that the fibers of the forgetful maps are either punctured disks or punctured circles, both are orientable. $\oCM_{0,k,l}(\vd)$ is obtained from $\CM_{0,k,l}(\vd)$ first by adding strata of codimension at least $2$ which are not the topological boundary of the moduli. Adding them cannot affect orientability. We then add the boundary, which contains codimension $1$ strata, but adding boundary does not affect orientation as well.

Recall that an orientable manifold or orbifold $M$ with boundary or corners induces orientation on its boundary. Let $o_M\in\det(M)^\times/\R_+$ denote the orientation of $M.$ For any $p\in\partial M,$ which does not belong to any corner of codimension at least $2$ we have the short exact sequence
\begin{equation}\label{eq:induced_or_to_bdrt}
0\to \iota^*T_p\partial M\to T_pM\to N\to 0\Rightarrow \det(N)\otimes \det(T\partial M)\simeq \det(M),
\end{equation}
where $N$ is the normal and $\iota:\partial M\to M$ the inclusion. The \emph{induced orientation} on $\partial M$ at $p$ is the element $o_{\partial M}\in \det(T\partial M)/\R_+$ such that under the isomorphism \eqref{eq:induced_or_to_bdrt} we have
\[
o_N\otimes o_{\partial M}=\iota^*o_M,
\]
where $o_N$ is the orientation on $N$ as an outward pointing normal.

\begin{definition}\label{def:assoc_map}
A generic $[u:(\Sigma,\partial\Sigma)\to(\CP^1,\RP^1)]\in\oCM_{0,k,l}(\vd)$ defines an \emph{associated decorated map} $u^{-1}(\RP^1)\subseteq\Sigma.$ This decorated map belongs to $\Maps_{0,k,l,\vd}$ and has a canonical \emph{metric} for boundary segments, loops and arcs, $\ell=\ell_{(\Sigma,u)}$ defined by $\ell(\gamma)=|\int_\gamma u^*d\theta|,$ where $\gamma$ is a boundary segment, loop or an arc.
\end{definition}
\begin{nn}\label{nn:chambers}
For $\mathfrak{G}\in \Maps_{0,k,l,\vd},$ denote by $\CM_{0,k,l}(\vd)^{\mathfrak{G}}$ the subset of the coarse moduli space of $\CM_{0,k,l}(\vd)$ made of elements whose associated decorated map is $\mathfrak{G}.$ The \emph{chamber} of $\mathfrak{G}$ is the closure $\oCM_{0,k,l}(\vd)^{\mathfrak{G}}.$
\end{nn}

The next observation follows from standard complex analysis and dimensional arguments.
\begin{obs}\label{obs:chambers}
For $\mathfrak{G}\in \Maps_{0,k,l,\vd},~\CM_{0,k,l}(\vd)^{\mathfrak{G}}$ is an open subspace of full dimension of the coarse $\oCM_{0,k,l}(\vd).$ Moreover, the space
\[
\oCM_{0,k,l}(\vd)\backslash\bigcup_{\mathfrak{G}\in \Maps_{0,k,l,\vd}}\CM_{0,k,l}(\vd)^{\mathfrak{G}}
\]
is covered by finitely many real suborbifolds of codimension $1.$
\end{obs}
\begin{ex}
The moduli of degree $(d,0)$ maps consists of a single chamber -- the one of the decorated map without loops or arcs, except the boundary of the underlying disk, whose perimeter is $d.$ See Figure \ref{fig:moduli}, left picture, for the case $d=2$.

The moduli of maps of degree $(1,1)$ consists of a single chamber -- of the decorated map with a single arc and no loops, such that one region is positive, the other is negative and the perimeters of their boundaries are $1.$
This moduli is topologically a cylinder $\RP^1\times [0,2\pi],$ see the right picture in Figure~\ref{fig:moduli}. A map is completely determined by the images its two ramification points on $\RP^1,~u_\pm,$ chosen in a way that in the domain curve the oriented boundary arc between the preimage of $u_-$ to that of $u_+$ is in the boundary of the positive (the one which maps to the upper hemisphere).

One boundary of the moduli (the exceptional boundary) is obtained when $u_-\to u_+$ from the left, which results in a contracted boundary component, or equivalently, a degree one sphere with a single marked point constrained to lay in $\RP^1.$
The other boundary is obtained when $u_+\to u_-$ from the left. The result is a map from a pair of disks, one maps to upper hemisphere (degree $(1,0)$) and the other to lower hemisphere (degree $(0,1)$). The moduli is parameterized by the image of the node, which is an arbitrary point of $\RP^1.$ Thus, the two boundaries are indeed topologically~$S^1.$

The moduli of maps of degree $(2,1)$ consists of two chambers. One has the boundary loop and an internal loop, both of perimeter $1.$ These loops are both boundaries of a region which is a positive cylinder. The other region is a negative disk.
The second chamber corresponds to a decorated map with a single arc but no loops. It divides the disk into a positive topological disk region with two corners whose perimeter is $2,$ and a negative one of perimeter $1.$
\begin{figure}[t]
\centering
\includegraphics[scale=.4]{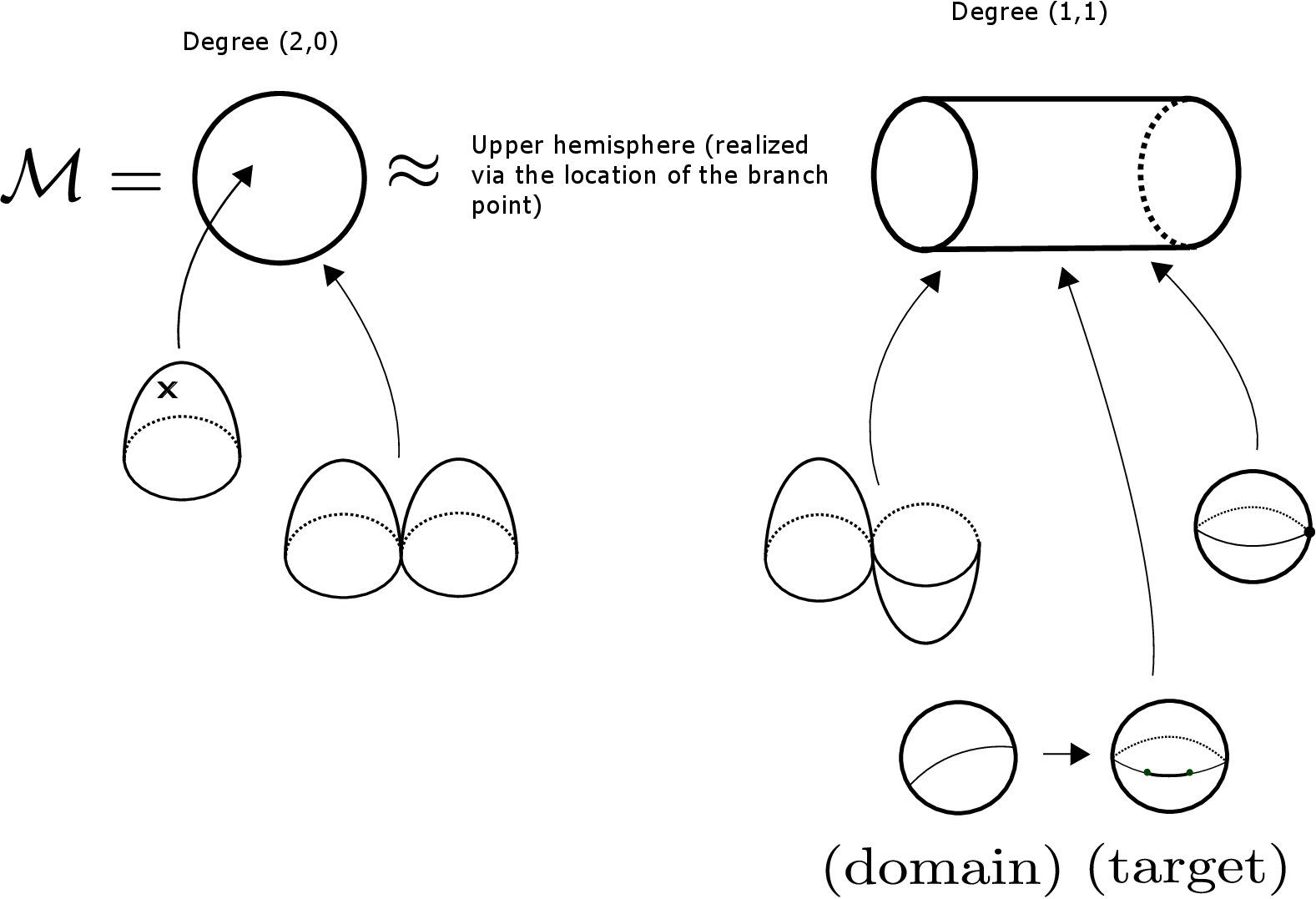}
\caption{An illustration of the moduli spaces for genus $0,$ degrees $(2,0),~(1,1)$ and no internal markings. $\CM_{0,0,0}(2,0)$ is made of a single chamber, without arcs and loops. The moduli is equivalent to a disk via the location of the branch point. $\CM_{0,0,0}(1,1)$ is equivalent to a cylinder. There is a single chamber, which corresponds to a decorated map with a single arc and no loops.}
\label{fig:moduli}
\end{figure}
\end{ex}

We now turn to define the orientation of $\oCM_{0,k,l}(\vd).$ Although Theorem \ref{thm:moduli_is_orientable} guarantees the existence of such an orientation to each moduli separately, in order to define invariants we must choose a specific orientation for each moduli, and our calculations required to understand the relations between different orientations. Our approach resembles the approach of \cite{Comb}, where the orientation of the moduli space of graded surfaces with boundary is constructed using Strebel's stratification and its boundary behaviour is explored. We assume $\vd\neq\vec{0}.$ We define the orientation by explicitly writing an expression for the orientations of the chambers of decorated maps. There are alternative ways to define the orientation, but this definition is best suited for our needs.

Observe that a generic $u:(\Sigma,\partial\Sigma)\to(\CP^1,\RP^1)$ of degree $\vd$ has only simple ramification points with distinct images under $u,$ called the \emph{branch points} and moreover it has $l_u$ internal ramification points and $k_u$ boundary ramification points where
\begin{equation}\label{eq:num_ramification_pts}2l_u+k_u = 2(g+d_{+}+d_{-})-2.\end{equation}
Indeed, the doubled map $u_\C:\Sigma_\C\to \CP^1$ generically has $2l_u+k_u$ ramification points on the one hand, and by Riemann-Hurwitz this number equals $2(g+d_{+}+d_{-})-2.$ See \cite{PZ}, Section 5, for more details.

In fact, the same argument shows the following slightly stronger claim.
\begin{obs}\label{obs:num_ramification_pts_in_region}
The number of (complex) branch points inside a region $\mathcal{R}$ of genus $g$ (the genus is the doubled genus of $\mathcal{R}$ thought of as a topological surface with boundary) which has $b$ boundary components, the $i^{th}$ one covers $\RP^1$ $d_i$ times is
\[
g-1+\sum_{i=1}^b |d_i|.
\]
\end{obs}
Other ramification types or nodes appear in positive codimension. Since any automorphism of a stable map must take branch points to branch points it is straightforward to verify that the locus of orbifold points in the moduli is of real codimension at least $2.$

It is a classical fact that the branch points form local holomorphic coordinates to an open dense subset of the moduli of maps of degree $d$ from $\CP^1$ to $\CP^1$ (\cite{FantechiRahul}, page 17, proof of Proposition $2,$ shows an all genus claim).  The same holds for maps of degree $\vd$ from a disk to $(\CP^1,\RP^1).$
Let $U$ be an open subset of $\CM_{0,k,l}(\vd),$ all of whose points correspond to maps with only simple ramifications and no automorphisms. Assume that $U$ is of the form $V\cap\For_{marking}^{-1}(U'),$ where $V \subseteq \CM_{0,k,l}(\vd),~U'\subset\CM_{0,0,0}(\vd)$ are open, $U'$ is contractible and $\For_{marking}$ is the map which forgets all markings.
Local coordinates for $U$  can be taken to be
\[
\left((W_i)_{i=1}^{l_u}, (Y_i)_{i=1}^{k_u},(Z_i)_{i=1}^l,(X_i)_{i=1}^k\right):U\to(\CP^1\backslash\RP^1)^{l_u}\times(\RP^1)^{k_u}\times\mathbb{H}^l\times\partial\mathbb{H}^k,
\]
where $W_i\in\CP^1$ is the $i^{th}$ complex branch point, $Y_i\in\RP^1$ is the $i^{th}$ real branch point, and $Z_i,X_j$ are defined as follows. We assume that $U'$ is small enough so that there are smooth functions
\[
P:U'\to\RP^1,\quad Q:U'\to\CP^1\backslash\RP^1
\]
which satisfy that $P(\Sigma,u),Q(\Sigma,u)$ are never branch points of
\[
u:(\Sigma,\partial\Sigma)\to(\CP^1,\RP^1).
\]
Such functions can be found for all $U'$ small enough, and any $U'$ as above can be covered by smaller open sets for which such functions $P,Q$ can be found. Define the smooth maps $\widetilde{P},\widetilde{Q}$ from $U'$ to the universal curve restricted to $U'$ by the requirement \[
\widetilde{P}(\Sigma,u)\in\partial\Sigma,\quad u(\widetilde{P}(\Sigma,u))=P(\Sigma,u),\quad\widetilde{Q}(\Sigma,u)\in\Sigma\backslash\partial\Sigma,\quad u(\widetilde{Q}(\Sigma,u))=Q(\Sigma,u).
\]
There is a unique holomorphic biholomorphism $\phi_{(\Sigma,u)}$ which takes $(\Sigma,\partial\Sigma)$ to $(\overline{\mathbb{H}},\partial\mathbb{H}),$ where $\mathbb{H}$ is the upper hemisphere, and which takes $\widetilde{P}$ to $0$ and $\widetilde{Q}$ to $+\sqrt{-1}\in\mathbb{H}.$ The images of the internal and boundary markings under $\phi_{[(\Sigma,u)]}$ are the functions $X_i,Z_j$ respectively. We call them the \emph{locations} of the markings in the domain disk (identified with $\overline{\mathbb{H}}$).

We call these coordinates the \emph{branching and marking coordinates (B$\&$M coordinates)}.
\begin{rmk}
Even when $\kf,\lf=\emptyset $ the branch points give only a local description of the moduli, and are not sufficient for reconstructing a map to $(\CP^1,\RP^1)$. For instance, even in the generic case an additional discrete data is needed for reconstruction: how do the different sheets which are preimages of neighborhoods of the branch point glue to give the whole surface. %In addition, given a map, its branch points are not numbered, and a local numbering will not extend globally.
\end{rmk}

For a decorated map $\mathfrak{G}$ let $s_\mathfrak{G}$ be $-1$ when $k_u=0$ and $d_{-}>d_{+}.$ Otherwise put $s_\mathfrak{G}=1.$
We say that $Y_i$ is \emph{negative} if the region on the left of the \emph{ramification} point $y_i\in u^{-1}(Y_i)$ is negative, otherwise it is \emph{positive}. We similarly call the corresponding ramification points $y_i\in u^{-1}(Y_i)$ positive or negative. The points $y_i\in u^{-1}(Y_i)$ are ordered along the boundary of the disk in a way that a negative point follows a positive one, and a positive point follows a negative one.

The B\&M coordinates and an order $\pi$ on the boundary marked points determine orientations $\mathfrak{o}_\mathfrak{G}^{\pi}$ on $\oCM_{0,k,l}(\vd)^{\mathfrak{G}}$ given by\begin{equation}\label{eq:or}\R_+s_\mathfrak{G}\bigwedge \frac{\partial}{\partial X_i}\otimes\bigwedge \frac{\partial}{\partial Y_i} \otimes\bigwedge \frac{\sqrt{-1}}{2}\frac{\partial}{\partial Z_i}\wedge \frac{\partial}{\partial \bar{Z}_i}\otimes\bigwedge \frac{\sqrt{-1}}{2}\frac{\partial}{\partial W_i}\wedge \frac{\partial}{\partial \bar{W}_i},\end{equation}
where $\frac{\partial}{\partial X_i}$ are ordered in reverse $\pi$ order (from the last to the first, with respect to $\pi),$
and $\frac{\partial}{\partial Y_i}$ are ordered according to
the cyclic order of the ramification points $y_i$ %u^{-1}(Y_i)
 on the boundary of the domain disk, \emph{starting from a positive point}. The orientations of all $1-$forms are the canonical ones, defined by the oriented domain and target surfaces, and the positive orientation on $\RP^1,$ the one induced by the upper hemisphere. When $k\leq 1$ we omit $\pi$ from the notation.
 \begin{rmk}\label{rmk:Y_real_or_ang}In expression \eqref{eq:or} we consider $W_i$ to be complex and $Y_i$ to be real, via the embedding $(\mathbb{C},\mathbb{R})\hookrightarrow(\CP^1,\RP^1),$ and use the upper half-plane model for the domain disk.
Later, for some purposes it will be more convenient to consider $\RP^1$ as the boundary of the unit disk in $\mathbb{C}\subset\CP^1,$ with its angular parametrization. The transition between the two coordinate system changes \eqref{eq:or} by an irrelevant positive function of the $Y_i$ variables. Thus, even in this case we will consider the explicit expression \eqref{eq:or} as the orientation expression.
 \end{rmk}

Some straightforward corollaries of the definitions are
\begin{obs}\label{indep_of_P_n_Q}
The above orientation is independent of the choices of $P,Q.$
\end{obs}
Indeed, changing the functions P,Q amounts to changing the coordinate orientation expression by a positive function.

Under $z\to1/z,$ positive and negative regions are interchanged and if there are no real branchings also $s_\mathfrak{G}$ is flipped, therefore
\begin{obs}\label{obs:behavior of orientation under inversion}
Fix an order $\pi$ on the boundary markings. The map $z\to1/z$ on the target $\CP^1$ induces a map
\[\oCM_{0,k,l}(\vd)\to\oCM_{0,k,l}(d_{-},d_{+})\] by composition.
This map is orientation reversing, with respect to the orientations of \eqref{eq:or}.
\end{obs}

In addition, since also in the closed case the branch points form local holomorphic coordinates on an open dense subset, we have
\begin{obs}\label{obs:or_agree_in_closed_case}
In the closed case, the orientation \eqref{eq:or} agrees with the standard complex orientation of the moduli of maps.
\end{obs}

Let $\mathfrak{G}'$ be the decorated map obtained from $\mathfrak{G}$ by erasing the last internal or boundary marking. Let $\pi'$ be the order on the remaining boundary points induced from $\pi.$
Consider the fiber of the forgetful map
\[\For:\CM_{0,k,l}(\vd)^{\mathfrak{G}}\to\CM_{0,k-\varepsilon_k,l-\varepsilon_l}(\vd)^{\mathfrak{G}'},\]
where $\varepsilon_k=1,\varepsilon_l=0$ if we forget the last boundary marking, and otherwise $\varepsilon_k=1,\varepsilon_l=0.$
The fiber is identified with an interval embedded in the boundary or the whole boundary if $\varepsilon _k=1,$ and otherwise it is a surface embedded in the domain disk.
In both cases, if $U'\subseteq \CM_{0,k-\varepsilon_k,l-\varepsilon_l}(\vd)^{\mathfrak{G}'}$ is an open set with B\&M coordinates, then on $U=\For^{-1}(U')$ one may define B\&M coordinates by $W_i^U=W_i^{U'}\circ \For,~Y_i^U=Y_i^{U'}\circ \For$ and using $\widetilde{P}^U=\widetilde{P}^{U'}\circ \For,~\widetilde{Q}^U=\widetilde{Q}^{U'}\circ \For.$ Thus, for all $Z_i$ except $Z_l,$ when $\varepsilon_l=1,$ it holds that $Z_i^U=Z_i^{U'}\circ \For.$ Similarly for $X_i.$ The additional coordinate $Z_l$ or $X_k$ may be taken as a coordinate of the fiber of $\For.$ With these identifications of coordinates, and the orientation \eqref{eq:or}, the following holds.
\begin{obs}\label{obs:behavior of orientation under forgetful}
The orientation of the fiber, $o_{Fiber},$ induced by the equation \[o_{Fiber}\otimes \For^*\mathfrak{o}_\mathfrak{G'}^{\pi'}=\mathfrak{o}_\mathfrak{G}^{\pi},\] agrees with the one induced from the embedding of the fiber in the domain disk or its boundary.
\end{obs}

The final result of this section is
\begin{prop}\label{prop:orientations glue}
$\mathfrak{o}_\mathfrak{G}^{\pi}$ glue to give a consistent orientation $\mathfrak{o}^{\pi}$ on $\oCM_{0,k,l}(\vd).$
 \end{prop}
\begin{proof}
Consider $\mathfrak{G}\in \Maps_{0,k,l,\vd},$ and let $U$ be as above. Then $U\cap \oCM_{0,k,l}(\vd)^{\mathfrak{G}}$ is an open dense subset of $\oCM_{0,k,l}(\vd)^{\mathfrak{G}}.$
The ramification picture allows us to understand codimension $1$ boundaries of the moduli and more generally of the specific chambers.
\emph{Walls} or codimension $1$ boundaries of the chambers which are not boundaries of the moduli, occur when there is a change in the topology of the decorated maps which does not involve a change in the topology of the underlying surface. In codimension $1$ it happens in one of the following cases. The simplest case is when a boundary marking moves from one region to another. It is straightforward, from the explicit dependence of the orientations on the boundary markings, to see that the orientations of the corresponding chambers glue.

The more interesting cases are when either two regions meet, or when one region splits into two, or when a region hits the boundary. These cases can be equivalently described in terms of ramification points, as can be seen directly from the decorated map picture, or by counting arguments using Observation \ref{obs:num_ramification_pts_in_region}: from this point of view a wall is described by a ramification point which hits a boundary of a region, or by contracting the boundary interval between two boundary ramification points. In the first case the point will actually hit two boundaries (which are different at least locally) and will map to $\RP^1.$ In the second the two ramification points will unite to create a single ramification point of order $3.$ When the internal ramification point hits the boundary of the surface, if the result is a true wall between chambers (and not a boundary component of the moduli), then after crossing that wall the internal ramification point splits into two boundary ramification points.
The local pictures appear in Figure \ref{fig:local_pic}.
\begin{figure}[t]
\centering
\includegraphics[scale=.4]{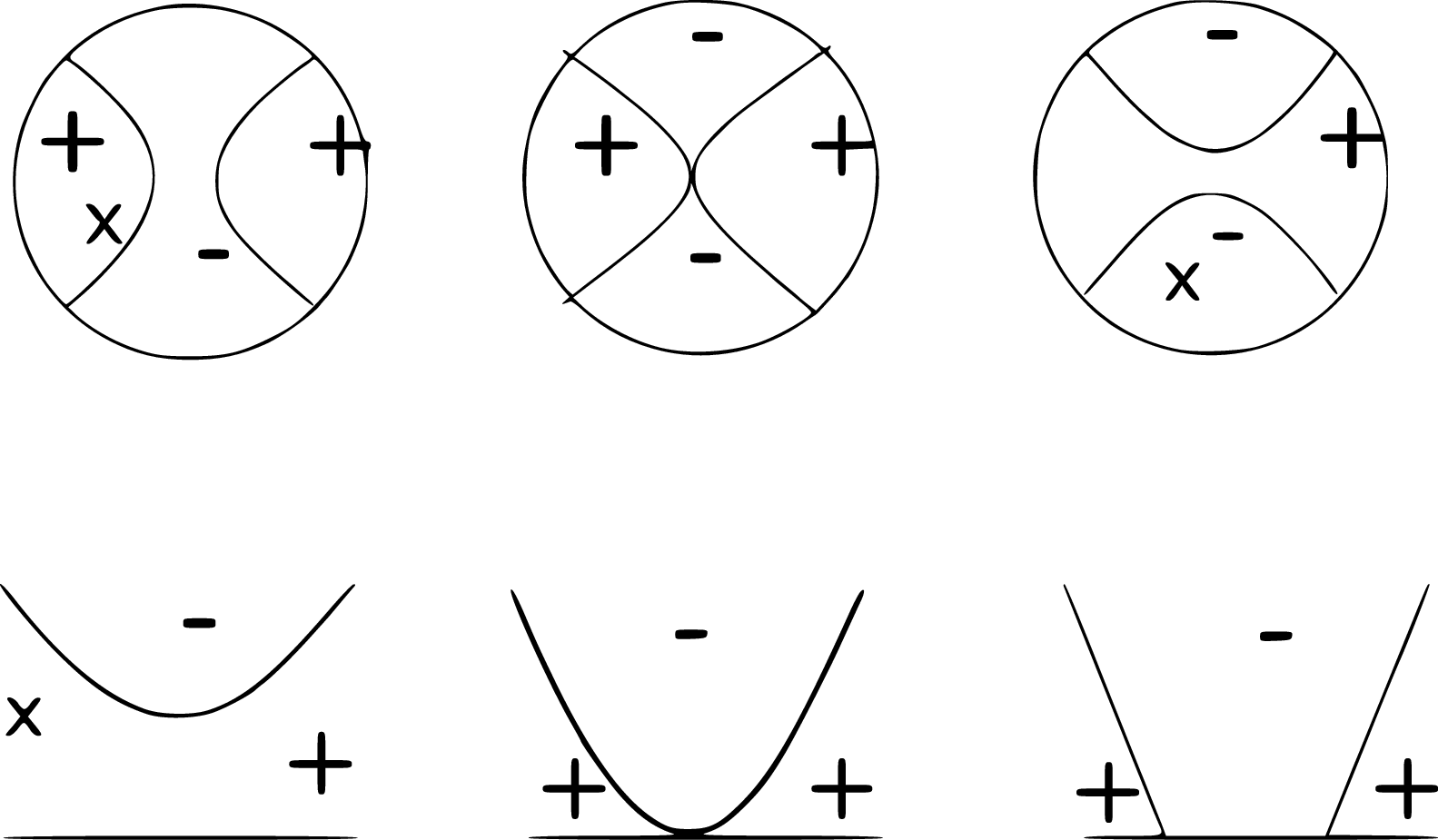}
\caption{Local picture of the change in the topology of decorated maps. In the upper row we zoom in the neighborhood of an internal ramification point (denoted by $\times$). In the left picture, this point approaches a boundary of a positive region. In the middle it hits that boundary, and simultaneously other boundary hits it as well, in the same point. On the right the ramification point moves to a neighboring negative region, and the topology of the decorated map change. In the lower row an internal ramification point $\times$ approaches the boundary of the underlying surface. The middle picture describes the moment it hits, the right picture shows how it "splits" into two boundary ramification points.}
\label{fig:local_pic}
\end{figure}

In order to show that $\mathfrak{o}^{\pi}$ exists it is enough to show that when two chambers share a common wall the orientation varies continuously between them. By the definition of orientation and Observation \ref{obs:behavior of orientation under forgetful}, the claim holds for $\oCM_{0,k,l}(\vd)$ if and only if it holds for $\oCM_{0,0,0}(\vd).$ Let $\mathfrak{G}_1,\mathfrak{G}_2$ be two decorated maps which represent neighboring chambers. In case they have the same number of arcs, then the passage between $\oCM_{0,0,0}(\vd)^{\mathfrak{G}_1},\oCM_{0,0,0}(\vd)^{\mathfrak{G}_2}$ is the result of a ramification point which moves from one region to another one and crosses a loop or an arc. See again Figure \ref{fig:local_pic}, upper row. Locally near a generic point of the wall between the chambers we can use the same set of B\&M coordinates (in fact, only branching coordinates as $k=l=0$) to see that the orientation expression does not change.

Suppose now that $\mathfrak{G}_1,\mathfrak{G}_2$ differ in their number of arcs.
The difference must be one, so suppose $\mathfrak{G}_1$ has $m$ arcs, and $\mathfrak{G}_2$ has $m+1.$
Geometrically we are in the scenario that one complex ramification point, which belongs to a region of $\mathfrak{G}_1$ touching the boundary, hits the boundary. The local picture is as in Figure \ref{fig:local_pic}, lower row.

Suppose this is the point with coordinate $W=W_j.$ Assume first that $W$ is in the upper hemisphere, which we now consider in the upper half-plane model. Let $W=x+iy,~x,y\in\R.$
We have \[\mathfrak{o}_{\mathfrak{G_1}} = \frac{\sqrt{-1}}{2}\frac{\partial}{\partial W_1}\wedge \frac{\partial}{\partial\bar{W}_1}\wedge \tilde{\mathfrak{o}}=-\frac{\partial}{\partial y}\wedge \frac{\partial}{\partial x}\wedge \tilde{\mathfrak{o}},\]
where $\tilde{\mathfrak{o}}$ is the remaining expression for the orientation.
The outward normal to the corresponding boundary of $\oCM_{0,k,l}(\vd)^{\mathfrak{G}_1}$ is $-\frac{\partial}{\partial y}.$ The induced orientation on the wall is therefore $\frac{\partial}{\partial x}\wedge \tilde{\mathfrak{o}}.$
After hitting the boundary, the ramification point "splits" into two real \emph{ramification} points, which we denote by $a=Y_{2m+1},b=Y_{2m+2},$ where $a$ is to the left of $b,$ and $a$ is positive (the region to its left is positive), while $b$ is negative. After applying the map, the corresponding \emph{branch points}, $A,B$ are in opposite order -- $B$ is to the left of $A.$
See Figure \ref{fig:localpicII}.
\begin{figure}[t]
\centering
\includegraphics[scale=.5]{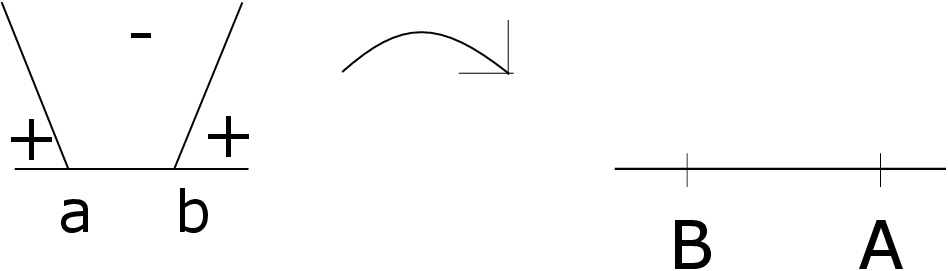}
\caption{Local picture near consecutive boundary ramification points, in the domain (left) and the target (right). Under the map $a\to A,b\to B.$ The horizontal segments in the left picture, which stand for boundary segments in the domain, are mapped to the horizontal segment from $B$ to $A$ (in $\RP^1$) on the right. The two arcs on the left map to the segments on the left of $B$ and on the right of $A.$}
\label{fig:localpicII}
\end{figure}
We can thus write
\[
\mathfrak{o}_{\mathfrak{G_2}} =\frac{\partial}{\partial A}\wedge \frac{\partial}{\partial B}\wedge \tilde{\mathfrak{o}}=-\left(\frac{\partial}{\partial B}-\frac{\partial}{\partial A}\right)\wedge \frac{\partial}{\partial A}\wedge \tilde{\mathfrak{o}},
\]
where $\tilde{\mathfrak{o}}$ is as before.
The outward normal for the boundary, from the $\oCM_{0,k,l}(\vd)^{\mathfrak{G}_2}$ side is $(\frac{\partial}{\partial B}-\frac{\partial}{\partial A}),$ since $A-B$ becomes smaller as we get closer to the wall $B=A.$ Thus, the orientation on the wall, induced by the orientation of the chamber of $\mathfrak{G}_2$ is $-\frac{\partial}{\partial A}\wedge \tilde{\mathfrak{o}}.$ Since along the wall $\frac{\partial}{\partial A}$ can be identified with $\frac{\partial}{\partial x},$ we see that the two chambers induce opposite orientations on the wall, and thus agree. The case that $W$ is in the lower hemisphere follows similarly, or can be deduced from this case by applying inversion $z\to 1/z$ on the target, which interchanges positive and negative regions.
\end{proof}

\begin{rmk}
The loci in the moduli where a ramification point hits the boundary, are either a wall between two chambers, as considered in Proposition \ref{prop:orientations glue}, or a true boundary of the moduli.
The way to distinguish between these two cases is as follows. By following the number of complex ramification points in each region, using Observation \ref{obs:num_ramification_pts_in_region}, we see that whenever a ramification point hits the boundary either a boundary point of another \emph{region} collides with the boundary of the disk (and in the same point), or two boundary points of the disk collide. The former case gives rise to a change in chambers, the latter give rise to the codimension $1$ boundary discussed below. See also Figure \ref{fig:local_pic_nodal} for the second case.
\end{rmk}
\begin{figure}[t]
\centering
\includegraphics[scale=.45]{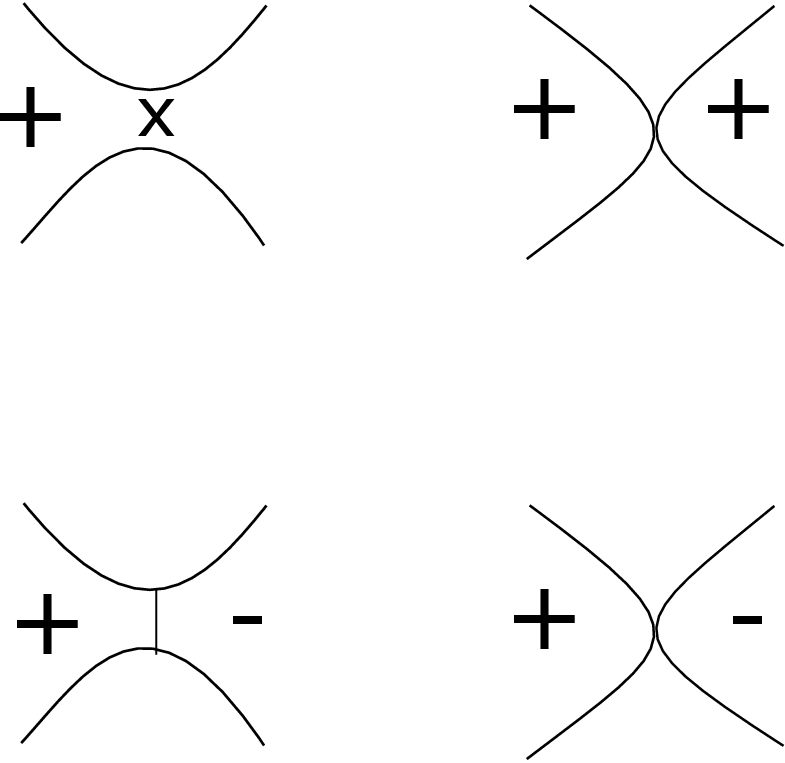}
\caption{Local picture of the creation of boundary nodes.
In the upper row a ramification point, denoted $\times$, approaches two (locally) different boundary components of the underlying domain. This results in a boundary node between two regions of the same sign. The lower row describes an arc which contracts to a point and creates a boundary nodes between regions of different signs.}
\label{fig:local_pic_nodal}
\end{figure}

\subsection{Properties of the orientation $\mathfrak{o}^\pi$}
The goal of this section is to understand the behaviour of the orientation $\mathfrak{o}^\pi$ when restricted to nodal strata.
The section is divided into three parts. The first considers the induced orientation on real codimension one standard boundary strata, which parameterizes maps from nodal disks with a boundary node. The second considers strata of nodal disks with an internal node. The last handles the exceptional boundary stratum of the moduli, where the boundary of the disk contracts to a point which is mapped to $\RP^1.$

\subsubsection{Induced orientation to standard boundary strata of codimension $1$}
Suppose now $k=0$ and consider a real codimension $1$ boundary component
\begin{equation*}
\bb=\CM_{0,\sstar_{1},\lf_{1}}\left(\vd_{1}\right)\times_{L}\CM_{0,\sstar_{2},\lf_{2}}\left(\vd_{2}\right)\xrightarrow{i_{\bb}}\partial\oCM_{0,\emptyset,\lf}\left(\vd\right)
\end{equation*}
The component $\bb$ is endowed with an orientation $\mathfrak{o}_\bb$ coming from the orientation on $\CM_{0,\emptyset,\lf}\left(\vd\right)$ and also with a map
\[
f_\bb=(f_1\times f_2)\circ g:\bb\to\CM_1^{\DIAMOND}\times\CM_2^{\DIAMOND},
\]
where $f_i$ is the map of \eqref{eq:f_i}, defined by forgetting the node, $\CM_i^{\DIAMOND}= \CM_{0,\emptyset,\lf_{i}}\left(\vd_{i}\right),$ and $g:\bb\hookrightarrow\CM_{0,\sstar_{1},\lf_{1}}\left(\vd_{1}\right)\times\CM_{0,\sstar_{2},\lf_{2}}\left(\vd_{2}\right)$ is the structure map. Moreover, $\bb$ is endowed with an evaluation map at the node, $\ev_\sstar:\bb\to\RP^1.$ Its generic fiber $F$ is of real codimension $1$ and the map $\ev_\sstar|_F$ is a submersion away from a finite number of points. $F$ is endowed with two orientations. One is the ratio $\mathfrak{o}_F=\frac{\mathfrak{o}_\bb}{f_\bb^*(\mathfrak{o}_{0,\emptyset,\lf_1;\vd_1}\boxtimes\mathfrak{o}_{0,\emptyset,\lf_2;\vd_2})},$ defined using the canonical isomorphism
\[
\det T_pF\otimes f_\bb^*\det(T_{f_\bb(p)}(\CM_1\times\CM_2))\simeq \det T_p\bb,
\]
which follows from the exact sequence
\[
0\to T_pF\to T_p\bb\to f_\bb^*T_{f_\bb(p)}(\CM_1\times\CM_2)\to0.
\]
The other orientation is defined away from points where $\ev_\sstar|_F$ is not submersion by \[d \ev_\sstar\mathfrak{o}'_F=\frac{\partial}{\partial \theta}.\]
One may compare the two orientations.

Assume now that $\vd_1,\vd_2\neq0$ and let $(\Sigma,u)\in\bb.$
Generically the node is not a ramification point for either of the disks. Therefore it touches a single region in each disk which has a definite sign, defined as usual according to whether this region maps to the positive or negative hemisphere.
\begin{lemma}\label{lem:induced for bdry node}
With the above notations, consider a $(\Sigma,u)\in\bb,$ where the node of $\Sigma$ is not a ramification point. The orientations $\mathfrak{o}_F,\mathfrak{o}'_F$ induced on the fiber of the map which forgets the node agree, locally near $(\Sigma,u),$ precisely if the two regions which touch the half nodes $\sstar_i$ have the same sign.
\end{lemma}
\begin{proof}
Let $U$ be a neighborhood of $ (\Sigma,u)$ small enough such that $U\subseteq\oCM_{0,\emptyset,\lf}\left(\vd\right)^{\mathfrak{G}}$ for a fixed decorated map $\mathfrak{G}.$ Denote by $\mathfrak{G}_i$ the associated decorated map of the projection of $U\cap\bb$ to the $\CM_i^{\DIAMOND}$ component. Consider first the case that the signs of the two regions near the node are the same.
The topological picture near $U\cap\bb$ is that some ramification point approaches the boundary, and its corresponding branch point approaches $\RP^1,$ see the upper row of Figure \ref{fig:local_pic_nodal}.

By using Observation \ref{obs:behavior of orientation under inversion} it is enough to handle the case where the two regions touching the node are positive. Let $W_1$ be the coordinate of the branch point which hits the boundary and thus creates the node. One may write
\begin{equation}\label{eq:or_1}\mathfrak{o}_{\mathfrak{G}} = \frac{\sqrt{-1}}{2}\frac{\partial}{\partial W_1}\wedge \frac{\partial}{\partial \bar{W}_1}\wedge \tilde{\mathfrak{o}},\end{equation}
where $\tilde{\mathfrak{o}}$ is the rest of the orientation expression \eqref{eq:or} in the B\&M coordinates. We can write \[\frac{\sqrt{-1}}{2}\frac{\partial}{\partial W_1}\wedge \frac{\partial}{\partial \bar{W}_1}=\frac{\partial}{\partial x_1}\wedge \frac{\partial}{\partial y_1}=-\frac{\partial}{\partial y_1}\wedge \frac{\partial}{\partial x_1},\] where $W_1$ is considered in the upper half-plane, $x_1=Re(W_1),~y_1=Im(W_1).$ As $W_1$ approaches the boundary, $-\frac{\partial}{\partial y_1}$ can be taken as the external normal, while $\frac{\partial x_1}{\partial\theta}>0.$ From equations \eqref{eq:or},\eqref{eq:or_1} it is evident that
\[(f_\bb)_*\left(\tilde{\mathfrak{o}}|_{\bb\cap U}\right) = (\mathfrak{o}_{0,\emptyset,\lf_1;\vd_1}\boxtimes\mathfrak{o}_{0,\emptyset,\lf_2;\vd_2})|_{f_\bb(\bb\cap U)}.\] Thus, in this case $\mathfrak{o}_F=\mathfrak{o}'_F.$

We now compare the two orientations in case the signs of the two regions touching the node are opposite. Assume that there are $2r$ real ramification points. We show that in this case the ratio is $-1.$ By the usual inversion trick we see that proving this claim for $(d_{+},d_{-})$ is equivalent for proving for the pair $(d_{-},d_{+}).$ We may assume $d_{+}\geq d_{-},$ so that$s_\mathfrak{G}=1.$
From the point of view of topology of decorated maps, the case under consideration occurs when some arc between two real ramification points contracts to a point. Recall from Definition \ref{def:assoc_map} that arcs and boundary segments between ramification points have an intrinsic length.
Let $y_i$ be the ramification point in the preimages of the real branch point $Y_i.$ We choose the indexing so that for $i<2r~y_{i+1}$ follows $y_i$ cyclically according to the orientation of the domain, and $y_{2i}$ are negative points for $i\in[r].$  We will use angular parametrization for $\RP^1$\begin{equation}\label{eq:Y_angular}\alpha\to w_0 e^{2\pi i\alpha}\end{equation} for some $w_0$ with $|w_0|=1,$ and consider $Y_i$ to belong to the interval $[0,2\pi).$
Assume it is the arc $\alpha$ from $y_1,$ which is positive, to $y_{2s},$ which is negative, that contracts to yield a node, as in the lower row of Figure \ref{fig:local_pic_nodal}. Consider $(\Sigma,u)\in\bb,$ and a neighborhood $U$ of it such that $U\cap\oCM_{0,k,l}(\vd)\subset\oCM_{0,k,l}(\vd)^{\mathfrak{G}}.$ Suppose $U$ is small enough that we can choose the coordinates $Y_i$ in a way that for $i<2r$
\begin{equation}\label{eq:Y_diff}
Y_{i}-Y_{i+1}=%\{\ell(\overline{y_iy_{i+1}})\},
\begin{cases}\{\ell(\overline{y_iy_{i+1}})\},&\text{if $y_i$ is positive},\\
-\{\ell(\overline{y_iy_{i+1}})\},&\text{if $y_i$ is negative},
\end{cases}
\end{equation}
where $\overline{y_iy_{i+1}}$ is the positively oriented boundary segment from $y_i$ to $y_{i+1},$ and $\{\}$ stands for fractional part (the requirement on $U$ being small enough amount to the existence of $w_0$ as in \eqref{eq:Y_angular} which makes the sequence $(Y_i)_{i\in[2r]}\subset[0,2\pi)$ monotonically increasing. That would guarantee property \eqref{eq:Y_diff}).
In this case, since boundary components of regions have integer perimeters, it is easy to see that the length $\ell(\alpha)$ equals, modulo $1,$ to
\[
-\left(Y_1-Y_{2s}+\sum_{i\neq 1,2s} a_i Y_i\right) \mod 1,
\]
for some integers $a_i,$ see the example in Figure \ref{fig:ex}. $\ell(\alpha)$ is a coordinate for the inward normal, therefore the outward normal can locally be written as
\[
\frac{\partial}{\partial N}=\frac{\partial}{\partial Y_1}-\frac{\partial}{\partial Y_{2s}}+\sum_{i\neq 1,2s} a_i \frac{\partial}{\partial Y_i}.
\]
Now, the orientation $\mathfrak{o}_{0,\emptyset,\lf;\vd}$ can be locally written as
\[
\frac{\partial}{\partial Y_1}\wedge \frac{\partial}{\partial Y_{2s}}\bigwedge_{i\neq 1,2s}\frac{\partial}{\partial Y_i}\wedge \mathfrak{o}'\wedge\mathfrak{o}''=\left(\frac{\partial}{\partial Y_1}-\frac{\partial}{\partial Y_{2s}}+\sum_{i\neq 1,2s} a_i \frac{\partial}{\partial Y_i}\right)\wedge \frac{\partial}{\partial Y_{2s}}\bigwedge_{i\neq 1,2s}\frac{\partial}{\partial Y_i}\wedge \mathfrak{o}'\wedge\mathfrak{o}'',
\]
where the order of taking the wedge is just the numerical order, $\mathfrak{o}',\mathfrak{o}''$ stand for the wedge of terms which come from complex branch points or internal marked points, in the two disk components. We assume that the disk with the first $\lf_1$ points is the one which touches the node in a positive region. The induced orientation on $\bb$ is therefore \[\frac{\partial}{\partial Y_{2s}}\bigwedge_{i\neq 1,2s}\frac{\partial}{\partial Y_i}\wedge \mathfrak{o}'\wedge\mathfrak{o}''.\] As $\frac{\partial}{\partial Y_{2s}}$ and $\frac{\partial}{\partial \theta}$ point to the same direction, it remains to check that %(we omit $f_\bb^*$ from the notation since the B\&M coordinates give rise to a natural lift, as in the previous case)
\begin{equation}\label{eq:or56}
(f_\bb)_*\left(\bigwedge_{i\neq 1,2s}\frac{\partial}{\partial Y_i}\wedge \mathfrak{o}'\wedge\mathfrak{o}''\right)=-\mathfrak{o}_{0,\emptyset,\lf_1;\vd_1}\boxtimes\mathfrak{o}_{0,\emptyset,\lf_2;\vd_2}.
\end{equation}
We can write
\[
\bigwedge_{i\neq 1,2s}\frac{\partial}{\partial Y_i}\wedge \mathfrak{o}'\wedge\mathfrak{o}''=\left(\bigwedge_{i=2}^{2s-1}\frac{\partial}{\partial Y_i}\wedge \mathfrak{o}''\right)\wedge\left(\bigwedge_{i>2s}\frac{\partial}{\partial Y_i}\mathfrak{o}'\right),
\]
the wedge is taken again in the numerical order.
Now, with $\Prr_i:\CM_1^{\DIAMOND}\times\CM_2^{\DIAMOND}\to\CM_i^{\DIAMOND},$ we have

\[(\Prr_1\circ f_\bb)_*\left(\bigwedge_{i>2s}\frac{\partial}{\partial Y_i}\wedge\mathfrak{o}'\right)=\mathfrak{o}_{0,\emptyset,\lf_1;\vd_1},\] since if there are $Y_i$ with $i>2s,$ they are taken in the correct order, as $Y_{2s+1}$ is positive, while if not, then it must be that $d'_1>d'_2,$ and again we find agreement, since the node touches the corresponding disk component in a positive region.
Similarly, \[(\Prr_2\circ f_\bb)_*\left(\bigwedge_{i=2}^{2s-1}\frac{\partial}{\partial Y_i}\wedge \mathfrak{o}''\right)=-\mathfrak{o}_{0,\emptyset,\lf_2;\vd_2},\]
since if $s>1,$ the wedge is taken starting from a negative point $Y_2,$ while if $s=1$ an extra sign $s_{\mathfrak{G}'}$ appears in the definition of $\mathfrak{o}_{0,\emptyset,\lf_2;\vd_2}^{\mathfrak{G}'}$ for any decorated map $\mathfrak{G}'$ without real branches. Equation~\eqref{eq:or56} follows.
\end{proof}
\begin{figure}[t]
\centering
\begin{tikzpicture}[scale=1]

\draw[very thick] (0,0) ellipse (5 and 3);
\draw (-3,2.4) to[out=-80,in=70] (-3.5,-2.162);
\draw (-1,2.95) to[out=-110,in=100] (-1.7,-2.81);
\draw (1.5,2.85) to[out=-135,in=100] (0.3,-3);
\draw (2.5,2.6) to[out=-90,in=175] (4.98,0.2);
\draw (2,-2.75) to[out=95,in=-175] (4.97,-0.35);

\coordinate [label=90: $y_9$] (B) at (-3,2.4);
\coordinate [label=-90: $y_{10}$] (B) at (-3.5,-2.162);
\coordinate [label=90: $y_8$] (B) at (-1,2.95);
\coordinate [label=-90: $y_1$] (B) at (-1.7,-2.81);
\coordinate [label=90: $y_7$] (B) at (1.5,2.85);
\coordinate [label=-90: $y_2$] (B) at (0.3,-3);
\coordinate [label=90: $y_6$] (B) at (2.5,2.6);
\coordinate [label=0: $y_5$] (B) at (4.98,0.2);
\coordinate [label=-90: $y_3$] (B) at (2,-2.75);
\coordinate [label=0: $y_4$] (B) at (4.97,-0.35);

\coordinate [label=center: $\scriptstyle -$] (B) at (-4,-0.7);
\coordinate [label=center: $\scriptstyle 8$] (B) at (-4,0.7);
\coordinate [label=center: $\scriptstyle +$] (B) at (-2.5,-1);
\coordinate [label=center: $\scriptstyle 6$] (B) at (-2.3,1);
\coordinate [label=center: $\scriptstyle -$] (B) at (-0.9,-1);
\coordinate [label=center: $\scriptstyle 9$] (B) at (-0.7,1);
\coordinate [label=center: $\scriptstyle +$] (B) at (1.4,-1);
\coordinate [label=center: $\scriptstyle 4$] (B) at (1.8,1);
\coordinate [label=center: $\scriptstyle -$] (B) at (3.4,1.6);
\coordinate [label=center: $\scriptstyle 4$] (B) at (4,1);
\coordinate [label=center: $\scriptstyle -$] (B) at (2.8,-1.9);
\coordinate [label=center: $\scriptstyle 7$] (B) at (3.7,-1.2);

\end{tikzpicture}
\caption{We consider a map with several arcs which divide the disk into domains, some positive, and some negative. The perimeter of each domain also appears in the figure. Lengths of arcs, modulo $1$, can be calculated using the lengths of boundary segments between the ramification points. $Y_i$ is the branching coordinate of $y_i.$ For example, modulo $1,$ the length of boundary segment from $y_3$ to $y_4$ is $Y_3-Y_4,$ the length of the arc from $y_3$ to $y_4$ is $Y_4-Y_3,$ and the length of the arc from $y_1$ to $y_8$ is
$-Y_1+Y_8-Y_7+Y_2-
(-(Y_3-Y_2+Y_5-Y_4+Y_7-Y_6 - (Y_4-Y_3+Y_6-Y_5))).$ }
\label{fig:ex}
\end{figure}
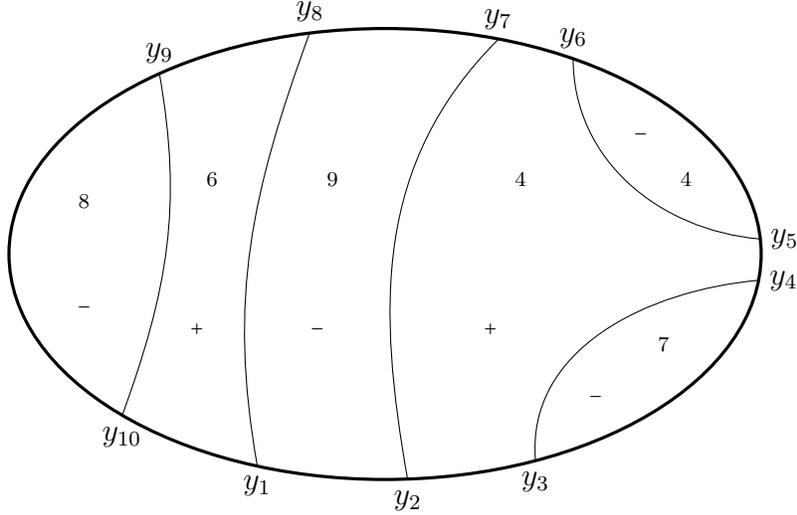

Still working with the same notations, we have the following simple proposition.
Let $\bb$ parameterize stable maps from a domain made of two nodal disks, one is of degree $0$ and contains the $i^{th}$ internal marking, and the other of non-zero degree, and contains the other markings. Thus, we can identify $\CM_{0,\sstar_{1},\{i\}}\left(\vec{0}\right)$ with a point, and then
\[
\bb=\CM_{0,\sstar_{1},\{i\}}\left(\vec{0}\right)\times_L\CM_{0,\sstar_{2},\lf\setminus\{i\}}\left(\vd\right)
\]
canonically inherits a canonical $S^1-$bundle structure
\[
f_2:{\bb}\to\CM_2^{\DIAMOND}=\CM_{0,\emptyset,\lf\setminus\{i\}}\left(\vd\right)
\]
whose fiber at $(\Sigma_2,u_2)$ is $\partial\Sigma_2.$ Let $\mathfrak{o}_F$ be the orientation of the fiber which agrees, via this identification, with the orientation of the boundary of the \emph{domain} disk $\Sigma_2.$
\begin{prop}\label{prop:deg0 bubbling}
The orientations $\mathfrak{o}_\bb,~\mathfrak{o}_F\wedge f_2^*\mathfrak{o}_{\CM_2^{\DIAMOND}}$ agree. Hence, for any fiber $F$ of $f_2:\bb\to\CM_2^{\DIAMOND},$ oriented by $\mathfrak{o}_F,$ we have
\[
\int_F\ev_{\sstar_2}^*d\theta =d_{+}-d_{-}.
\]
\end{prop}%
\begin{proof}
The proof of the first part is similar to the proof of Lemma \ref{lem:induced for bdry node}, the case when a ramification point hits the boundary, we only mention the differences: this time it is a marked point whose coordinate is $Z_i$ which hits the boundary, rather than a branch point. Writing $Z_i=x_i+\sqrt{-1}y_i,$ we see that $-y_i$ corresponds to the external normal, and $x_i$ is the location of the node in the domain, seen from the second component (of non-zero degree). Hence $\frac{\partial}{\partial x_i}$ is oriented as $\mathfrak{o}_F.$

The degree of ${\ev_{\sstar_2}}|_F:F\to\RP^1$ is $\pm(d_{+}-d_{-}).$ In order to determine the sign, observe that when the node is in a positive region, ${\ev_{\sstar_2}}|_F$ takes positively oriented vectors, with respect to $\mathfrak{o}_F,$ to positively oriented vectors, with respect to $\frac{\partial}{\partial \theta}.$ In the negative regions positively oriented vectors are taken negatively oriented vectors. The sign is thus $+$ and the 'Hence' part follows.
\end{proof}

\subsubsection{Induced orientation to strata of one internal node}
Let $i_{\ii}:\ii\hookrightarrow\oCM_{0,\kf,\lf}(\vd)$ be the stratum
\begin{equation*}
\ii=\CM_{0,\lf_{c}\cup\{\sstar_c\}}\left(d_c\right)\times_{\CP^1}\CM_{0,\kf,\lf_{o}\cup\{\sstar_o\}}\left(\vd_{o}\right)=:\CM_c\times_{\CP^1}\CM_o.
\end{equation*}
we assume $\vd_o=(d^{+}_o,d^{-}_o)\neq (0,0).$ Then $\mathcal{N}_\ii,$ the normal bundle of $\CM_\ii$ in $\oCM_{0,\kf,\lf}(\vd)$, is canonically
\begin{equation}\label{eq:normal in internal node}
 \mathcal{N}_\ii\simeq\CL^\vee_c\times\CL^\vee_o,
\end{equation}
where $\CL^\vee_c$ is the tangent at the node $\sstar_c$ in the closed component and $\CL^\vee_o$ is the tangent at $\sstar_o$ in the open component, see, for example, \cite{PZ2}. %\footnote{Ran: add citation. note that maybe it's more accurate to describe via universal curve. Anyway, a point where it's described is https://arxiv.org/pdf/1504.06617.pdf page 9. a proof of a lemma similar to below is in math/1507.06633v2} In particular it carries a canonical complex orientation.
$\ii$ is induced an orientation $\mathfrak{o}^\pi_\Gamma$ from the orientations $\mathfrak{o}^\pi_{o},\mathfrak{o}_{c}$ on $\CM_o,\CM_c$ respectively (which, as mentioned above, agree with the standard complex orientation on the moduli $\oCM_{c}$), and the complex orientation of $\CP^1$ by
\begin{equation}\label{eq:fibered product-internal node}
0\to T_{(\Sigma,u)}\ii\to T_{(\Sigma_c,u_c)}\CM_{c}\boxplus T_{(\Sigma_o,u_o)}\CM_{o}\to T_{u(\sstar)}\CP^1\to 0,
\end{equation}
for any $(\Sigma,u)\in\ii$ with a sphere component $(\Sigma_c,u_c),$ a disk component $(\Sigma_o,u_o),$ and a node $\sstar.$
\begin{lemma}\label{lem:induced for internal nodes}
The orientation $o_{\mathcal{N}_\ii}^{\pi}$ induced on $\mathcal{N}_\ii$ by the sequence
\[0\to T\ii\to i_\ii^*T\oCM_{0,\kf,\lf}(\vd)\to\mathcal{N}_\ii\to 0\]
and the orientations $\mathfrak{o}_\ii^{\pi},\mathfrak{o}_{0,\kf,\lf;\vd}$
agrees with the complex orientation of $\mathcal{N}_\ii.$
\end{lemma}
\begin{proof}
In case we work in moduli space of maps to $\CP^1$ from Riemann surfaces without boundary, the corresponding claim is known to hold.
For the open case, we work in the B\&M coordinates. Assume first $\vd_c\neq 0.$ It will be convenient to assume $\kf=[k],\lf=[l].$

By the definition of the orientations we have that the induced orientation on the fiber of $\For_l,$ the map forgetting the $l^{th}$ internal marked point, agrees with the complex orientation induced from the domain surface (or equivalently from the universal curve). If there are markings on the closed components, using the commutativity of
\begin{equation}\label{eq:normal under forgetful of internal}
\xymatrix{
\ii \ar[r]\ar[d]& \oCM_{0,k,l}(\vd) \ar[d] \\
\ii' \ar[r] & \oCM_{0,k,l-1}(\vd),
}
\end{equation}
where $\ii'=\For_l(\ii),$
and the fact that, since $\For_{l}$ does not contract the closed component, \eqref{eq:normal in internal node} is compatible with $\For_l,$  we see that the claim for $l$ points is equivalent to the one with $l-1$ points.
So we assume there are no markings on the closed component.

By using Observation \ref{obs:behavior of orientation under inversion}, we may assume without loss of generality that the node is in the upper hemisphere. An internal bubbling is then the result of a collision between two branch points in the upper hemisphere. In fact, a more accurate statement holds. Fix a map $(\Sigma,u),$ and let $W_1,W_2\in\CP^1$ be two very close branch points in the upper hemisphere. Let $V\subset \CP^1$ be a small common contractible neighborhood containing no other branch points. Consider $u^{-1}(V)\subset \Sigma,$ and $w_i=u^{-1}(W_i)$ the corresponding ramification points.

Let $\gamma:[0,1]\to V$ be a path with endpoints $\gamma(0)=W_1,~\gamma(1)=W_2.$ Define \[\Gamma:[0,1]\to U\subset\oCM_{0,\kf,\lf}(\vd),\] where $U$ is an open neighborhood of $(\Sigma,u)$ for which B\&M coordinates can be found, to be the path in the moduli with $\Gamma(0)=(\Sigma,u),$ and $\Gamma(t)$ characterized by having all B\&M coordinates the same as those of $(\Sigma,u)$ except for $W_1(\Gamma(t))=\gamma(t).$ Let $(\Sigma',u')=\Gamma(1).$ When $V$ is small enough there is no difficulty in finding such $\gamma,U,\Gamma.$

If $w_1,w_2$ are not in the same connected component of $u^{-1}(V),$ then the moduli point $(\Sigma',u')$ has a ramification of type $(12)(34)$ over $W_1$ (see \cite{OP06} for definitions of the different ramification types, \cite{KLZ15} is another helpful reference). Suppose now that they are in the same connected component of $u^{-1}(V).$ Let $\tilde\gamma$ be a path in $u^{-1}(V),$ connecting $w_1$ to $w_2,$ and we may take $\gamma=u\circ\tilde\gamma$ to be its projection to $V.$ Recall that a for a generic $P\in\CP^1$ in the upper hemisphere $|u^{-1}(P)|=d_{+}.$ If the number of connected components of $u^{-1}(\gamma)$ is $d_{+}-2,$ then $(\Sigma',u')$ has a ramification of type $(123)$ at $W_1.$ The complementary case, when the number of connected components of $u^{-1}(\gamma)$ is $d_{+}-1,$ is the one which gives a nodal map $(\Sigma',u'),$ when $W_1,W_2$ collide.

With this topological observation, the proof of the lemma, in case $d_c\neq0,$ is now simple.
Let $W_1,W_2$ be the coordinates of the branch points that collide and create the node. We identify the upper hemisphere with the upper half-plane. One may write
\[
\mathfrak{o}_{\mathfrak{G}} = \left(\frac{\sqrt{-1}}{2}\frac{\partial}{\partial W_1}\wedge \frac{\partial}{\partial \bar{W}_1}\right)\wedge\left(\frac{\sqrt{-1}}{2}\frac{\partial}{\partial W_2}\wedge \frac{\partial}{\partial \bar{W}_2}\right) \wedge\tilde{\mathfrak{o}},
\]
where $\tilde{\mathfrak{o}}$ is the rest of the orientation expression \eqref{eq:or} in the B\&M coordinates. Locally $\ii$ is the locus $W_1=W_2.$ Thus locally the normal bundle has a complex trivialization given by the section \[\frac{\partial}{\partial W_1}-\frac{\partial}{\partial W_2},\] and the complex orientation induced by this trivialization agrees with the complex orientation of $\mathcal{N}_\ii$ defined by \eqref{eq:normal in internal node}.
In addition,
\[
\left.\left(\frac{\sqrt{-1}}{2}\frac{\partial}{\partial W_2}\wedge \frac{\partial}{\partial \bar{W}_2}\wedge\tilde{\mathfrak{o}}\right)\right|_{\ii\cap U}
\]
is identified with $\mathfrak{o}^\pi_\ii,$ by \eqref{eq:fibered product-internal node} and the definition of the canonical orientations. This settles the case $d_c\neq 0.$

Suppose now $d_c=0.$ In this case there are at least two markings on the closed component. By applying the same trick of forgetting markings on the closed component, we see that it is enough to prove the claim when the closed component has exactly two markings. Assume one of them is the $l^{th}$ marking, and the other marking is the $(l-1)^{th}.$ Locally near a generic point of $\ii$ we can consider B\&M coordinates. Locally a dense open subset of $\ii$ is specified by the equation $Z_l=Z_{l-1},$ and can be described by the same B\&M coordinates, only without $Z_l.$ Denote by $\tilde{\mathfrak{o}}$ the orientation expression \eqref{eq:or} in these new B\&M coordinates. The normal bundle is locally trivialized by (the image of) \[\frac{\partial}{\partial Z_l}-\frac{\partial}{\partial Z_{l-1}}.\] Comparing orientation expressions we see that $\tilde{\mathfrak{o}}$ will agree with the orientation induced from
the fibered product
\[\CM_{0,\{l-1,l,\sstar_c\}}\left(0\right)\times_{\CP^1}\CM_{0,k,[l-2]\cup\{\sstar_o\}}\left(\vd\right)\simeq \CM_{0,k,[l-2]\cup\{\sstar_o\}}\left(\vd\right),\]
precisely if the orientation on the normal bundle defined using \eqref{eq:normal in internal node}, where $\CL^\vee_o,\CL^\vee_c$ are given their natural complex orientations, equals the complex orientation defined on the normal using the local section \[\frac{\partial}{\partial Z_l}-\frac{\partial}{\partial Z_{l-1}}.\]
But, by identifying a small neighborhood of the node $\sstar_o,$ located at $Z_{l-1},$ with a small neighborhood of $0$ in its tangent line, we see that $\frac{\partial}{\partial Z_l}- \frac{\partial}{\partial Z_{l-1}}$ can also be identified as a complex trivialization of $\CL^\vee_o$ at the node. In addition $\CL_c^\vee$ is a trivial complex line. Putting together we see that the two orientations for the normal bundle agree.
\end{proof}

\subsubsection{Induced orientation to the exceptional boundary}
Suppose $\kf=\emptyset$ and $d_{+}=d_{-}=d.$ Consider the exceptional boundary component
 \[\ee
\hookrightarrow\oCM_{0,\emptyset,\lf}(d,d)\]
which parameterizes maps from domains with a contracted boundary, elements of $\ee$ are equivalent to closed maps of degree $(d,d)$ with markings $\lf\sqcup\star,$ the last one $\star,$ is constrained to $\RP^1,$ so we can also write
\[\ee=\ev_{\star}^{-1}\left(L\right)\subset{\oCM}_{0,\lf\sqcup\star}\left(d\right).\] The generic fiber of the forgetful map $f=\For_{\star},$ is canonically a $d-$cover of $\RP^1,$ and is induced an orientation $\frac{\partial}{\partial \phi}=(d\ev_\star)^{-1} \frac{\partial}{\partial\theta}$ defined by pulling back the canonical orientation of $\RP^1$ along this cover.

\begin{lemma}\label{lem:contracted boundary}
The orientation induced on $\ee$ by $\mathfrak{o}_{0,\lf,(d,d)}$ is \[-f^*\mathfrak{o}_c\otimes \frac{\partial}{\partial \phi},\] where $\mathfrak{o}_c$ is the complex orientation of the moduli ${\oCM}_{0,\lf}\left(d\right).$% and $f:\ee\to{\oCM}_{0,\lf}\left(d\right)$ the map which forgets the node.

In particular, integrating $d\theta$ along the fiber of $f,$ with its induced orientation, gives $-d.$
\end{lemma}
\begin{proof}
Chambers whose closure intersects $\ee$ in codimension $0$ are those whose associated decorated maps have a single arc, from the positive ramification point $x$ to the negative $y.$ Choose such a chamber, and let $X,Y$ be the coordinates of the corresponding branch points in angular parametrization as in Lemma \ref{lem:induced for bdry node}. We choose the angular parametrization so that near $\ee$ it holds that \[|X-Y|=\ell(\overline{xy}),\] the length of the boundary segment from $x$ to $y.$ The orientation near $\ee$ can be written as $\frac{\partial}{\partial X}\wedge \frac{\partial}{\partial Y}\wedge \tilde{\mathfrak{o}}=-(\frac{\partial}{\partial Y}-\frac{\partial}{\partial X})\wedge \frac{\partial}{\partial X}\wedge \tilde{\mathfrak{o}},$ where $\tilde{\mathfrak{o}}$ is the wedge of the terms correspond to the remaining complex branch points. Near $\ee$ the length of the arc between $x,y$ is close to some positive integer $\leq d,$ while the length of the boundary segment between them is close to $0.$ Thus, only a small fraction of $\RP^1$ is covered by the boundary, and is in fact covered twice, with opposite orientations. The covered segment, in the image, is oriented from $Y$ to $X,$ meaning that $X$ is its right endpoint, hence near $\ee$ its length is given by $X-Y.$ Thus, $\frac{\partial}{\partial Y}-\frac{\partial}{\partial X}$ is an outward normal. The induced orientation on $\ee$ is therefore $-\frac{\partial}{\partial X}\wedge \tilde{\mathfrak{o}}.$ Since on $\ee$ the vectors $d\ev_\star\frac{\partial}{\partial X}$ and $\frac{\partial}{\partial \theta}$ point to the same direction, the main statement of the lemma follows.

The 'In particular' follows by integration: for a generic $(\Sigma,u)\in\oCM_{0,\lf}(d),$ the fiber $f^{-1}(\Sigma,u)$ is canonically identified with the subset $u^{-1}(\RP^1)\subset\Sigma,$ which is a union of loops. The degree of
\[u|_{u^{-1}(\RP^1)}:u^{-1}(\RP^1)\to\RP^1\]
is $\pm d.$ The sign is determined by the previous paragraph to be $-1.$
\end{proof}

\end{document}